\documentclass[10pt]{amsart}

\usepackage{amsmath,amsthm,amscd,amsfonts,amssymb,epic,eepic,bbm,comment,enumitem,mathdots,mathtools}
\usepackage{mathrsfs} 
\usepackage[pagebackref,colorlinks=true,linkcolor=blue,citecolor=blue]{hyperref}
\usepackage{tikz-cd}

\allowdisplaybreaks
\newcommand{\norm}[1]{\left\lVert#1\right\rVert} 

\newtheorem{theorem}{Theorem}[section]
\newtheorem{corollary}{Corollary}[theorem]
\newtheorem{lemma}[theorem]{Lemma}
\newtheorem{proposition}[theorem]{Proposition}
\newtheorem{conjecture}[theorem]{Conjecture}

\theoremstyle{definition}
\newtheorem{remark}[theorem]{Remark}

\newtheorem{definition}[theorem]{Definition}

\newcounter{remarkscounter}

\numberwithin{equation}{section}

\DeclareMathOperator{\tr}{\mathrm{tr}}

\DeclareMathOperator{\vol}{\mathrm{vol}}
\DeclareMathOperator{\supp}{\mathrm{supp}}
\DeclareMathOperator{\Hom}{\mathrm{Hom}}
\DeclareMathOperator{\Spec}{\mathrm{Spec}}
\DeclareMathOperator{\GL}{\mathrm{GL}}
\DeclareMathOperator{\SL}{\mathrm{SL}}
\DeclareMathOperator{\Gal}{\mathrm{Gal}}

\numberwithin{equation}{subsection}

\newcommand{\Ind}{\mathrm{Ind}}

\newcommand{\ZZ}{\mathbb{Z}}
\newcommand{\RR}{\mathbb{R}}
\newcommand{\CC}{\mathbb{C}}
\newcommand{\QQ}{\mathbb{Q}}
\newcommand{\GG}{\mathbb{G}}

\newcommand{\calO}{\mathcal{O}}
\newcommand{\calB}{\mathcal{B}}

\newcommand{\Temp}{\mathrm{Temp}}
\renewcommand{\Im}{\mathop{\mathrm{Im}}}
\renewcommand{\Re}{\mathop{\mathrm{Re}}}

\newcommand{\HP}{\mathop{\mathrm{HP}}\nolimits}

\newenvironment{psmatrix}
  {\left(\begin{smallmatrix}}
  {\end{smallmatrix}\right)}

\begin{document}

\author{Chun-Hsien Hsu}
\address{Department of Mathematics\\
University of Chicago\\
Chicago, IL 60637}
\email{chunhsien@uchicago.edu}

\author{HaoYun Yao}
\address{Department of Mathematics\\
Duke University\\
Durham, NC 27708}
\email{haoyun.yao@duke.edu}


\title{Schwartz spaces on $L$-monoids: non-Archimedean}

\begin{abstract}
    We complete the Braverman-Kazhdan-Ng\^{o} program over non-Archimedean local fields assuming local Langlands conjecture for tempered representations and a natural assumption on the $\gamma$-factors of non-supercuspidal discrete series. In particular, the program is complete unconditionally for general linear groups over non-Archimedean local fields.
\end{abstract}

\maketitle 
\tableofcontents

\section{Introduction} 

Let $F$ be a non-Archimedean local field. Let $\mathcal{O}_F$ denote its ring of integers and $\varpi=\varpi_F\in \mathcal{O}_F$ be a uniformizer. Let $|\cdot|=|\cdot|_F$ be the norm on $F$ such that $q=|\varpi|^{-1}$ is the cardinality of the residue field of $\mathcal{O}_F$. Let $\psi:F\to\mathbb{C}^\times$ be a nontrivial additive character.

Let $G$ be a connected reductive group over $F$ with a nontrivial rational character $\nu=\nu_G:G\longrightarrow\mathbb{G}_m$. Let $\rho:{}^LG\to \GL_{V_\rho}(\mathbb{C})$ be a representation of the $L$-group such that $\rho\circ\nu^\vee(z) =z\cdot\mathrm{id}_{V_\rho}$ for $z\in\mathbb{C}^\times.$ Braverman and Kazhdan \cite{BK-lifting} observed that Langlands functoriality conjecture would imply that one can associate to $\rho$ a Fourier theory on $G(F)$ generalizing Tate's thesis and the work of Godement and Jacquet \cite{GodementJacquetBook}. More explicitly, they conjectured the existence of a Schwartz space $\mathcal{S}\subseteq C^\infty(G(F))$ and a unitary Fourier transform $\mathcal{F}_{\rho,\psi}:\mathcal{S}\longrightarrow \mathcal{S}$ that realize local $L$-factors as the greatest common divisors of local zeta integrals. A more precise formulation is stated in  \cite[\S4]{Ngo:Hankel} and \cite[\S2]{Luo:Ngo}. 

Assume $\rho$ is tempered, i.e., the image of the Weil group under $\rho$ has bounded image. In a recent preprint \cite{DRS}, the authors observed the possibility of constructing a Fourier theory attached to $\rho$ in the Harish-Chandra Schwartz space $\mathcal{C}(G(F))$. Assuming tempered local Langlands correspondence for all Levi subgroups of $G$ (we refer one to \S\ref{ssec:assumptions} for more precise statements), they constructed a Fourier transform on $\mathcal{C}(G(F)).$ We summarize their results as follows.
\begin{theorem}[{\cite{DRS}}\footnote{ In \cite{DRS} local fields are assumed to have characteristic zero. However, all constructions and proofs work without any change for local fields of positive characteristic. }]
    Suppose $\rho$ is tempered and \eqref{LLC:LCFT}--\eqref{LLC:para} hold for all Levi subgroups of $G$. There exists a linear operator 
    \begin{align*}
        \mathcal{F}_{\rho,\psi}:\mathcal{C}(G(F))\to \mathcal{C}(G(F))
    \end{align*}
    that extends to an isometric automorphism of $L^2(G(F),dg)$ with the following properties:
    \begin{enumerate}
        \item $\mathcal{F}_{\rho,\bar{\psi}}\circ \mathcal{F}_{\rho,\psi}=\mathrm{id}$ and $\mathcal{F}_{\rho,\psi}^4=\mathrm{id}.$
        \item $\mathcal{F}_{\rho,\psi}$ is twisted-equivariant, i.e., $\mathcal{F}_{\rho,\psi}\circ R(g,g') = R(g',g)\circ \mathcal{F}_{\rho,\psi}$ for all $g,g'\in G(F),$ where $R:G(F)\times G(F)\longrightarrow\mathrm{Aut}(L^2(G(F)))$ is the regular action.
        \item Let $f\in \mathcal{C}(G(F))$ and $\pi$ be an irreducible tempered representation of $G(F)$. Then
        \begin{align*}
            \pi(\mathcal{F}_{\rho,\psi}(f)^\lor)=\gamma(\tfrac{1}{2},\pi,\rho,\psi)\pi(f),
        \end{align*}
        where $\mathcal{F}_{\rho,\psi}(f)^\lor(g):=\mathcal{F}_{\rho,\psi}(f)(g^{-1})$. Here the $\gamma$-factors are defined using the local Langlands correspondence \eqref{eq:Lgamma:LL}.
    \end{enumerate}
\end{theorem}

Under the same assumption, the authors proceeded to cut out a $\mathcal{F}_{\rho,\psi}$-stable subspace $\mathcal{S}^{\mathrm{as}}_\rho(G(F))\subseteq \mathcal{C}(G(F)),$ called the \textbf{asymptotic $\rho$-Schwartz space} (c.f.  \cite[\S 5]{DRS} or \S \ref{ssec:ass} below). In what follows, we briefly review the definition of $\mathcal{S}^{\mathrm{as}}_\rho(G(F))$ and explain why it is not the correct Schwartz space attached to $\rho$ when $G/Z_G$ is isotropic, where $Z_G$ is the center of $G$. For detailed discussions, we refer the reader to \S \ref{sec:Schwartz:spectral}.

Fix a maximal split torus $A_0$ and a minimal parabolic subgroup $P_0\ge A_0$ of $G$. Let $P\ge P_0$ be a standard parabolic subgroup of $G$ and $M\ge A_0$ be the standard Levi subgroup of $P$. Let $\sigma$ be an irreducible unitary discrete series of $M(F)$. Let $\chi:M(F)\to \CC^\times$ be an unramified quasi-character, and let $\sigma_\chi:=\sigma\otimes \chi.$ Let $K$ be the stabilizer of a special point in the apartment of $G$ associated to $A_0$. We identify $\mathrm{Ind}_P^G\sigma_\chi$ with $\Ind_{K_P}^K\sigma|_{K_M}:=\Ind_{K\cap P(F)}^K\sigma|_{K\cap M(F)}$ under restriction.

For $f\in \mathcal{C}(G(F))$ and $(v,w)\in \Ind_{K_P}^K\sigma|_{K_M}\times \Ind_{K_P}^K\sigma^\lor|_{K_M},$ the zeta integral 
\begin{align*}
    Z_0(\mathrm{Ind}_P^G \sigma_\chi,f,v,w):=\int_{G(F)} f(g)\langle (\mathrm{Ind}_P^G\sigma_\chi)(g)v,w\rangle dg
\end{align*} 
converges absolutely for $\Re (\sigma_\chi)=\Re(\chi)=0.$ The subspace  $\mathcal{S}^{\mathrm{as}}_\rho(G(F))$ consists of functions $f$ such that $Z_0(\mathrm{Ind}_P^G \sigma_\chi,f,v,w)$ extends to a rational function in $\chi$ that is a holomorphic multiple of $L(\tfrac{1}{2},\sigma_\chi,\rho|_{{}^LM})$. Here the $L$-factors are defined via the local Langlands correspondence \eqref{eq:Lgamma:LL}. 

A function $f\in \mathcal{S}_\rho^{\mathrm{as}}(G(F))$ can have bad analytic behaviors if $G/Z_G$ is isotropic as observed in \cite[Example 6.2]{DRS}. The main obstruction arises as follows: if $\pi$ is a non-supercuspidal irreducible unitary discrete series of $G(F),$ then there is a proper standard parabolic subgroup $P$ of $G$  such that $\pi$ is a subrepresentation of $\Ind_P^G \sigma$ for some (nonunitary) supercuspidal representation $\sigma$ of $M(F).$ We can then identify a matrix coefficient $\langle \pi(g)v,w\rangle$ of $\pi$ as a matrix coefficient $\langle (\mathrm{Ind}_P^G\sigma)(g)v,\tilde{w}\rangle$ of $\Ind_P^G \sigma$ for some lift $\tilde{w}$ of $w$. However, one can find $f\in \mathcal{S}_\rho^{\mathrm{as}}(G(F))$ such that $Z_0(\pi,f,v,w)\neq 0$ while $Z_0(\mathrm{Ind}_P^G\sigma,f, v,\tilde{w})=0,$ which contradicts the expectation that these two quantities are equal. In the above example, one can take $f$ supported on $G(F)^1$ (see \eqref{eq:GF1}), on which it agrees with a matrix coefficient of $\pi$ restricted to $G(F)^1$. Thus $f$ is not compactly supported in $G(F)^1.$

To avoid the pathological examples above, \cite{DRS} introduced the space $C^\infty_{\mathrm{ac}}(G(F))$ consisting of smooth functions on $G(F)$ whose restriction to any coset of $G(F)^1$ in $G(F)$ are compactly supported (see \eqref{eq:ac}). As a more direct approach, we introduce a space $\mathcal{C}_{\mathrm{cs}}(G(F))\subseteq \mathcal{C}(G(F))$ consisting of functions with compatible zeta functions (see \hyperref[def:cs]{Definition \ref{def:cs}}). The third approach is to observe that examples above arise because the zeta functions of general $f\in \mathcal{S}^{\mathrm{as}}_\rho(G(F))$ only need to extend meromorphically but do not necessarily converge absolutely in an open region. Therefore, we define in \S \ref{sec:Schwartz:analytic} a space $\mathcal{S}_\rho(G(F)),$ referred to as  the \textbf{$\rho$-Schwartz space},  consisting of functions whose zeta integrals converge in positive cones (see \hyperref[def:positivecone]{Definition \ref{def:positivecone}}).

One of our main results is to show that all three fixes coincide and give the correct Schwartz spaces attached to $\rho$ (up to a central shift).
\begin{theorem}\label{thm:main:compactsupport}
    We have
    \begin{align*}
        \mathcal{S}_\rho(G(F))=|\nu|^{-1/2}\left(\mathcal{S}_\rho^{\mathrm{as}}(G(F))\cap \mathcal{C}_{\mathrm{cs}}(G(F))\right)=|\nu|^{-1/2}\mathcal{S}_\rho^{\mathrm{as}}(G(F))\cap C^\infty_{\mathrm{ac}}(G(F)).
    \end{align*}
    Furthermore, any function in the spaces above has compact support in $X_{\rho}(F),$ the $L$-monoid attached to $\rho$ constructed in \cite[\S5]{Ngo:Hankel}.
\end{theorem}

\begin{remark}
    We have chosen the normalization shift $|\nu|^{-1/2}$ in the definition of $\mathcal{S}_\rho(G(F))$ so that analytic behaviors of functions in $\mathcal{S}_\rho(G(F))$ are controlled via $L(s,\sigma,\rho|_{{}^LM})$ instead of $L(\tfrac{1}{2}+s,\sigma,\rho|_{{}^LM}).$ Under this shift, the Fourier transform on $\mathcal{S}_\rho(G(F))$ is $|\nu|^{-\frac{1}{2}}\mathcal{F}_{\rho,\psi}|\nu|^{\frac{1}{2}}$.
\end{remark}

Sections \ref{sec:Schwartz:spectral} and  \ref{sec:Schwartz:analytic} are devoted to establishing the identities of these three spaces. The most technical part of the proof is to show the inclusion $|\nu|^{-1/2}\mathcal{S}_\rho^{\mathrm{as}}(G(F))\cap C^\infty_{\mathrm{ac}}(G(F))\subseteq \mathcal{S}_\rho(G(F)),$ which is done in \S \ref{ssec:Supportproof}. Our argument is a modification of the proof of matrical Paley-Wiener theorem (\hyperref[BH:PW]{Theorem \ref{BH:PW}}) in \cite{Heiermann:HeckePlan} that uses an analogue of Harish-Chandra's Plancherel formula. 

As observed in \cite{DRS}, to prove that $\mathcal{S}_\rho^{\mathrm{as}}(G(F))\cap C^\infty_{\mathrm{ac}}(G(F))$ is stable under $\mathcal{F}_{\rho,\psi}$ requires information of cuspidal supports of discrete series of all Levi subgroups of $G$. Indeed, we show in \hyperref[LLC7=Fourierstable]{Proposition \ref{LLC7=Fourierstable}} that $\mathcal{S}_\rho^{\mathrm{as}}(G(F))\cap C^\infty_{\mathrm{ac}}(G(F))$ being stable under $\mathcal{F}_{\rho,\psi}$ is equivalent to \eqref{LLC:gamma}. The assumption \eqref{LLC:gamma} holds for general linear groups \cite[Proposition 3.3]{DRS} due to the Bernstein-Zelevinsky classification.

 Assuming in addition that \eqref{LLC:gamma} holds, we prove that local $L$-factors can be realized as the greatest common divisors of zeta integrals in \hyperref[multigcd]{Theorem \ref{multigcd}} and \hyperref[thm:gcd]{Theorem \ref{thm:gcd}}. For simplicity, we state a weaker form of \hyperref[thm:gcd]{Theorem \ref{thm:gcd}} below.

By the Langlands classification \cite{Silb:classification, BW:book, Representation-p-adic}, each irreducible smooth representation $\pi$ of $G(F)$ is the unique irreducible quotient of some $\mathrm{Ind}_P^G \sigma,$ where $P$ is a standard parabolic subgroup of $G$, and $\sigma$ is an irreducible essentially tempered representation of $M(F)$ with dominant $\Re(\sigma)$. We let
\begin{align*}
    L(s,\pi,\rho):=L(s,\sigma,\rho|_{{}^LM }),\quad \quad \gamma(s,\pi,\rho,\psi):=\gamma(s,\sigma,\rho|_{{}^L M},\psi).
\end{align*}

\begin{theorem}\label{BKN:conj} Suppose $G$ is split and \eqref{LLC:LCFT}--\eqref{LLC:gamma} hold for all Levi subgroups of $G$. Let $\pi$ be an irreducible smooth representation of $G(F)$ and let $\mathcal{C}(\pi)$ denote the space of matrix coefficients of $\pi.$
\begin{enumerate}
    \item\label{BKN:zeta} Let $f\in\mathcal{S}_\rho(G(F))$ and $c\in\mathcal{C}(\pi).$ The local zeta integral
    \begin{align*}
        Z(s,f,c) := \int_{G(F)} f(g)c(g)|\nu(g)|^s dg
    \end{align*}
    is absolutely convergent for $\mathrm{Re}(s)$ large and extends to a rational function in $q^{-s}$.
    \item\label{BKN:span:L} The set
    \begin{align*}
        \mathrm{Span}_{\mathbb{C}}\{Z(s,f,c):(f,c)\in\mathcal{S}_\rho(G(F))\times\mathcal{C}(\pi)\}
    \end{align*}
    is a nonzero fractional ideal of $\mathbb{C}[q^{\pm s}]$ generated by $L(s,\pi,\rho)$.
    \item\label{BKN:basic} There exists a basic function $|\nu|^{-\frac{1}{2}}b_\rho\in\mathcal{S}_\rho(G(F))$ such that $Z(s,|\nu|^{-\frac{1}{2}}b_\rho,c)$ is only nonzero when $\pi$ is unramified and $c$ is bi-$K$-invariant, in which case $Z(s,|\nu|^{-\frac{1}{2}}b_\rho,c^\circ) = L(s,\pi,\rho)$ where $c^\circ$ is the zonal spherical function of $\pi$. Furthermore, $\mathcal{F}_{\rho,\psi}(b_\rho) = b_\rho$ when $\psi$ is unramified.
       
    \item\label{BKN:FT}  We have functional equations
        \begin{align*}
        Z(1-s,|\nu|^{-\frac{1}{2}}\mathcal{F}_{\rho,\psi}|\nu|^{\frac{1}{2}} f,c^\vee) = \gamma(s,\pi,\rho,\psi)Z(s,f,c).
        \end{align*}
\end{enumerate}
\end{theorem}

Our proof of \hyperref[BKN:conj]{Theorem \ref{BKN:conj}.(2)} relies on Casselman's criterion on square-integrability, the Geometric lemma, the matrical Paley-Wiener theorem and \hyperref[thm:main:compactsupport]{Theorem \ref{thm:main:compactsupport}} to construct functions in $\mathcal{S}_\rho(G(F))$ spectrally. Indeed, we prove a multivariate generalization of \hyperref[BKN:conj]{Theorem \ref{BKN:conj}.(2)} in \hyperref[multigcd]{Theorem \ref{multigcd}}: for an irreducible unitary discrete series $\sigma$ of a standard Levi subgroup $M(F)$
\begin{align*}
    \CC[\Lambda_M]=\mathrm{Span}_\CC\left\{\frac{Z_0(\mathrm{Ind}_P^G \sigma_\chi,f,v,w)}{L(0,\sigma_\chi,\rho|_{{}^LM})}: f\in \mathcal{S}_\rho(G(F)),(v,w)\in \Ind_{K_P}^K\sigma\vert_{K_M}\times \Ind_{K_P}^K\sigma^\vee\vert_{K_M}\right\},
\end{align*}
where $\Lambda_M$ is the complex variety of the unramified quasi-characters of $M(F)$. As a consequence, we deduce that the na\"ive Schwartz space $C^\infty_c(G(F))+\mathcal{F}_\rho(C^\infty_c(G(F)))$ does not contain the basic function $b_\rho$ except possibly when $G$ has split rank $1$! In particular, in general the na\"ive Schwartz space is a proper subspace of $|\nu|^{1/2}\mathcal{S}_\rho(G(F))$ (see \hyperref[remark:small]{Remark \ref{remark:small}}).

Finally, in \S\ref{sec:toric:Schwartz} and \S\ref{sec:Lmonoids} we study the underlying geometry of the $\rho$-Schwartz spaces. In particular, we show that $b_\rho$ is supported on the integral points of the $L$-monoid $X_\rho$. As every reductive monoid is governed by an affine toric variety, we discuss in \S\ref{sec:toric:Schwartz}  in detail $\mathcal{S}_\rho(G(F))$ when $G$ is a torus, extending the discussion in \cite[\S 2.2]{Luo:Ngo}. Examples related to symmetric powers of $\GL_2$ are discussed in \S\ref{ssec:example}.

A key conjectural property of $\mathcal{S}_\rho(G(F))$ that we do not discuss in this paper is \textbf{locality}. Let $C^\infty(X_\rho(F))$ be the space of locally constant functions on $X_\rho(F),$ viewed as a function space on $G(F)$ by restriction.
\begin{conjecture}
    The Schwartz space $\mathcal{S}_\rho(G(F))$ is local, i.e., it is a $C^\infty(X_\rho(F))$-module under the usual function multiplication.
\end{conjecture}

We verify the conjecture when $G$ is a torus in \hyperref[thm:localT]{Theorem \ref{thm:localT}}. In an attempt to tackle the conjecture in general, we propose the following problem that may serve as an intermediate step to understand the $G(F)\times G(F)$-module structure of $\mathcal{S}_\rho(G(F))$. Let $C^\infty_c(X_\rho(F))$ be the space of compactly supported locally constant functions on $X_\rho(F)$. View it as a function space on $G(F)$ by restriction.
\begin{conjecture}\label{conj:main:SES}
    The space $C^\infty_c(X_\rho(F))$ is contained in $\mathcal{S}_\rho(G(F))$, up to a unique central twist.
\end{conjecture}
\hyperref[conj:main:SES]{Conjecture \ref{conj:main:SES}} if proven would allow us to write down easily more functions in $\mathcal{S}_\rho(G(F)).$ We verify the conjecture when $G$ is a split torus in \hyperref[prop:ESinrho]{Proposition \ref{prop:ESinrho}} and for examples in \S\ref{ssec:example}.  Establishing both conjectures may require deeper understanding of discrete series and geometric objects associated to $\rho$. We leave this for future work.

\vspace{0.2cm}
\begin{center}
    \textsc{Acknowledgment}
\end{center}
The first author thanks Bảo Ch\^au Ng\^o for teaching a course on the paper \cite{Luo:Ngo} and suggesting the study of the $\rho$-Schwartz space for $\GL_2,$ which leads to the author's interest in this topic.  
The second author thanks his advisor, Jayce R. Getz, for proposing this topic.

\section{Preliminaries}\label{sec:prelim}

We set up notations and review Harish-Chandra Schwartz spaces and the Harish-Chandra Plancherel theorem in \S \ref{ssec:characters}-\S \ref{ssec:constant}. Our main references are \cite{Waldspurger:Plancherel, Representation-p-adic}. In \S\ref{ssec:assumptions} we state the tempered local Langlands correspondence which we assume in the rest of the paper. We deduce some consequences on local $L$-factors in \S\ref{ssec:LLC}. 

\subsection{Notations} For $n\in \ZZ_{>0},$ let $[n]:=\{1,\ldots, n\}$. For two positive constants $a,b$, we write $a\ll_{?} b$ if there exists a constant $C_{?}>0$ depending on $?$ such that $a\leq C_{?}b$. For two sets $X,Y,$ we often write $X\le Y$ instead of $X\subseteq Y$ if the inclusion map preserves additional structures shared by both $X$ and $Y$. The same convention also applies to $\subset, \subsetneq, \not\subseteq$. For a subset $X$ in an $E$-module for some subring $E$ of $\RR,$ we write
\begin{align*}
    E_{\geq 0}X&:=\left\{\sum_{x\in I}a_xx: a_x\in E\cap \RR_{\geq 0},I\subseteq X \textrm{ finite}\right\}.
\end{align*}
Define analogously for $E_{>0}X$ and $E_{\le 0} X$.
For two subsets $X,Y$ in a module, let $X+Y:=\{x+y: x\in X, y\in Y\}.$ For a $\ZZ$-module $M$ and a (unital) ring $R$, we write $M_R := M\otimes_\ZZ R$. For two complex vector spaces $V$ and $W$, we write $V\otimes W = V\otimes_\CC W$ unless otherwise stated. For a group $G$, we write $\mathrm{id}=\mathrm{id}_G$ for its identity element. For a topological space $X$, let $\pi_0(X)$ be the set of connected components of $X$.

\subsection{Groups, characters and cocharacters}\label{ssec:characters}

Let $G$ be a connected reductive group over $F$. Let $Z_G$ denote its center and let $A_G$ be the maximal split subtorus of $Z_G$. Let $X^\ast(G)$ be the group of rational characters of $G$. When $G$ is commutative, let $X_*(G)$ be the group of rational cocharacters of $G$. Let $H_G:G(F)\to\Hom_{\ZZ}(X^*(G),\ZZ)$ be the function given by
\begin{align*}
    |\chi(g)| = q^{-\langle\chi,H_G(g)\rangle}
\end{align*}
for $(g,\chi)\in G(F)\times X^*(G).$ Set
\begin{align}\label{eq:GF1}
    G(F)^1 := \ker H_G = \bigcap_{\chi\in X^*(G)} \ker |\chi|
\end{align}
and put
\begin{align*}
    \Lambda_G &:= \Hom_{\mathbf{TopGp}}(G(F)/G(F)^1,\mathbb{C}^\times),\\
    \Re\Lambda_G &:= \Hom_{\mathbf{TopGp}}(G(F)/G(F)^1,\mathbb{R}_{>0}),\qquad 
    \Im\Lambda_G := \Hom_{\mathbf{TopGp}}(G(F)/G(F)^1,S^1).
\end{align*}
Under the usual identification $\mathbb{C}^\times = \mathbb{R}_{>0}\times S^1$, one has $\Lambda_G = \Re\Lambda_G\times \Im\Lambda_G$. For $\chi\in\Lambda_G$, we write $\Re(\chi)$ for its projection to $\Re\Lambda_G$. One has a surjection
\begin{equation}\label{XG->LambdaG:surj}
    \begin{tikzcd}[row sep=tiny]
        X^*(G)_\mathbb{C}\arrow[r]&\Lambda_G\\
        \chi\otimes z\arrow[r,maps to]&{(g\mapsto |\chi(g)|^z)}.
    \end{tikzcd}
\end{equation}
 The map restricts to a group isomorphism
\begin{equation}\label{XG->LambdaG:bij:rel}
    \begin{tikzcd}[row sep=tiny]
        X^*(G)_\mathbb{R}\arrow[r,"\sim"]&\Re\Lambda_G
    \end{tikzcd}
\end{equation}
and a surjection
\begin{center}
    \begin{tikzcd}[row sep=tiny]
        X^*(G)\otimes_{\mathbb{Z}}i\mathbb{R}\arrow[r,two heads]&\Im\Lambda_G.
    \end{tikzcd}
\end{center}
This identifies $\Lambda_G$ with the $\mathbb{C}$-points of an algebraic torus over $\mathbb{C}$. We sometimes write $X^\ast(G)_\RR$ and $\Re \Lambda_G$ interchangeably. 

 We can describe the algebraic structure of $\Lambda_G$ more explicitly. Let $\Omega_G := G(F)/G(F)^1$. We often identify $\Omega_G$ with the image of $H_G$ so that $\Omega_G\subseteq \Hom_{\ZZ}(X^*(G),\ZZ)$. Then
\begin{align*}
    \bigg(\Spec \CC[\Omega_G]\bigg)(\CC) = \Hom_{\CC\textrm{-}\textbf{Alg}}(\CC[\Omega_G],\CC) = \Hom_{\textbf{Gp}}(\Omega_G,\CC^\times) = \Lambda_G.
\end{align*}
This implies $X^*(\Lambda_G) \cong \Omega_G$. Explicitly, we associate to $\alpha\in\Omega_G$ a homomorphism $\Lambda_G\ni\chi\mapsto \chi(\alpha),$ where $\chi(\alpha)$ is the value of $\chi$ on the coset $\alpha G(F)^1$. Choose any lift $\tilde{\chi}\in X^\ast(G)_\CC$ of $\chi,$ then $\chi(\alpha)=q^{-\langle \tilde{\chi},\alpha\rangle}.$ Thus we also write $\chi(\alpha) =q^{-\langle\chi,\alpha\rangle}.$ Then an element $f\in \CC[\Lambda_G]$ has the form
\begin{align}\label{LambdaG:poly:supp}
    f(\chi) = \sum_{\alpha\in \Omega_G} c_\alpha \chi(\alpha) = \sum_{\alpha\in \Omega_G} c_\alpha q^{-\langle\chi,\alpha\rangle},
\end{align}
where $c_\alpha\in \CC$ is zero for all but finitely many $\alpha\in\Omega_G$.

Fix a maximal split torus $A_0$ of $G$, and let $M_0:=C_G(A_0)$ be the centralizer of $A_0$ in $G$. Fix a minimal parabolic subgroup $P_0$ with Levi component $M_0$ and unipotent radical $N_0$. A parabolic subgroup $P$ of $G$ is said to be semi-standard (resp. standard) if $A_0\le P$ (resp. $P_0\le P$). Let $\mathcal{P}$ (resp. $\mathcal{P}^{\mathrm{std}}$) denote the set of semi-standard (resp. standard) parabolic subgroups of $G$. A semi-standard (resp. standard) Levi
subgroup $M$ is the unique Levi subgroup of some $P\in \mathcal{P}$ (resp. $P\in\mathcal{P}^{\mathrm{std}}$) that contains
$A_0$. Let $\mathcal{M}$ (resp. $\mathcal{M}^{\mathrm{std}}$) denote the set of semi-standard (resp. standard) Levi subgroups of $G$. For each $M\in \mathcal{M},$ let $\mathcal{P}(M)$ be the set of $P\in \mathcal{P}$ with Levi component $M$. When $M\in \mathcal{M}^{\mathrm{std}},$ there is a unique standard parabolic $P\in \mathcal{P}(M),$ so we have a bijection $\mathcal{P}^{\mathrm{std}}\stackrel{\sim}{\longrightarrow} \mathcal{M}^{\mathrm{std}}.$ Whenever we say $P=MN$ is a semi-standard parabolic subgroup or $P=MN\in \mathcal{P}$, $M$ is the semi-standard Levi subgroup of $P$ and $N$ is the unipotent radical of $P$. We let  $\overline{P}=M\overline{N}$ be its opposite parabolic. Note that the product map
\begin{align*}
    \overline{N}\times M\times N\longrightarrow G
\end{align*}
is an open immersion. Also, if $M\in \mathcal{M}^{\mathrm{std}},$ we write $P=MN$ for the unique $P\in \mathcal{P}^{\mathrm{std}}\cap\mathcal{P}(M).$

For each $M\in \mathcal{M},$ let
\begin{align*}
    W(G,A_M):=N_G(A_M)(F)/C_G(A_M)(F)=N_G(A_M)(F)/M(F)=N_G(M)(F)/M(F)
\end{align*}
be the Weyl group. We have a natural exact sequence of finite groups
\begin{align*}
    1\longrightarrow W(M,A_0)\longrightarrow W(M,M):=(N_G(M)(F)\cap N_G(A_0)(F))/M_0(F)\longrightarrow W(G,A_M)\longrightarrow 1.
\end{align*}
For each $s\in W(G,A_M),$ we choose a representative in $N_G(A_M)(F),$ which we also denote by $s$. The Weyl group $W(G,A_M)$ has a natural action on $X^\ast(A_M)$ given by 
\begin{align*}
    s.\chi(a):=\chi(s^{-1}as), \quad \textrm{for } a\in A_M(F).
\end{align*}
Therefore, it also acts on $X^\ast(M)_\CC, X^\ast(M)_\RR\cong \Re\Lambda_M, \Im\Lambda_M$ and $\Lambda_M$. 

For semi-standard Levi subgroups $M'\leq M,$ we have a natural inclusion
\begin{center}
    \begin{tikzcd}
        X^\ast(M)\arrow[rr,hook,"\iota_{M\ge M'}"]&&X^\ast(M').
    \end{tikzcd}
\end{center}
Since $A_M\le A_{M'}$, the restriction gives a canonical projection
\begin{center}
    \begin{tikzcd}
        X^\ast(M')_\RR=X^\ast(A_{M'})_\RR\arrow[rr,two heads,"p_{M'\leq M}"]&&X^\ast(A_M)_\RR = X^\ast(M)_\RR
    \end{tikzcd}
\end{center}
which is a section to $\iota_{M\ge M'}$ over $\RR$. Let $X_{M'}^{M} := \mathrm{ker}(p_{M'\leq M})$ and let $\Re\Lambda_{M'}^{M}$ be its image in $\Re \Lambda_{M'},$ so we have splittings
\begin{align*}
    X^\ast(M')_\RR=X^\ast(M)_\RR\oplus X_{M'}^{M} \quad (\textrm{resp.}\quad \Re\Lambda_{M'} &= \Re\Lambda_M \times \Re\Lambda_{M'}^{M}).
\end{align*}
For an element $\chi$ in $X^\ast(M')_\RR$ (resp. $\Re\Lambda_{M'}$) we write
\begin{align*}
    \chi=\chi|_M+ \chi|_{M^\perp} \quad (\textrm{resp}.\quad \chi=\chi|_M\cdot \chi|_{M^\perp})
\end{align*}
for the decomposition. Thus $\chi|_M=p_{M'\le M}(\chi).$

For $M\in \mathcal{M},$ let $\Sigma(A_M)$ be the set of roots of $A_M$ in $\mathfrak{g}=\mathrm{Lie}(G).$ For $P=MN\in \mathcal{P},$ let $\Sigma(P)\subseteq \Sigma(A_M)$ be the set of positive roots relative to $P,$ and let $\Delta(P)$ be the set of simple roots in the set of reduced positive roots $\Sigma_{\mathrm{red}}(P)$. Let $\Delta_G=\Delta:=\Delta(P_0)$ be the subset of simple roots of $\Sigma_{\mathrm{red}}(P_0)\subseteq\Sigma:=\Sigma(P_0).$ Let $\Sigma^{\lor}:=\{ \alpha^\lor:\alpha\in \Sigma\}$ be the set of coroots. Let $\widehat{\Delta}:=\{\omega_\alpha:\alpha\in \Delta\}\subset X^\ast(M_0)_\QQ$ be the set of fundamental weights, i.e., 
\begin{align*}
    \langle \omega_\alpha,\beta^\lor\rangle=\delta_{\alpha,\beta}, \quad \quad \textrm{for all } \alpha,\beta\in \Delta,
\end{align*}
where $\delta_{\alpha,\beta}$ is the Kronecker $\delta$.
Let $\eta_G$ be the half sum of positive roots of $G$ (with multiplicities), i.e.,
\begin{align*}
    \eta_G:=\frac{1}{2}\sum_{\alpha\in \Sigma(P_0)} (\dim_F \mathfrak{g}_\alpha)\alpha,
\end{align*}
where $\mathfrak{g}_\alpha$ is the root space of $\alpha.$ Note that $\eta_G\in \QQ_{>0}\widehat{\Delta}.$ For each $P=MN\in \mathcal{P}$ and $s\in W(G,A_0),$ set $s.P:=sPs^{-1}\in \mathcal{P},$  $s.M:=sMs^{-1}\in \mathcal{M},$ and $s.N=sNs^{-1}$ is the unipotent radical of $s.P$. For each minimal parabolic  $P\in\mathcal{P},$ there is a unique $s\in W(G,A_0)$ such that $s. P=P_0$. Therefore, for each $P=MN\in \mathcal{P}$ there is a unique $s\in W(G,A_0)/W(M,A_0)$ such that $s.P$ is standard.

For each $M\in \mathcal{M}^{\mathrm{std}},$ let
\begin{align*}
    X^*(M)^+_\RR &:= \{\lambda\in X^*(M)_\RR : \langle \lambda,\alpha^\lor\rangle:=\langle \iota_{M\ge M_0}(\lambda),\alpha^\lor\rangle > 0\text{ for all }\alpha\in \Delta_G-\Delta_M\}\\
    &=\left\{\sum_{\alpha\in \Delta(P)} a_\alpha\omega_\alpha+\chi: a_\alpha>0,\,\chi\in X^{\ast}(G)_\RR\right\}
\end{align*}
be the dominant Weyl chamber, and $\overline{X^\ast(M)}^+_\RR$ be its closure.

\subsection{Decompositions and measures}\label{subsec:Cartan} 

Let $K$ be the stabilizer of a special point in the apartment of $G$ associated to $A_0$.  When $G$ is unramified, we assume $K$ is the stabilizer of a hyperspecial point.  Then $K\leq G(F)$ is a maximal compact subgroup, and by \cite[Theorem 5.3.4]{KP:BruhatTits} one has $G(F) = P(F)K = KP(F)$ for any parabolic subgroup $P$ of $G$. For a closed algebraic $F$-subgroup $H\le G,$ let $K_H:=H(F)\cap K.$  Let $d_\ell h$ be the left-invariant Haar measure on $H(F)$ such that $\mathrm{vol}(K_H,d_\ell h)=1$. For $P\in \mathcal{P},$ let $\delta_P$ be the modular character of $P(F)$. Note that for $m\in M_0(F)$
\begin{align*}
    \delta_{P_0}(m)=q^{-\langle 2\eta_G,H_{M_0}(m)\rangle}.
\end{align*}

For any $M\in \mathcal{M}$, by \cite[\S 6.7.2 and Corollary 9.7.3]{KP:BruhatTits} $K_M $ is again the stabilizer of a special point in the apartment of $M$ associated to $A_0$, and $K_P = K_MK_N$ for any $P=MN\in \mathcal{P}$. Choose continuous maps $\mathbf{a}_P:G(F)\to M(F),$ $\mathbf{n}_P:G(F)\to N(F)$ and $\mathbf{k}_P:G(F)\to K$ so that
\begin{align*}
    g = \textbf{a}_P(g)\mathbf{n}_P(g)\textbf{k}_P(g)
\end{align*}
for all $g\in G(F)$. When $P=P_0,$ we drop the subscript $P_0$ and write $\mathbf{a}=\mathbf{a}_{P_0},\mathbf{n}=\mathbf{n}_{P_0},\mathbf{k}=\mathbf{k}_{P_0}.$

Put $\Omega:=\Omega_{M_0}=M_0(F)/M_0(F)^1.$ If $G$ is unramified, one has $\Omega=A_0(F)/A_0(F)^1=X_*(A_0)$. Let
\begin{align*}
    M_0(F)^+ := \{m\in M_0(F) : \langle \alpha,H_{M_0}(m)\rangle\ge 0 \text{ for all }\alpha\in\Delta\}.
\end{align*}
The subgroup $M_0(F)^1 = K_{M_0}$ is compact, so $M_0(F)^1\le M_0(F)^+$ and $\Omega^{+} := M_0(F)^+ / M_0(F)^1$ is a submonoid of $\Omega$.  By \cite[Theorem 5.2.1]{KP:BruhatTits}, one has the Cartan decomposition
\begin{align}\label{decomp:group:Cartan}
    G(F) = \bigsqcup_{m\in \Omega^+} KmK.
\end{align}
Furthermore, there are positive constants $C_1<C_2$ such that for all $m\in \Omega^+$
\begin{align}\label{Cartan:measure}
    C_1\delta_{P_0}^{-1}(m)<\mathrm{vol}(KmK,dg)<C_2\delta_{P_0}^{-1}(m).
\end{align}

 For a compact open group $K'$ of $G(F)$ and a closed algebraic $F$-subgroup $H\le G,$ let $H_{K'}:=H(F)\cap K'.$ Let $P=MN\in \mathcal{P}$. We say $K'$ possesses an Iwahori factorization with respect to $P(F)$ if the product map
\begin{align*}
    \overline{N}_{K'}\times M_{K'} \times N_{K'}\longrightarrow K'
\end{align*}
is an isomorphism, and $ aN_{K'}a^{-1}\le N_{K'}, a^{-1} \overline{N}_{K'}a\le  \overline{N}_{K'}$ for every $a\in A_{M}^+,$ where
\begin{align*}
    A_{M}^+:=\{a\in A_M(F): |\alpha(a)|\le 1 \quad\forall \alpha\in \Delta(P)\}.
\end{align*}
By \cite[Proposition 1.4.4]{Casselman1995} there exists an open neighborhood basis $\{K_n\}_{n\ge 0}$ of the identity consisting of compact open subgroups such that the following holds: 
\begin{enumerate}
    \item Each $K_n$ is normal in $K_0,$ where $K_0\le K$ is the standard Iwahori subgroup with respect to $P_0$.
    \item Each $K_n$ has an Iwahori factorization with respect to all standard parabolic subgroups of $G$.
\end{enumerate}

\subsection{Representations} 

For any (complex) smooth representation $(\pi,V)$ of $G(F)$ and $\chi\in\Lambda_G$, let $$\pi_\chi(g) := \pi(g)\chi(g)$$ for $g\in G(F)$. Let $\pi^\lor$ be the smooth dual of $\pi$, and let $\langle \cdot ,\cdot \rangle_\pi$ be the natural pairing on $\pi\times \pi^\lor.$ We will drop the subscript $\pi$ whenever the context is clear. For $\pi$ a smooth admissible representation, let $\mathrm{End}_\CC(\pi)$ denote the space of smooth endomorphisms of the underlying vector space of $\pi$. Note that $\mathrm{End}_\CC(\pi)\cong \pi\otimes \pi^\vee$ as representations of $G(F)\times G(F)$.

Let
\begin{align*}
    \Pi_0(G)\subseteq \Pi_2(G)\subseteq \Temp(G)
\end{align*}
be the isomorphism classes of smooth irreducible unitary representations of $G(F)$ that are supercuspidal, (relatively) square integrable, and tempered respectively.  We equip $\Pi_2(G)_\CC := \{\pi_\chi: (\pi,\chi)\in\Pi_2(G)\times \Lambda_G\}$ with the unique topology so that the action map $\Lambda_G\times \Pi_2(G)_\CC\to \Pi_2(G)_\CC$ is continuous, and for each $\Lambda_G$-orbit $\mathcal{O}$ of  $\Pi_2(G)_\CC,$ the induced map $\Lambda_G\longrightarrow\mathcal{O}$ is a quotient map. Note that $\Pi_0(G)$, $\Pi_2(G)$ and $\Temp(G)$ are invariant under $\Im\Lambda_G$. We equip $\Pi_0(G)$ and $\Pi_2(G)$ with the subspace topology.

For $P=MN\in \mathcal{P}$ and a smooth representation $\sigma$ of $M(F),$ viewed as a smooth representation of $P(F)$ via the natural map $P(F)\to P(F)/N(F)\cong M(F),$ let $\mathrm{Ind}_P^G \sigma$ be the normalized parabolic induction in the category of smooth representations. For each $\chi\in\Lambda_M,$ restriction to $K$ induces a $K$-equivariant isomorphism
\begin{center}
    \begin{tikzcd}
        \Ind_P^G\sigma_\chi\arrow[r]&\Ind_{K_P}^{K}\sigma\vert_{K_M }.
    \end{tikzcd}
\end{center}
For $v\in \Ind_{K_P}^{K}\sigma\vert_{K_M }$, let $v_{\sigma_\chi}=v_\chi\in \Ind_P^G\sigma_\chi$ denote its fibre so that for $p\in P(F)$ and $k\in K$
\begin{align*}
    v_{\sigma_\chi}(pk)=v_\chi(pk) = \delta_P^{\frac{1}{2}}(p)\sigma_\chi(p)v(k).
\end{align*}
We have a canonical isomorphism $(\mathrm{Ind}_P^G \sigma_\chi)^\lor\cong \Ind_{P}^G (\sigma_\chi)^\lor.$ Under this isomorphism, by the Iwasawa decomposition  $\langle\cdot,\cdot\rangle_{\mathrm{Ind}_P^G\sigma_\chi}$ is given by
\begin{align*}
    \langle (\Ind_P^G\sigma_\chi)(g)v_\chi,w_\chi\rangle_{\Ind_P^G\sigma_\chi}&=\int_{P(F)\backslash G(F)} \langle v_\chi(hg),w_\chi(h)\rangle_\sigma\, d\dot{h} = \int_K \langle v_\chi(kg),w(k)\rangle_\sigma dk
\end{align*}
for $(v,w)\in \Ind_{K_P}^{K}\sigma\vert_{K_M }\times\Ind_{K_P}^{K}\sigma^\vee\vert_{K_M }$. We often suppress the subscript $\chi$ on vectors when the representation involved is obvious from the context. In particular,
\begin{align*}
    \langle v,w\rangle_{\mathrm{Ind}_P^G\sigma_\chi}=\langle v,w\rangle_{\mathrm{Ind}_{K_P}^K \sigma|_{K_M}}.
\end{align*}

Let
\begin{align*}
    \widetilde{\Temp}_\Ind(G) &:= \bigsqcup\limits_{M\in\mathcal{M}^{\mathrm{std}}} \Pi_2(M),\quad\quad\quad    \widetilde{\Temp}_{\Ind,0}(G) := \bigsqcup\limits_{M\in\mathcal{M}^{\mathrm{std}}} \Pi_0(M)
\end{align*}
equipped with the colimit topology. Let $\Temp_\Ind(G)$ (resp. $\Temp_{\Ind,0}(G)$) denote the isomorphism classes of representations of the form $\Ind_P^G\sigma$ with $(M,\sigma)\in\widetilde{\Temp}_\Ind(G)$ (resp. with $(M,\sigma)\in\widetilde{\Temp}_{\Ind,0}(G)$). Then $\Temp_\Ind(G)$ (resp. $\Temp_{\Ind,0}(G)$) is a quotient of $\widetilde{\Temp}_\Ind(G)$ (resp. of $\widetilde{\Temp}_{\Ind,0}(G)$). 

For $M\in \mathcal{M}^{\mathrm{std}},$ we equip $\Im\Lambda_{A_M}$ with the Haar measure of total mass $1$, and equip $\Im\Lambda_M$ with the pullback measure from that on $\Im\Lambda_{A_M}$ along the surjection $\Im\Lambda_{M}\to\Im\Lambda_{A_M}$. Equip $\Pi_2(M)$ with the measure $d\sigma$ so that each orbit acquires the pushforward measure of $\Im\Lambda_M$ along the action map. 

Define a regular Borel measure $d\pi$ on $\Temp_\Ind(G)$ as follows. For $f\in C_c(\Temp_\Ind(G))$, view it as a function on $\widetilde{\Temp}_\Ind(G)$ and set
\begin{align*}
    \int_{\Temp_\Ind(G)} f(\pi) d\pi := \sum_{M\in\mathcal{M}^{\mathrm{std}}} \frac{\gamma(G|M)}{\#W(G|M)} \int_{\Pi_2(M)} f(M,\sigma) \mu(\sigma)j(\sigma)^{-1} d\sigma.
\end{align*}
Here 
\begin{align*}
    W(G|M):=\{s\in W(G,A_0):s. M=M\}/W(M,A_0)=W(M,M)/W(M,A_0)\cong W(G,A_M),
\end{align*}
and $\gamma(G|M)$ is defined in \cite[p.241]{Waldspurger:Plancherel}, $\mu(\sigma)$ is the formal degree of $\sigma$, and $j(\sigma)$ is defined in \cite[p.285]{Waldspurger:Plancherel}.

\subsection{Intertwining operators}\label{subsec:intertwining:recall} We recall several facts of intertwining operators in \cite{Waldspurger:Plancherel}. Let $P=MN\in \mathcal{P}$ and $(\sigma,V)$ a smooth representation of $M(F)$. For $s\in W(G,A_0)$, let $(s.\sigma,V)$ denote the representation of $s.M(F)$ given by 
\begin{align*}
    s.\sigma(m) := \sigma(s^{-1}ms).
\end{align*}
Set
\begin{align*}
\lambda(s):\Ind_{P}^G \sigma &\longrightarrow \Ind_{s.P}^{G} s.\sigma\\
    v&\longmapsto v_s:g\mapsto v(s^{-1}g).
\end{align*}
This is a $G(F)$-equivariant isomorphism, and the same formula restricts to a $K$-equivariant isomorphism
\begin{align}\label{leftW}
\begin{split}
    \lambda(s):\Ind_{K_P}^K \sigma\vert_{K_M}&\longrightarrow \Ind_{K_{s.P}}^{K} s.\sigma\vert_{K_{s.M}}\\
    v&\mapsto v_s
\end{split}
\end{align}
if we choose (and we do) a representative of $s$ in $K$.


Let $P=MN,P'=MN'\in \mathcal{P}$ and $\sigma$ be an admissible representation of $M(F)$ of finite length. For $v\in\Ind_P^G\sigma$, consider the condition: for all $g\in G(F)$ and $w\in\sigma^\vee$, the integral
\begin{align}\label{J:absconv}\tag{$\clubsuit$}
    \int_{(N\cap N')(F)\backslash N'(F)} \langle v(ng), w\rangle_{\sigma} \,dn
\end{align}
is absolutely convergent. If \eqref{J:absconv} holds for all $v\in \Ind_P^G \sigma$, we obtain a $G(F)$-equivariant homomorphism
\begin{align*}
    J_{P'|P}(\sigma):\mathrm{Ind}_P^G\sigma&\longrightarrow\Ind_{P'}^G\sigma\\v&\longmapsto\left(g\mapsto  \int_{(N\cap N')(F)\backslash N'(F)} v(ng)\right).
\end{align*}
In this case we say that the intertwining operator $J_{P'|P}(\sigma)$ is defined via absolutely convergent integrals.

\begin{lemma}\label{lem:J} There exists $R\in\RR$ such that for $\chi\in\Lambda_M$ with $\langle\mathrm{Re}(\chi),\alpha\rangle>R$ for all $\alpha\in\Sigma(P)\cap\Sigma(\overline{P'})$, the intertwining operator $J_{P'\mid P}(\sigma_\chi)$ is defined via absolutely convergent integrals. If $\sigma$ is tempered, then one can take $R=0$.

Furthermore, the function $\chi\mapsto J_{P'\mid P}(\sigma_\chi)$ extends to a rational function on $\Lambda_M$. Moreover, there are finitely many $R_i\in \CC$ and $\alpha_i\in \Sigma(P)\cap \Sigma(\overline{P'})$ such that for any $(v,w)\in \Ind_{K_P}^K\sigma|_{K_M}\times \Ind_{K_{P'}}^K\sigma^\lor|_{K_M},$ poles of $\langle J_{P'|P}(\sigma_\chi)v_\chi,w\rangle,$ as functions on $X^\ast(M)_\CC$ modulo $\tfrac{2\pi i}{\log q}X^\ast(M)$, are contained in the hyperplanes $\langle \chi,\alpha_i^\lor\rangle=R_i.$
\end{lemma}
\begin{proof} 
This follows from \cite[Th\'eor\`eme IV.1.1 and Proposition IV.2.1]{Waldspurger:Plancherel}.     
\end{proof}

Let $M\in \mathcal{M}.$ Let $s\in W(M,M)$ and $\sigma\in\Pi_2(M)$. For $P,P'\in \mathcal{P}(M),$ define 
\begin{align*}
    {}^\circ c_{P'\mid P}(s,\sigma) &:= (J_{P'\mid s.P}(s.\sigma)\circ \lambda(s))\otimes (J_{s.P\mid P'}(s.(\sigma^\vee))^{-1}\circ \lambda(s))\\
    &\in \mathrm{Hom}_{G(F)\times G(F)} \left(\Ind_{K_P}^K\sigma\vert_{K_M}\otimes  \Ind_{K_P}^K \sigma^\lor|_{K_M},\Ind_{K_{P'}}^Ks.\sigma\vert_{K_{M}}\otimes\Ind_{K_{P'}}^Ks.(\sigma^\lor)\vert_{K_{M}}\right)\\
    &\cong  \mathrm{Hom}_{G(F)\times G(F)} \left(\mathrm{End}_\CC(\Ind_{K_P}^K\sigma\vert_{K_M}),\mathrm{End}_\CC(\Ind_{K_{P'}}^Ks.\sigma\vert_{K_{M}})\right).
\end{align*}
By \cite[Lemme V.3.1]{Waldspurger:Plancherel}, this is a well-defined unitary operator, and the map $\chi\mapsto {}^\circ c_{P'\mid P}(s,\sigma_\chi)$ extends to a rational function on $\Lambda_M$ that is holomorphic on $\Im\Lambda_M$. 

\begin{lemma}[{\cite[Lemme V.3.2]{Waldspurger:Plancherel}}]\label{oc:cocycle} Let $M\in \mathcal{M}$ and $P,P',P''\in \mathcal{P}(M)$. For $s,t\in W(M,M)$ and $\sigma\in \Pi_2(M)$, one has
\begin{align*}
    ^\circ c_{P''\mid P'}(s,t.\sigma)\circ{}^\circ c_{P'\mid P}(t,\sigma) = {}^\circ c_{P''\mid P}(st,\sigma).
\end{align*}\qed
\end{lemma}

\subsection{Functions on groups} 
For a function $f:G(F)\to \CC,$ let $f^\lor(g):=f(g^{-1}).$ Let
\begin{center}
    \begin{tikzcd}[row sep=tiny]
        R:G(F)\times G(F)\times \Hom_{\textbf{Set}}(G(F),\mathbb{C})\arrow[r]& \Hom_{\textbf{Set}}(G(F),\mathbb{C})\\
        (g,h,f)\arrow[r,maps to]&{R(g,h)f:x\mapsto f(g^{-1}xh)}
    \end{tikzcd}
\end{center}
denote the regular $G(F)\times G(F)$-action on functions on $G(F)$. If $V\subseteq \Hom_{\textbf{Set}}(G(F),\mathbb{C})$ is an invariant subspace, we continue to use $R$ to denote the restriction of the action to $V$. If a function on $G(F)$ is bi-$K_0$-invariant for some compact open subgroup $K_0$ of $K$, we sometimes identify it as a function on $G(F)/\!/K_0:=K_0\backslash G(F)/K_0$. We use the same notation for function spaces, e.g., $C^\infty_c(G(F))^{K_0\times K_0}=C^\infty_c(G(F)/\!/K_0).$

Choose a closed $F$-embedding $\iota:G\to \GL_n$ and let $\norm{\cdot}_{M_n(F)}$ be the box norm on $M_n(F)$. For $g\in G(F),$ set
\begin{align*}
    \norm{g}&:= \max\big(\norm{\iota(g)}_{M_n(F)},\norm{\iota(g^{-1})}_{M_n(F)}\big),\\
    \sigma(g)&:= 1+\log \norm{g}.
\end{align*}
We always choose $\iota$ so that $\norm{\cdot}$ is bi-$K$-invariant. Note that both $\norm{\cdot}$ and $\sigma$ are submultiplicative, i.e.,
\begin{align*}
    \norm{hg}\ll \norm{h}\norm{g},\quad \sigma(hg)\ll \sigma(h)\sigma(g)
\end{align*}
for all $h,g\in G(F)$.

Let $v_0\in \Ind_{K_{P_0}}^K1$ be the unique element such that $v_0\vert_{K}\equiv 1$. For $\chi\in \Lambda_M$ and $g\in G(F)$, set
\begin{align}\label{defn:zonal:spherical}
    \Xi_{G,\chi}(g) :=\Xi_{\chi}(g):= \langle(\Ind_{P_0}^G \chi)(g)v_0,v_0\rangle.
\end{align}
This is the zonal spherical function for $\Ind_{P_0}^G \chi$. When $\chi=1$, we write $\Xi:=\Xi_G := \Xi_{G,1}$ for short; this is the Harish-Chandra $\Xi$-function. For $f\in C(G(F))$ and $d\in\mathbb{Z}_{\geq 0}$, set
\begin{align*}
    p_d(f) := \sup_{g\in G(F)} |f(g)|\Xi_G(g)^{-1}\sigma(g)^d.
\end{align*}
Let $C_u^\infty(G(F))\subseteq C^\infty(G(F))$ denote the subspace of $K\times K$-finite functions. The Harish-Chandra Schwartz space is defined as
\begin{align*}
    \mathcal{C}(G(F)):= \{f\in C_u^\infty(G(F)): p_d(f)<\infty \text{ for all }d\in\mathbb{Z}_{\geq 0}\}.
\end{align*}
For a compact open subgroup $K_0\le K,$ the seminorms $\{p_d\}_{d\ge 0}$ endow $\mathcal{C}(G(F)/\!/K_0)$ with a Fr\'echet space topology. Thus
$$\mathcal{C}(G(F))=\bigcup_{K_0\le K}\mathcal{C}(G(F)/\!/K_0)$$
is an LF-space. Under this topology, $C^\infty_c(G(F))$ is a dense subset of $\mathcal{C}(G(F)).$

\begin{lemma}\label{lem:CGconverge} Let $f\in \mathcal{C}(G(F))$ and let $\pi\in \Temp(G)\cup\Temp_\Ind(G)$. For $(v,w)\in \pi\times\pi^\vee$, one has $g\mapsto f(g)\langle \pi(g)v,w\rangle_\pi\in L^1(G(F))$. In particular, there exists a unique operator $\pi(f)\in\mathrm{End}_\mathbb{C}(\pi)$ such that
\begin{align*}
    \langle \pi(f)v,w\rangle_\pi = \int_{G(F)} f(g)\langle \pi(g)v,w\rangle_\pi dg
\end{align*}
for all $(v,w)\in \pi\times \pi^\vee$.
\end{lemma}
\begin{proof} See \cite[III.7, p.~273]{Waldspurger:Plancherel}.    
\end{proof}

\subsection{Harish-Chandra Plancherel theorem}\label{ssec:HCplancherel}

For $P=MN\in\mathcal{P}$ and an $\Im\Lambda_M$-orbit $\mathcal{O}\subseteq\Pi_2(M)$, let $C^\infty(\mathcal{O},P)$ be defined as in \cite[p.296]{Waldspurger:Plancherel}. Explicitly, choose a base point $\sigma\in \mathcal{O},$ and let $U_\sigma\subseteq \mathrm{Im}\, \Lambda_M$ consist of $\chi$ such that $\sigma\cong \sigma_\chi$ as $M(F)$-representations. Note that $U_\sigma$ is a finite group depending only on $\mathcal{O}$. Let $\chi\in U_\sigma.$ An isomorphism $\varphi:\sigma\cong \sigma_\chi$ induces an isomorphism 
\begin{align*}
    \mathrm{Ind}_P^G\sigma\otimes\Ind_P^{G} \sigma^\lor\cong \mathrm{Ind}_P^G \sigma_\chi\otimes \Ind_{P}^G(\sigma_\chi)^\lor,
\end{align*} 
that is independent of the choice of $\varphi.$ This isomorphism transports to an automorphism of $\Ind_{K_P}^K \sigma|_{K_M}\otimes \Ind_{K_P}^K \sigma^\lor|_{K_M},$ and hence an automorphism of $\mathrm{End}_\CC(\Ind_{K_P}^K \sigma|_{K_M}).$ Thus, we have a group action $U_\sigma$ on $\mathrm{End}_\CC(\Ind_{K_P}^K \sigma|_{K_M}).$ On the other hand, $U_\sigma$ acts naturally on $C^\infty(\Im\Lambda_M)$ by translation. Define
\begin{align*}
    C^\infty(\mathcal{O},P) := (C^\infty(\Im\Lambda_M)\otimes \mathrm{End}_\mathbb{C}(\Ind_{K_P}^K\sigma\vert_{K_M}))^{U_\sigma}
\end{align*}
to be the subspace invariant under the diagonal action of $U_\sigma$. We equip $C^\infty(\mathcal{O},P)$ with the colimit topology (c.f. \cite[VI.3, p.301]{Waldspurger:Plancherel}). Let
\begin{align*}
    \Theta = \Theta_G := \{(\mathcal{O},P) : P=MN\in\mathcal{P},\,\mathcal{O}\in \pi_0(\Pi_2(M))\}
\end{align*}
and
\begin{align*}
    C^\infty(\Theta) := \bigoplus_{(\mathcal{O},P)\in\Theta} C^\infty(\mathcal{O},P).
\end{align*}
For $T\in C^\infty(\Theta)$, $P=MN\in\mathcal{P}$ and $\sigma\in \Pi_2(M)$, we write $T(\sigma,P)\in\Ind_{K_P}^K\sigma\vert_{K_M}$ for its evaluation at the point $(\sigma,P)$. Set
\begin{align*}
    C^\infty(\Theta)^{\mathrm{inv}} := \left\{T\in C^\infty(\Theta) : T(s.\sigma,P') = {}^\circ c_{P'\mid P}(s,\sigma) T(\sigma,P)\text{ for all }\begin{array}{l}
         P\in\mathcal{P},\,\sigma\in\Pi_2(M),\\
         s\in W(G,A_0),\,P'\in\mathcal{P}(s.M) 
    \end{array}\right\}.
\end{align*}

Let $P=MN, P'=M'N'\in \mathcal{P}.$ Let $\sigma\in \Pi_2(M)$ and $\sigma'\in \Pi_2(M').$ By \cite[Th\'eor\`eme VII.2.6]{Representation-p-adic} and  \hyperref[oc:cocycle]{Lemma \ref{oc:cocycle}}, $\Ind_P^G\sigma\cong \Ind_{P'}^G{\sigma'}$ if and only if $s.M = M'$ and $s.\sigma \cong \sigma'$ as $M'(F)$-representations for some $s\in W(G,A_0)$. In this case, we have a natural $G(F)\times G(F)$-equivariant isomorphism
\begin{align*}
    \mathrm{End}_\CC(\Ind_P^G\sigma)&\longrightarrow \mathrm{End}_\CC(\Ind_{P'}^G\sigma')\\
    T&\longmapsto  {}^\circ c_{P'\mid P}(s,\sigma)T.
\end{align*}
These maps form an inverse system by \hyperref[oc:cocycle]{Lemma \ref{oc:cocycle}}. Let $\mathrm{End}$ denote the LF bundle over $\Temp_\Ind(G)$ whose fibre over $\pi$ is given by 
\begin{align*}
    \mathrm{End}_\pi := \varprojlim_{(\sigma,P):\pi\cong \Ind_P^G\sigma}\,\mathrm{End}_\CC(\Ind_P^G\sigma).
\end{align*}
The bundle $\mathrm{End}\to\Temp_\Ind(G)$ is trivial over each connected component of $\Temp_\Ind(G)$, as one has
\begin{align*}
    \mathrm{End}_\pi \cong \varprojlim_{(\sigma,P):\pi\cong \Ind_P^G\sigma}\,\mathrm{End}_\CC(\Ind_{K_P}^K\sigma\vert_{K_M}).
\end{align*}
Put
\begin{align*}
    \mathcal{C}(\Temp_\Ind(G)) := C^\infty(\widetilde{\Temp}_{\Ind}(G),\mathrm{End}) \cap \Gamma_c(\Temp_\Ind(G),\mathrm{End}).
\end{align*}
That is, $\mathcal{C}(\Temp_\Ind(G))$ consists of smooth functions  $\widetilde{\Temp}_{\Ind}(G)\to \mathrm{End}$ that descend to compactly supported sections of $\mathrm{End}\to \Temp_\Ind(G)$. By the universal property of limits, one has a canonical isomorphism
\begin{align*}
    \mathcal{C}(\Temp_\Ind(G)) \cong C^\infty(\Theta)^{\mathrm{inv}}.
\end{align*}

Equip $C^\infty(\Theta)$ with the inductive limit topology. Endow $\mathcal{C}(\Temp_\Ind(G))$ with the topology transferred from the subspace topology of $C^\infty(\Theta)^{\mathrm{inv}}$. The space $\mathcal{C}(\Temp_\Ind(G))$ acquires a natural $G(F)\times G(F)$-action: for $g,h\in G(F)$ and $T\in \mathcal{C}(\Temp_\Ind(G)),$ define the section $g.T.h$ by
\begin{align*}
    (g.T.h)(\pi):=\pi(g)\circ T(\pi)\circ \pi(h).
\end{align*}

\begin{theorem}[\cite{Waldspurger:Plancherel}]\label{HCPlan} The map 
\begin{center}
    \begin{tikzcd}[row sep=tiny]
        \mathrm{HP} = \HP_G:\mathcal{C}(G(F))\arrow[r]& \mathcal{C}(\Temp_{\Ind}(G))\\
f \arrow[r,maps to]&{\big(\pi\mapsto \pi(f)\big)}
    \end{tikzcd}
\end{center}
is a $G(F)\times G(F)$-equivariant topological $\mathbb{C}$-algebra isomorphism. Its inverse $ \mathrm{HP}^{-1}$ is given by 
\begin{align*}
    \mathcal{C}(\Temp_{\Ind}(G))\ni T\longmapsto \mathrm{HP}^{-1}(T)(g) := \int_{\Temp_{\Ind}(G)} \tr(\pi(g^{-1})\circ T(\pi)) d\pi.
\end{align*}
\qed
\end{theorem}

For a compact open subgroup $K_0\le K$, let 
\begin{align}\label{thetaK0}
    \Theta^{K_0} := \{(M,\sigma)\in\widetilde{\Temp}_\Ind(G): (\Ind_P^G\sigma)^{K_0}\neq 0\}.
\end{align}
Note that $\Theta^{K_0}$ consists of finitely many connected components. Then $\HP_G$ induces an isomorphism  
\begin{align*}\mathcal{C}(G(F)/\!/K_0) \cong \mathcal{C}(\Temp_{\Ind}(G))^{K_0\times K_0}=\mathcal{C}(\Temp_{\Ind}(G))\cap C^\infty(\Theta^{K_0},\mathrm{End}).
\end{align*}

\subsection{Constant terms}\label{ssec:constant} For $P=MN\in \mathcal{P}$, $f\in \mathcal{C}(G(F))$ and $m\in M(F)$ the integral
\begin{align*}
    f^{(P)}(m) := \delta_P^{\frac{1}{2}}(m) \int_{N(F)} f(mn) dn 
\end{align*}
is absolutely convergent. It defines a $K_M\times K_M$-equivariant  homomorphism 
\begin{align*}
    (-)^{(P)}:\mathcal{C}(G(F))\to \mathcal{C}(M(F)).
\end{align*}
On the other hand, there is a map $(-)^{(P)}:\mathcal{C}(\Temp_\Ind(G))\longrightarrow \mathcal{C}(\Temp_\Ind(M))$ defined as follows. For $P'=M'N'\in \mathcal{P}$ with $M'\le M$ and $\sigma\in\Pi_2(M')$, there is a canonical isomorphism
\begin{align*}
    \mathrm{End}_\mathbb{C}(\Ind_{K_{P'}}^K\sigma\vert_{K_{M'}}) \cong \mathrm{End}_{\mathbb{C}}(\Ind_{K_{P}}^K\Ind^{K_{M}}_{K_{P'\cap M}}\sigma\vert_{K_{M'}}).
\end{align*}
Evaluating at the identity of $K\times K$ gives a homomorphism
\begin{align*}
    \mathrm{End}_{\mathbb{C}}(\Ind_{K_{P}}^K\Ind^{K_{M}}_{K_{P'\cap M}}\sigma\vert_{K_{M'}}) &\cong \Ind_{K_{P}\times K_{P}}^{K\times K}\Ind^{K_{M}\times K_{M}}_{K_{P'\cap M}\times K_{P'\cap M}}\sigma\otimes\sigma^\vee\vert_{K_{M'}\times K_{M'}} \\
    &\to \Ind^{K_{M}\times K_{M}}_{K_{P'\cap M}\times K_{P'\cap M}}\sigma\otimes\sigma^\vee\vert_{K_{M'}\times K_{M'}} \cong \mathrm{End}_{\mathbb{C}}(\Ind^{K_{M}}_{K_{P'\cap M}}\sigma\vert_{K_{M'}}).
\end{align*}
Composing the two maps gives $\mathrm{End}_\mathbb{C}(\Ind_{K_{P'}}^K\sigma\vert_{K_{M'}}) \longrightarrow \mathrm{End}_{\mathbb{C}}(\Ind^{K_{M}}_{K_{P'\cap M}}\sigma\vert_{K_{M'}})$, which globalizes to a map
\begin{align*}
    (-)^{(P)}:C^\infty(\Theta_G)\longrightarrow C^\infty(\Theta_M)
\end{align*}
that descends to
\begin{align*}
    (-)^{(P)}:\mathcal{C}(\Temp_\Ind(G))\longrightarrow \mathcal{C}(\Temp_\Ind(M))
\end{align*}
by {\cite[Lemma 2.14]{DRS}}. Furthermore, we have a commutative diagram
\begin{equation}\label{CT:spectral}
    \begin{tikzcd}
        \mathcal{C}(G(F))\arrow[d,swap,"(-)^{(P)}"]\arrow[r,"\HP_G"]&\mathcal{C}(\Temp_\Ind(G))\arrow[d,"(-)^{(P)}"]\\
        \mathcal{C}(M(F))\arrow[r,"\HP_M"]&\mathcal{C}(\Temp_\Ind(M)).
    \end{tikzcd}
\end{equation}

\subsection{Working assumptions} \label{ssec:assumptions}
Fix a separable closure $F^{\mathrm{sep}}$ of $F$, and let $\Gal_F:=\Gal(F^{\mathrm{sep}}/F)$ denote the absolute Galois group. Let $W_F$ be the Weil group of $F$. Let ${}^LG:=G^\vee(\mathbb{C})\rtimes W_F$ be the $L$-group of $G$. 

Let $\mathrm{Temp}(G)_\CC:=\{\pi_\chi: \pi\in\Temp(G),\,\chi\in\Lambda_G\}.$ We pullback the action of $\Lambda_G$ on $\mathrm{Temp}(G)_\CC$ along \eqref{XG->LambdaG:surj} to get an $X^*(G)_\mathbb{C}$-action.  Let $\Phi(G)$ (resp. $\Phi_t(G)$) denote the set of equivalence classes of (resp. tempered) $L$-parameters of $G$. The canonical injection 
\begin{align}\label{X*G:pair:ZGvee}
    X^*(G) = X^*(G/G^{\mathrm{der}}) = X_*((G/G^{\mathrm{der}})^\vee)^{\Gal_F} = X_*(Z_{G^\vee}^\circ)^{\Gal_F} \subseteq \Hom_{\textbf{TopGp}}(\CC^\times,Z_{G^\vee}^\circ(\CC)^{\Gal_F})
\end{align}
provides an action of $X^*(G)_\mathbb{C}$ on $\Phi(G)$, and $\Phi_t(G)$ is stable by $X^*(G)\otimes_\mathbb{Z}i\mathbb{R}$.

A local Langlands correspondence of $G$ (for non-Archimedean $F$) is a map 
\begin{align*}
    \mathrm{LL}_G:\mathrm{Temp}(G)_\CC\to \Phi(G)
\end{align*}
satisfying a list of properties. Following \cite{DRS}, we require:
\begin{enumerate}[label=(LLC\text{\arabic*}), ref=LLC\text{\arabic*}]
    \item\label{LLC:LCFT} If $G$ is a torus, $\mathrm{LL}_G$ is given by the class field theory \cite{Langlands:tori}.
    \item\label{LLC:GL} $\mathrm{LL}_{\GL_n}$ is given by Harris-Taylor-Henniart \cite{Harris-Taylor,Henniart} and Laumon-Rapoport-Stuhler \cite{Laumon:Rapoport:Stuhler:LLGLpositive}.
    \item\label{LLC:temp} One has $\mathrm{LL}_G^{-1}(\Phi_t(G)) = \Temp(G)$.
    \item\label{LLC:twist} $\mathrm{LL}_G$ is $X^*(G)_\mathbb{C}$-equivariant.
    \item\label{LLC:dual} If $\rho:{}^LG\longrightarrow\GL_{V_\rho}(\mathbb{C})$ is any representation, then $\rho\circ \mathrm{LL}_G(\pi^\vee)\cong (\rho\circ\mathrm{LL}_G(\pi))^\vee$. 
    \item\label{LLC:para} Let $M\in\mathcal{M}$, $\sigma\in\Pi_2(M)$ and $\pi$ an irreducible subquotient of $\Ind_P^G\sigma$. Then $$\mathrm{LL}_G(\pi) = ({}^LM\longrightarrow {}^LG)\circ \mathrm{LL}_M(\sigma).$$
\end{enumerate}
We assume the existence of maps $\mathrm{LL}_M$ with such properties for all (semi-standard) Levi subgroups $M$ of $G$ in the paper.

Let $\rho:{}^LG\longrightarrow \GL_{V_\rho}(\mathbb{C})$ be a tempered representation, i.e., $\rho(W_F)$ has bounded image. We make the following assumption on $(G,\rho)$.
\begin{enumerate}[label=(G), ref=G]
    \item \label{G}
    There is $\nu=\nu_G\in X^\ast(G)$ such that the dual cocharacter $\nu^\lor:\CC^\times \longrightarrow Z_{G^\vee}(\CC)$ satisfies $\rho\circ \nu^\lor(z)=z$.
\end{enumerate}

\begin{remark}\label{rem:relaxation}
    One only needs the assumption that $\rho\circ \nu^\lor(z)=z^k$ for some $k\in \ZZ_{>0}.$ However, to ease the notation we will work with the normalized choice of $\nu$.
\end{remark}

    Decompose  $\rho\vert_{Z_{G^\vee}^\circ} =\mu_1\oplus\cdots\oplus\mu_n$ as a sum of (rational) characters of $Z_{G^\vee}^\circ$. Let 
\begin{align*}
\sigma_{\rho,G} := \sum_{i=1}^n \mathbb{Z}_{\geq 0}\mu_i\subseteq X^*(Z_{G^\vee}^\circ).
\end{align*}
Then \eqref{G} implies that $\mathbb{R}_{\geq 0}\sigma_{\rho,G}$ is a strongly convex rational cone in $X^*(Z_{G^\vee}^\circ)_{\mathbb{R}}$.

Let $M\in \mathcal{M}.$ Let $\rho_M := \rho\vert_{{}^LM}$ and let $\nu_M:=\nu\vert_{M}$. Then $\rho_M$ is a tempered representation and $\rho_M\circ \nu_M^\lor(z)=z$. Therefore, each $(M,\nu_M,\rho_M)$ satisfies the assumption \eqref{G}. We will often identify $\nu_M$ with $\nu$. We define $\sigma_{\rho,M}:=\sigma_{\rho_M,M}$ in the same way, so $\mathbb{R}_{\geq 0}\sigma_{\rho,M}$ is a strongly convex rational cone in $X^*(Z_{M^\vee}^\circ)_{\mathbb{R}}$.

\begin{definition}\label{def:positivecone}
    A subset $C$ of $\Re\Lambda_M$ is called a \textbf{positive cone} if the following holds:
\begin{enumerate}
    \item [(C1)] There are $\chi_0\in \Re\Lambda_M$ and a (nonempty) open cone $C'$ in $\Re\Lambda_M$ such that $C=\chi_0C'$.
    \item [(C2)] There exists $s_0\in\mathbb{R}$ such that $|\nu|^{s}\in C$ for $s>s_0$.
\end{enumerate}

\end{definition}

\begin{lemma}\label{intertwining:posCone}
   Let $P=MN\in \mathcal{P}$. For $R\in\mathbb{R}$ and $C$ a positive cone in $\Re\Lambda_M$, the intersection
   \begin{align*}
       \left\{\chi\in\Re\Lambda_M: \langle\chi,\alpha^\vee\rangle > R\text{ for all }\alpha\in \Sigma(P)\right\} \cap C
   \end{align*}
   is a nonempty open set.
\end{lemma}
\begin{proof} It suffices to show the intersection is nonempty. Let $\chi\in\Re\Lambda_M$ with $\langle\chi,\alpha^\vee\rangle>R$ for all $\alpha\in \Sigma(P)$. Since $C$ is a positive cone, we can find $r>0$ so that $\chi|\nu|^r\in C$. As $\nu\in X^*(G)$, it follows that $\langle \nu,\alpha^\vee\rangle = 0$ for all $\alpha\in \Sigma(P)$. Hence $\chi|\nu|^r$ lies in the intersection.
\end{proof}

\subsection{Some remarks on $L$-functions}\label{ssec:LLC}
Let $(G,\nu,\rho)$ satisfy the assumption \eqref{G}. Since $\rho$ is tempered, by \eqref{LLC:temp} we have an induced map
\begin{align*}
    \mathrm{Temp}(G)\xrightarrow{\quad\mathrm{LL}_G\quad}\Phi_t(G)\xrightarrow{\quad\rho\circ-\quad} \Phi_t(\GL_{V_\rho}) = \Phi_t(\GL_{\dim V_\rho}).
\end{align*}
For $\pi\in \Temp(G)$, put
\begin{align}\label{eq:Lgamma:LL}
\begin{split}
    L(s,\pi,\rho)&:=L(s,\rho\circ \mathrm{LL}_G(\pi)),\\
    \gamma(s,\pi,\rho,\psi)&:=\gamma(s,\rho\circ \mathrm{LL}_G(\pi),\psi).
\end{split}
\end{align}
It is well-known that $L(s,\pi,\rho)$ is holomorphic for $\Re(s)>0.$ Since the quotient $X_*(Z_{G^\vee}^\circ)/X_*(Z_{G^\vee}^\circ)^{\Gal_F}$ is torsion-free, the map \eqref{X*G:pair:ZGvee} induces a surjection
\begin{align}\label{X*ZGvee:as:HomX*GZ}
    X^*(Z^\circ_{G^\vee})\twoheadrightarrow \Hom_{\mathbb{Z}}(X_*(Z_{G^\vee}^\circ)^{\Gal_F},\mathbb{Z})=\Hom_{\mathbb{Z}}(X^*(G),\mathbb{Z}).
\end{align}
 Let
\begin{align*}
    \widetilde{\sigma}_{\rho,G} \subseteq \Hom_{\mathbb{Z}}(X^*(G),\mathbb{Z})
\end{align*}
be the image of $\sigma_{\rho,G}$ under \eqref{X*ZGvee:as:HomX*GZ}. As $\sigma_{\rho,G}$ is $\Gal_F$-invariant,  $\mathbb{R}_{\geq 0}\widetilde{\sigma}_{\rho,G}$ is a strongly convex rational cone in $\Hom_{\mathbb{Z}}(X^*(G),\mathbb{Z})_\RR$ by \cite[Lemma 6.11]{DRS}.

 Let $M\in \mathcal{M}$. Write $\rho_{M} = \oplus \rho_j$ where each $\rho_j$ is an irreducible representation of $^LM$. Further write $\rho_j\vert_{Z_{M^\vee}^\circ} = \mu_{j1}\oplus\cdots\oplus\mu_{jn_j}$ where the $\mu_{ji}\in X^*(Z^\circ_{M^\vee})$ are characters. Since $\rho_j$ is $^LM$-irreducible and $(Z^\circ_{M^\vee})^{\Gal_F}\le  Z_{{}^LM},$ all $\mu_{ji}\,(i\in[n_j])$ map to the same element $\widetilde{\mu}_j\in\Hom_\mathbb{Z}(X^*(M),\mathbb{Z})$ under \eqref{X*ZGvee:as:HomX*GZ}. Note that $\{\widetilde{\mu}_j\}_{j}$ generates the cone $\widetilde{\sigma}_{\rho,M}$. By \eqref{LLC:GL} and \eqref{LLC:twist}, for $\chi\in X^*(M)_{\mathbb{C}}$ and $\sigma\in \Pi_2(M)$ one has
\begin{align*}
L(0,\sigma_\chi,\rho_j) = L(\langle\chi,\widetilde{\mu}_j\rangle,\sigma,\rho_j).
\end{align*}
Then 
\begin{align}\label{Lfactor:factorize}
    L(0,\sigma_\chi,\rho_{M}) = \prod_{j}L( \langle\chi,\widetilde{\mu}_j\rangle,\sigma,\rho_j).
\end{align}

Let
\begin{align}\label{CrhoM:dualcone:defn}
    \overline{C}_{\rho,M} := \left\{\chi\in \Re \Lambda_M: \langle\chi,\widetilde{\mu}_j\rangle \ge  0 \textrm{ for all } j\right\} = \left(\RR_{\ge 0}\widetilde{\sigma}_{\rho,M}\right)^\vee
\end{align} 
be the dual cone in $\Re\Lambda_M\cong X^\ast(M)_\RR.$ Let $C_{\rho,M}:=\mathrm{int}(\overline{C}_{\rho,M})$ be its interior. It is an open cone. Since $\langle \nu,\widetilde{\mu}_j\rangle=1$ for all $j$, each $C_{\rho,M}$ is a positive cone. Note that the meromorphic function $\chi\mapsto L(0,\sigma_{\chi}, \rho_{M})$ is holomorphic on $C_{\rho,M}\times \Im \Lambda_M$ for any $\sigma\in \Pi_2(M)$. 

For semi-standard Levi subgroups $M'\leq M$, we have $\iota_{M\ge M'}(C_{\rho,M})\subseteq C_{\rho,M'}$ and $p_{M'\le M}(C_{\rho,M'})=C_{\rho,M}.$ For $\chi\in X^*(M')_\RR$ and $\beta\in\Hom_\ZZ(X^*(M),\ZZ)_\RR$, we write
\begin{align}\label{pairing:bw:Levis}
    \langle \chi,\beta\rangle := \langle \chi|_{M},\beta\rangle.
\end{align}

\begin{lemma}\label{Lcone:dualdual} Let $\mathcal{B}_0\subseteq X^*(M_0)$ be a finite generating set of the cone $\overline{C}_{\rho,M_0}$. For any $M\in \mathcal{M}$,
\begin{align*}
    \widetilde{\sigma}_{\rho,M} \subseteq \{\beta\in\Hom_\ZZ(X^*(M),\ZZ): \langle\tau,\beta\rangle\geq 0\text{ for all }\tau\in\mathcal{B}_0\}.
\end{align*}
Here $\langle\tau,\beta\rangle$ is interpreted as in \eqref{pairing:bw:Levis}.
\end{lemma}
\begin{proof} The image $p_{M_0\leq M}(\mathcal{B}_0)\subseteq X^*(M)_\RR$ is a generating set of the cone $\overline{C}_{\rho,M}$. Hence
\begin{align*}
    \{\beta\in\Hom_\ZZ(X^*(M),\ZZ)_\RR: \langle\tau,\beta\rangle\geq 0\text{ for all }\tau\in\mathcal{B}_0\} = (\overline{C}_{\rho,M})^\vee = \mathbb{R}_{\geq 0}\widetilde{\sigma}_{\rho,M}.
\end{align*}
For the last equality, see \cite[p.9 (1)]{Fulton:toric} for example. 
\end{proof}

Recall for $M\in \mathcal{M},$ $\Omega_M=M(F)/M(F)^1$ is identified with the image of $H_M$ in $\Hom_\ZZ(X^\ast(M),\ZZ).$

\begin{lemma}\label{Lfunction:expand:cone} Let $M\in \mathcal{M}$ and $\sigma\in \Pi_2(M)$. For $\chi\in X^*(M)_\CC$ one has
\begin{align*}
    L(0,\sigma_\chi,\rho_M) = \sum_{\beta\in\widetilde{\sigma}_{\rho,M}\cap \Omega_M}c_\beta(\sigma_\chi) = \sum_{\beta\in\widetilde{\sigma}_{\rho,M}\cap \Omega_M}c_\beta(\sigma) q^{-\langle\chi,\beta\rangle},
\end{align*}
and the sum is absolutely convergent for $\chi\in C_{\rho,M}$. Moreover, 
\begin{align*}
    L(0,\sigma_\chi,\rho_M) = \sum_{\alpha\in\widetilde{\sigma}_{\rho,G}\cap \Omega_G}\left(\sum_{\substack{\beta\in\widetilde{\sigma}_{\rho,M}\cap \Omega_M\\\beta\vert_{X^*(G)} = \alpha}}c_{\beta}(\sigma)q^{-\langle \chi,\beta-\alpha\rangle}\right)q^{-\langle\chi,\alpha\rangle},
\end{align*}
and the inner sum is a finite sum for each $\alpha\in\widetilde{\sigma}_{\rho,G}\cap \Omega_G$.
\end{lemma}

\begin{proof} Decompose $\rho_M = \bigoplus\limits_{j=1}^n\rho_j$ into irreducible $^LM$-representations, and let $\widetilde{\mu}_j$ be the common image of characters of $\rho_j|_{Z_{M^\lor}^\circ}$ in $\Hom_\ZZ(X^*(M),\ZZ)$. For $\chi\in X^*(M)_\CC$ recall from \eqref{Lfactor:factorize} that
\begin{align*}
    L(0,\sigma_\chi,\rho_M) = \prod_{j=1}^n L(0,\sigma_\chi,\rho_j) = \prod_{j=1}^n L(\langle\chi,\widetilde{\mu}_j\rangle,\sigma,\rho_j).
\end{align*}
 Write
\begin{align*}
    \rho_j\circ \mathrm{LL}_M(\sigma) = \bigoplus_{k=0}^R \phi_{j,k}(\sigma) \otimes\mathrm{Sym}^k 
\end{align*}
as representations of $W_F\times\mathrm{SL}_2(\CC)$. By allowing $\phi_{j,k}(\sigma)$ to be zero, we can choose $R$ to be uniform for $j\in[n]$. Let $I_F< W_F$ be the inertia subgroup and $\mathrm{Frob}\in W_F$ be any lift of the Frobenius in $\mathrm{Gal}_{\mathbb{F}_q}.$ Let $V_{j,k}$ be the underlying vector space of $\phi_{j,k}(\sigma)$. Then
\begin{align}\label{Lfunction:expand:cone:1}
    \nonumber L(s,\sigma,\rho_j) = \prod_{k=0}^R L(s,\phi_{j,k}(\sigma) \otimes\mathrm{Sym}^k) &= \prod_{k=0}^R \det\left(I - q^{-\frac{k}{2}-s}\phi_{j,k}(\sigma)(\mathrm{Frob})\big|_{V_{j,k}^{I_F}}\right)^{-1}\\
    &=\prod_{k=0}^R \sum_{\ell=0}^\infty \mathrm{Trace}\left(\mathrm{Sym}^\ell\left(\phi_{j,k}(\sigma)(\mathrm{Frob})\big|_{V_{j,k}^{I_F}}\right)\right)q^{-\ell\frac{k}{2}-\ell s}.
\end{align}
By assumption $\rho_j\circ \mathrm{LL}_M(\sigma)$ has bounded image when restricted to $W_F$, so the eigenvalues of $\phi_{j,k}(\sigma)(\mathrm{Frob})\big|_{V_{j,k}^{I_F}}$ all have norm $1$. Thus the expression \eqref{Lfunction:expand:cone:1} is absolutely convergent for $\mathrm{Re}(s)>0$. Also, we have
\begin{align}\label{eq:Lexpansion}
   \nonumber  L(0,\sigma_\chi,\rho_M) &= \prod_{j=1}^n\prod_{k=0}^R\sum_{\ell=0}^\infty  \mathrm{Trace}\left(\mathrm{Sym}^\ell\left(\phi_{j,k}(\sigma)(\mathrm{Frob})\big|_{V_{j,k}^{I_F}}\right)\right)q^{-\ell\frac{k}{2}-\ell \langle\chi,\widetilde{\mu}_j\rangle}\\
    \nonumber &=\sum_{(\ell_1,\ldots,\ell_n)\in\mathbb{Z}_{\geq 0}^n}\prod_{j=1}^n\prod_{k=0}^R  \mathrm{Trace}\left(\mathrm{Sym}^{\ell_j}\left(\phi_{j,k}(\sigma)(\mathrm{Frob})\big|_{V_{j,k}^{I_F}}\right)\right)q^{-\ell_j\frac{k}{2}-\ell_j \langle\chi,\widetilde{\mu}_j\rangle}\\
    \nonumber &=\sum_{(\ell_1,\ldots,\ell_n)\in\mathbb{Z}_{\geq 0}^n}\left(\prod_{j=1}^n\prod_{k=0}^R \mathrm{Trace}\left(\mathrm{Sym}^{\ell_j}\left(\phi_{j,k}(\sigma)(\mathrm{Frob})\big|_{V_{j,k}^{I_F}}\right)\right)q^{-\ell_j\frac{k}{2}}\right)q^{- \langle\chi,\sum_j\ell_j\widetilde{\mu}_j\rangle}\\
    &=\sum_{\beta\in \widetilde{\sigma}_{\rho,M}} \left(\sum_{\substack{(\ell_1,\ldots,\ell_n)\in\mathbb{Z}_{\geq 0}^n\\
    \sum_j\ell_j\widetilde{\mu}_j = \beta}} \prod_{j=1}^n\prod_{k=0}^R  \mathrm{Trace}\left(\mathrm{Sym}^{\ell_j}\left(\phi_{j,k}(\sigma)(\mathrm{Frob})\big|_{V_{j,k}^{I_F}}\right)\right)q^{-\ell_j\frac{k}{2}}\right) q^{-\langle\chi,\beta\rangle}.
\end{align}
The parenthetical sum is a finite sum since $\RR_{\ge 0}\widetilde{\sigma}_{\rho,M}$ is strongly convex. 

Now we show that the sum \eqref{eq:Lexpansion} over $\beta$ is supported on $\widetilde{\sigma}_{\rho,M}\cap \Omega_M$. That is, if we write \eqref{eq:Lexpansion} as
\begin{align*}
    L(0,\sigma_\chi,\rho_M) = \sum_{\beta\in\widetilde{\sigma}_{\rho,M}} c_\beta(\sigma_\chi) =\sum_{\beta\in\widetilde{\sigma}_{\rho,M}} c_\beta(\sigma)  q^{-\langle\chi,\beta\rangle},
\end{align*} 
then $c_\beta(\sigma) = 0$ for $\beta\in \widetilde{\sigma}_{\rho,M} - \Omega_M$. 
To see this, let $D$ be the kernel of the map $X^*(M)_\CC\relbar\joinrel\twoheadrightarrow \Lambda_M$ in \eqref{XG->LambdaG:surj}. For $\chi\in D,$ since $\sigma_{\chi} = \sigma,$  we have
\begin{align*}
    c_\beta(\sigma) = c_\beta(\sigma_\chi) = c_\beta(\sigma) q^{-\langle\chi,\beta\rangle}.
\end{align*}
If $c_\beta(\sigma)\neq 0$, then $q^{-\langle\chi,\beta\rangle} = 1$ for all $\chi\in D$. We must show this forces $\beta\in \Omega_M$. Clearly, $X^*(M)\otimes_\ZZ \tfrac{2\pi i}{\log q}\ZZ$ is a subgroup of $D$. Consider the short exact sequence
\begin{align*}
    1\longrightarrow \widetilde{D}:=\frac{D}{X^*(M)\otimes_\ZZ \tfrac{2\pi i}{\log q}\ZZ}\longrightarrow \Im T_M:=X^*(M)\otimes_\ZZ \frac{i\RR}{\tfrac{2\pi i}{\log q}\ZZ} \longrightarrow \Im\Lambda_M\longrightarrow 1.
\end{align*}
By Pontryagin duality, applying the functor $\Hom_{\textbf{TopGp}}(-,S^1)$ yields a short exact sequence
\begin{align*}
    1\longrightarrow \Omega_M\longrightarrow \Hom_{\textbf{TopGp}}(\Im T_M,S^1)\longrightarrow  \Hom_{\textbf{TopGp}}(\widetilde{D},S^1)\longrightarrow  1.
\end{align*}
View $\beta$ as an element in $\Hom_{\textbf{TopGp}}(\Im T_M,S^1)$ by $\chi\mapsto q^{-\langle\chi,\beta\rangle}$. By assumption we have $q^{-\langle\chi,\beta\rangle}=1$ for all $\chi\in \widetilde{D},$ so $\beta\in \Omega_M$ by exactness. This proves the first assertion.

For the second assertion, we can further write \eqref{eq:Lexpansion} as
\begin{align}\label{eq:Lexpansion:2}
    \sum_{\alpha\in\widetilde{\sigma}_{\rho,G}}\left(\sum_{\substack{\beta\in \widetilde{\sigma}_{\rho,M}\cap \Omega_M\\\beta\vert_{X^*(G)}=\alpha}} \sum_{\substack{(\ell_1,\ldots,\ell_n)\in\mathbb{Z}_{\geq 0}^n\\
    \sum_j\ell_j\widetilde{\mu}_j = \beta}} \prod_{j=1}^n\prod_{k=0}^R  \mathrm{Trace}\left(\mathrm{Sym}^{\ell_j}\left(\phi_{j,k}(\sigma)(\mathrm{Frob})\big|_{V_{j,k}^{I_F}}\right)\right)q^{-\ell_j\frac{k}{2}}q^{-\langle\chi,\beta-\alpha\rangle}\right) q^{-\langle\chi,\alpha\rangle}.
\end{align}
For $\alpha\in\widetilde{\sigma}_{\rho,G}$, since $\mathbb{R}_{\geq 0}\widetilde{\sigma}_{\rho,G}$ is strongly convex and $\widetilde{\mu}_j\vert_{X^*(G)}\in \widetilde{\sigma}_{\rho,G}$, there are only finitely many $(\ell_1,\ldots,\ell_n)\in\mathbb{Z}^n_{\geq 0}$ such that $\sum_j\ell_j\widetilde{\mu}_j\vert_{X^*(G)} = \alpha$. It follows that for a fixed $\alpha\in\widetilde{\sigma}_{\rho,G}$, the parenthetical term is nonzero only for finitely many $\beta\in\widetilde{\sigma}_{\rho,M}$. For the same reason, the sum over $\alpha$ is necessarily supported on $\widetilde{\sigma}_{\rho,G}\cap \Omega_G.$

\end{proof}

\section{Spectral construction of Schwartz spaces }\label{sec:Schwartz:spectral}

\subsection{Asymptotic Schwartz spaces}\label{ssec:ass}  Recall that we assume \eqref{LLC:LCFT}--\eqref{LLC:para} hold for all semi-standard Levi subgroups of $G$. We briefly review the Fourier theory developed in \cite[\S4 and \S5]{DRS}. As mentioned in the introduction, the results   in \cite{DRS} hold for all local fields including those of positive characteristic.

Let $f\in \mathcal{C}(G(F)).$ For $(M,\sigma)\in\widetilde{\Temp}_\Ind(G)$ and $(v,w)\in \Ind_{K_P}^K\sigma\vert_{K_M}\times \Ind_{K_P}^K\sigma^\vee\vert_{K_M},$ the zeta integral
\begin{align*}
     Z_0(\Ind_P^G\sigma,f,v,w):=\int_{G(F)} f(g)\langle (\Ind_P^G\sigma)(g)v,w \rangle dg
\end{align*}
is absolutely convergent by \hyperref[lem:CGconverge]{Lemma \ref{lem:CGconverge}}. Here the subscript $0$ in $Z_0$ is to emphasize that we are only allowing $f$ in $\mathcal{C}(G(F))$ in this notation, and that $Z_0$ is only necessarily defined by absolutely convergent integrals on imaginary axes $\Re \sigma=0$.

\begin{definition}  The \textbf{asymptotic $\rho$-Schwartz space} $\mathcal{S}_\rho^{\mathrm{as}}(G(F))\subseteq \mathcal{C}(G(F))$ is the subspace consisting of functions $f$ such that $\chi\mapsto \dfrac{\HP(f)(M,\sigma_\chi)}{L(\tfrac{1}{2},\sigma_\chi,\rho_M)}$ extends to a polynomial section in $\chi\in\Lambda_M$ for all $(M,\sigma)\in\widetilde{\Temp}_\Ind(G)$. That is, for any $(v,w)\in \Ind_{K_P}^K\sigma\vert_{K_M}\times \Ind_{K_P}^K\sigma^\vee\vert_{K_M},$
\begin{align*}
    \chi\mapsto \left\langle\frac{\HP(f)(M,\sigma_\chi)}{L(\tfrac{1}{2},\sigma_\chi,\rho_M)}v,w\right\rangle_{\Ind_P^G\sigma_\chi}\in \CC[\Lambda_M].
\end{align*}
\end{definition}

Therefore, for $f\in \mathcal{S}_\rho^{\mathrm{as}}(G(F))$, the function $\chi\mapsto Z_0(\Ind_P^G\sigma_\chi,f,v,w),$ originally only defined on $\Im\Lambda_{M}$ by absolutely convergent integrals, extends to a meromorphic function on $\Lambda_M$ by
\begin{align*}
    Z_0(\Ind_P^G\sigma_\chi,f,v,w) := L(\tfrac{1}{2},\sigma_\chi,\rho_M)\cdot\left\langle\frac{\HP(f)(M,\sigma_\chi)}{L(\tfrac{1}{2},\sigma_\chi,\rho_M)}v,w\right\rangle_{\Ind_P^G\sigma_\chi}.
\end{align*}
Then $Z_0(\Ind_P^G\sigma_\chi,f,v,w)$ is a rational function in $\chi$ which is holomorphic when $\Re\chi\in |\nu|^{-\frac{1}{2}}C_{\rho,M}$.

\begin{remark}
    If the integral $\int_{G(F)} f(g)\langle (\Ind_P^G\sigma_\chi)(g)v,w \rangle dg$ is absolutely convergent for $\chi$ in a connected open region $C$ in $\Lambda_M$ for any $v,w$, then it defines a holomorphic function on $C$ as explained in \S \ref{ssec:prep} below. In the case $\Im \Lambda_M\subset C,$ by the identity principle $Z_0(\Ind_P^G\sigma_\chi,f,v,w)$ agrees with the absolutely convergent zeta integral for $\chi\in C$. For instance, for $f\in C^\infty_c(G(F))$, $Z_0(\Ind_P^G\sigma_\chi,f,v,w)$ is defined by absolutely convergent integrals for all $\chi\in \Lambda_M.$
\end{remark}

\begin{lemma}\label{Sas:twist:HC} One has $\chi\mathcal{S}^{\mathrm{as}}_\rho(G(F))\subseteq \mathcal{C}(G(F))$ for $\chi\in \Lambda_G$ with $\Re(\chi)\in |\nu|^{-\frac{1}{2}}C_{\rho,G}$. Furthermore, for $f\in \mathcal{S}^{\mathrm{as}}_\rho(G(F))$ and $\chi\in |\nu|^{-\frac{1}{2}}C_{\rho,G}$
\begin{align*}
    Z_0(\mathrm{Ind}_P^G \sigma_{\lambda\chi},f,v,w)=Z_0(\mathrm{Ind}_P^G \sigma_\lambda,f\chi,v,w)
\end{align*}
as meromorphic functions in $\lambda\in \Lambda_M$ for any $(M,\sigma)\in\widetilde{\Temp}_\Ind(G)$ and $(v,w)\in \Ind_{K_P}^K\sigma\vert_{K_M}\times \Ind_{K_P}^K\sigma^\vee\vert_{K_M}.$
\end{lemma}

\begin{proof}
     This follows from the proof of \cite[Lemma 5.10]{DRS} and the identity principle.
\end{proof}

 \begin{lemma}\label{Sas:FT+constant} The space $\mathcal{S}_\rho^{\mathrm{as}}(G(F))$ is a $G(F)\times G(F)$-submodule of $\mathcal{C}(G(F))$. For $P=MN\in \mathcal{P}$, the constant term map restricts to a $K_M\times K_M$-equivariant map
\begin{align*}
    (-)^{(P)}:\mathcal{S}^{\mathrm{as}}_{\rho}(G(F))\longrightarrow \mathcal{S}^{\mathrm{as}}_{\rho_M}(M(F)).
\end{align*}
\qed
\end{lemma}

\begin{definition}[{\cite[\S4]{DRS}}] The \textbf{$\rho$-Fourier transform} 
\begin{align*}
    \mathcal{F}_{\rho}=\mathcal{F}_{\rho,\psi}:\mathcal{C}(G(F))\to\mathcal{C}(G(F))
\end{align*}
is defined spectrally via the \hyperref[HCPlan]{Harish-Chandra Plancherel theorem}:
\begin{align*}
    \mathcal{F}_\rho f (g) := \int_{\Temp_\Ind(G)} \gamma(\tfrac{1}{2},\pi,\rho,\psi)\mathrm{Trace}(\pi(g)\circ \pi(f)) d\pi.
\end{align*}
    \end{definition}

The following results can be found in \cite[\S4 and \S5]{DRS}.

\begin{theorem}\label{DRS:HC:FT}
    The Fourier transform $\mathcal{F}_{\rho}=\mathcal{F}_{\rho,\psi}$ enjoys the following properties.
    \begin{enumerate}[label=(F\arabic*), ref=F\arabic*]
        \item\label{HC:FT:unitary} It extends to a unitary operator $\mathcal{F}_{\rho}:L^2(G(F))\longrightarrow L^2(G(F)).$
        \item\label{FT:as}
        It restricts to an automorphism $\mathcal{F}_{\rho}:\mathcal{S}_\rho^{\mathrm{as}}(G(F))\longrightarrow \mathcal{S}_\rho^{\mathrm{as}}(G(F))$.
        \item\label{HC:FT:order} $\mathcal{F}_{\rho,\overline{\psi}}\circ \mathcal{F}_{\rho,\psi}=\mathrm{id}$ and $\mathcal{F}_{\rho,\psi}^4=\mathrm{id}$.
        \item\label{HC:FT:equivariant} For all $g,g'\in G(F),$ $\mathcal{F}_\rho\circ R(g,g')=R(g',g)\circ \mathcal{F}_\rho.$
        \item\label{HC:FT:zeta} For $(M,\sigma)\in\widetilde{\Temp}_\Ind(G)$ and $f\in\mathcal{C}(G(F))$, one has $$(\Ind_P^G\sigma)((\mathcal{F}_\rho f)^\vee) = \gamma(\tfrac{1}{2},\sigma,\rho_M,\psi)(\Ind_P^G\sigma)(f).$$
        \item\label{HC:FT:constant} Let $P=MN\in \mathcal{P}.$ We have a commutative diagram
        \begin{center}
        \begin{tikzcd}[sep=large]  
           \mathcal{S}^{\mathrm{as}}_{\rho}(G(F))\ar[r,"\mathcal{F}_{\rho}"]\ar[d,swap,"(-)^{(P)}"]&  {\mathcal{S}^{\mathrm{as}}_{\rho}(G(F))}\ar[d,"(-)^{(P)}"]\\
                 { \mathcal{S}^{\mathrm{as}}_{\rho_M}(M(F))} \ar[r,"\mathcal{F}_{\rho_M}"]& \mathcal{S}^{\mathrm{as}}_{\rho_M}(M(F)).
        \end{tikzcd}
        \end{center}
    \end{enumerate}
\qed
\end{theorem}

\begin{corollary}\label{Sas:FE:near0} Let $(M,\sigma)\in\widetilde{\Temp}_\Ind(G)$ and $(v,w)\in \Ind_{K_P}^K\sigma\vert_{K_M}\times \Ind_{K_P}^K\sigma^\vee\vert_{K_M}$. For $f\in \mathcal{S}_{\rho}^{\mathrm{as}}(G(F)),$ one has an identity of rational functions in $\chi\in \Lambda_M$
\begin{align*}
    Z_0(\Ind_P^G(\sigma_\chi)^\vee,\mathcal{F}_\rho f,w,v) = \gamma(\tfrac{1}{2},\sigma_\chi,\rho_M,\psi) Z_0(\Ind_P^G\sigma_\chi,f,v,w).
\end{align*}
\end{corollary}

\begin{proof} 
Observe that
\begin{align*}
    \gamma(\tfrac{1}{2},\sigma,\rho_M,\psi) Z_0(\Ind_P^G\sigma,f,v,w)&=
    \gamma(\tfrac{1}{2},\sigma,\rho_M,\psi)\langle\HP(f)(M,\sigma)v,w\rangle \\
    &\stackrel{\eqref{HC:FT:zeta}}{=}\langle\HP((\mathcal{F}_\rho f)^\vee)(M,\sigma)v,w\rangle \\
    &= \int_{G(F)} \mathcal{F}_\rho f(g^{-1})\langle(\Ind_P^G \sigma)(g)v,w\rangle dg\\
    &=\int_{G(F)} \mathcal{F}_\rho f(g)\langle v,(\Ind_P^G \sigma^\lor) (g)w\rangle dg\\
    &=Z_0(\Ind_P^G \sigma^\lor, \mathcal{F}_\rho f,w,v). 
\end{align*}
Therefore, the claimed identity holds for all $\chi\in \Im \Lambda_M,$ and the assertion follows from the identity principle.
\end{proof}


Assume $G$ is unramified. The \textbf{basic function} $b_\rho\in \mathcal{S}^{\mathrm{as}}_\rho(G(F)/\!/K)$ is defined in \cite[\S4.2]{DRS} to be the unique function such that 
\begin{align*}
    \HP(b_\rho)(M,\sigma)=L(\tfrac{1}{2},\sigma,\rho_M)\HP(\mathbf{1}_K)(M,\sigma)
\end{align*}
for $(M,\sigma)\in\widetilde{\Temp}_\Ind(G)$. If $\psi$ is unramified, one has $\mathcal{F}_{\rho}(b_\rho) = b_\rho$.

\subsection{$X^*(G)$-eigendecompositions} 

For $\alpha\in \Omega_G= G(F)/G(F)^1$, let $\mathbf{1}_\alpha$ denote the characteristic function of the coset
\begin{align*}
     \alpha G(F)^1 = \big\{g\in G(F): |\chi(g)| = q^{-\langle\chi,\alpha\rangle}\text{ for all }\chi\in X^*(G)\big\}.
\end{align*}
Recall for $\chi\in \Lambda_G,$ $\chi(\alpha)=q^{-\langle\chi,\alpha\rangle}$ is the value of $\chi$ on the coset $\alpha G(F)^1$.

By \cite[Lemma 6.8]{DRS}, each $f\in \mathcal{C}(G(F))$ can be written as a convergent sum $f=\sum_\alpha f\mathbf{1}_\alpha$ in $\mathcal{C}(G(F))$. In particular, $\bigoplus\limits_{\alpha\in\Omega_G}\mathcal{C}(G(F))\mathbf{1}_\alpha$ is dense in $\mathcal{C}(G(F))$. Let 
\begin{align*}
    \mathcal{C}(\Temp_\Ind(G))_\alpha:=\{T\in \mathcal{C}(\Temp_\Ind(G)): T(M,\sigma_\chi) = \chi(\alpha) T(M,\sigma) \text{ for }\chi\in \Im\Lambda_G\}.
\end{align*}
Let
$e_\alpha:\mathcal{C}(\Temp_\Ind(G))\to \mathcal{C}(\Temp_\Ind(G))_\alpha$ be the orthogonal projection. Explicitly, 
\begin{align}\label{ealpha:orthproj}
    e_\alpha T(M,\sigma) = \frac{1}{\vol(\Im \Lambda_G)}\int_{\Im \Lambda_G} \chi(\alpha)^{-1}T(M,\sigma_\chi)d\chi.
\end{align}
For $f\in \mathcal{C}(G(F)),$
\begin{align}\label{HP:orthoproj:alpha}
    \HP(f\mathbf{1}_\alpha) = e_\alpha\HP(f).
\end{align}
In particular, $\supp\HP(f\mathbf{1}_\alpha)\subseteq \supp\HP(f)$. Furthermore, for $(M,\sigma)\in\widetilde{\Temp}_\Ind(G)$,
    \begin{align}\label{eq:restricts:coset}
        Z_0(\Ind_P^G\sigma,f\mathbf{1}_\alpha,v,w) = \frac{1}{\mathrm{vol}(\Im \Lambda_G)}\int_{\Im \Lambda_G} \chi(\alpha)^{-1} Z_0(\Ind_P^G \sigma_\chi,f,v,w)d\chi
    \end{align}
for any $(v,w)\in\Ind_{K_P}^K\sigma\vert_{K_M}\times\Ind_{K_P}^K\sigma^\vee\vert_{K_M}.$

\begin{lemma}\label{alpha:supp:1} Let $f\in\mathcal{S}_\rho^{\mathrm{as}}(G(F))$ and let $K_0\leq K$ be a compact open subgroup such that $f$ is bi-$K_0$-invariant. Let  $\calB_0\subset X^\ast(M_0)$ be a finite generating set of the cone $\overline{C}_{\rho,M_0}$. There exists a constant $r\in\mathbb{Z}$ such that for all $\alpha\in \Omega_G,$   $P=MN\in \mathcal{P}, \sigma\in\Pi_2(M)$ and $(v,w)\in(\Ind_{K_P}^K\sigma\vert_{K_M})^{K_0}\times(\Ind_{K_P}^K\sigma^\vee\vert_{K_M})^{K_0}$, if we write
    \begin{align*}
        Z_0(\Ind_P^G\sigma_\chi,f\mathbf{1}_\alpha,v,w) = \sum_{\beta\in\Omega_M} a_{\beta,v,w} q^{-\langle\chi,\beta\rangle},
    \end{align*}
    then $a_{\beta,v,w}\neq 0$ only if $\beta\vert_{X^*(G)} = \alpha$ and $\langle\tau,\beta\rangle\geq -r$ for all $\tau\in\mathcal{B}_0$. Here the pairing $\langle\tau,\beta\rangle$ is interpreted as in  \eqref{pairing:bw:Levis}.
\end{lemma}
\begin{proof} 
 For $P=MN\in \mathcal{P}, \sigma\in\Pi_2(M)$ with $(\mathrm{Ind}_P^G \sigma)^{K_0}\neq 0,$ choose bases $\mathcal{E}_\sigma$ and $\mathcal{E}_{\sigma^\vee}$ of $(\Ind_{K_P}^K\sigma\vert_{K_M})^{K_0}$ and $(\Ind_{K_P}^K\sigma^\vee\vert_{K_M})^{K_0}$ respectively. By the definition of $\mathcal{S}_\rho^{\mathrm{as}}(G(F)),$ we can write
\begin{align*}
    \HP(f)(M,\sigma_\chi) = \sum_{(v,w)\in \mathcal{E}_\sigma\times \mathcal{E}_{\sigma^\lor}} L(\tfrac{1}{2},\sigma_\chi,\rho_M) p_{\sigma,v,w}(\chi) v\otimes w,
\end{align*}
where $p_{\sigma,v,w}\in \mathbb{C}[\Lambda_M]$. By \eqref{LambdaG:poly:supp} and \hyperref[Lfunction:expand:cone]{Lemma \ref{Lfunction:expand:cone}}, we can write
\begin{align*}
    L(\tfrac{1}{2},\sigma_\chi,\rho_M) p_{\sigma,v,w}(\chi) = \sum_{\beta\in S + \widetilde{\sigma}_{\rho,M}\cap\Omega_M} a_{\beta,\sigma,v,w}q^{-\langle\chi,\beta\rangle}
\end{align*}
for some finite set $S\subseteq\Omega_M.$ Since $\Theta^{K_0}$ consists of finitely many connected components, we can make $S$ uniform in $\sigma\in\Pi_2(M)$ and $v,w$.

 By \eqref{ealpha:orthproj} we have
\begin{align*}
      \HP(f\mathbf{1}_\alpha)(M,\sigma_\chi) &= \sum_{(v,w)\in \mathcal{E}_\sigma\times\mathcal{E}_{\sigma^\lor}} \left(\frac{1}{\mathrm{vol}(\Im\Lambda_G)}\int_{\Im \Lambda_G} \chi'(\alpha)^{-1}L(\tfrac{1}{2},\sigma_{\chi\chi'},\rho_M) p_{\sigma,v,w}(\chi\chi')d\chi'\right)v\otimes w\\
     &=\sum_{(v,w)\in \mathcal{E}_\sigma\times\mathcal{E}_{\sigma^\lor}} \left(\sum\limits_{\substack{\beta\in S+\widetilde{\sigma}_{\rho,M}\cap\Omega_M\\\beta\vert_{X^*(G)} = \alpha}}a_{\beta,\sigma,v,w}q^{-\langle\chi,\beta\rangle}\right)v\otimes w.
\end{align*}
Since $\widetilde{\sigma}_{\rho,M}$ is strongly convex, the sum is finite. The assertion then follows from \hyperref[Lcone:dualdual]{Lemma \ref{Lcone:dualdual}}.
\end{proof}

\begin{corollary}
    For $f\in \mathcal{S}^{\mathrm{as}}_\rho(G(F))$ and $\alpha\in \Omega_G$,  $\HP(f\mathbf{1}_\alpha)$ is a polynomial section.\qed
\end{corollary}

Fix a finite generating set $\calB_G\subseteq X^*(G)$ of the cone $\overline{C}_{\rho,G}$. For $r\in \RR,$ define
\begin{align*}
    S_{G}(r):=S_{G,\calB_G}(r):=\{g\in G(F): |\chi|(g)\le q^{r} \text{ for all }\chi\in\calB_G\}.
\end{align*}

\begin{corollary}\label{Sas:f1alpha:polynomial}  Let $f\in \mathcal{S}_\rho^{\mathrm{as}}(G(F))$. Then $\supp f\subseteq S_G(r)$ for some constant $r$.
\end{corollary}
\begin{proof} We may enlarge $\mathcal{B}_0$ so that it contains $\mathcal{B}_G.$ \hyperref[alpha:supp:1]{Lemma \ref{alpha:supp:1}} implies the corollary.
\end{proof}

\subsection{Compatible spectrum} 

We recall a matrical Paley-Wiener theorem for $C_c^\infty(G(F))$ due to Bernstein and Heiermann. Continue to denote by $\mathrm{End}$ the restriction of the bundle $\mathrm{End}\to \Temp_\Ind(G)$ to the open subspace $\Temp_{\Ind,0}(G)$. Let
\begin{align*}
     \mathrm{Poly}(\Temp_{\Ind}(G)) &\subseteq \Gamma_c(\Temp_{\Ind}(G),\mathrm{End}),\\
    \mathrm{Poly}(\Temp_{\Ind,0}(G)) &\subseteq \Gamma_c(\Temp_{\Ind,0}(G),\mathrm{End})
\end{align*}
denote the spaces of polynomial sections of $\mathrm{End}$ with compact support.

\begin{theorem}[\cite{Heiermann:HeckePlan}]\label{BH:PW} The map $\HP$ induces a $\mathbb{C}$-algebra isomorphism
\begin{center}
    \begin{tikzcd}[row sep=tiny]
        C_c^\infty(G(F))\arrow[r]& \mathrm{Poly}(\Temp_{\Ind,0}(G))\\
f \arrow[r,maps to]& \HP(f)\vert_{\Temp_{\Ind,0}(G)}.
    \end{tikzcd}
\end{center}
\qed
\end{theorem}

Clearly, $\HP(f)\in \mathrm{Poly}(\Temp_{\Ind}(G))$ for $f\in C^\infty_c(G(F)).$ However, as demonstrated in \cite[Example 6.2]{DRS}, a section in $\mathrm{Poly}(\Temp_{\Ind}(G))$ needs not arise from a function in $C_c^\infty(G(F)).$ To explain the obstruction, let us recall the \textbf{Bernstein variety}  $\Temp_{\Ind,0}(G)_\CC:$
\begin{align*}
    \Temp_{\Ind,0}(G)_\CC = \bigsqcup\limits_{M\in\mathcal{M}^{\mathrm{std}}} \Pi_0(M)_\CC / W(G|M) = \bigsqcup\limits_{(M,\mathcal{O})} \mathcal{O}/ W(M,\mathcal{O}),
\end{align*}
where the disjoint union is over all $\Lambda_M$-orbits $\mathcal{O}\subseteq \Pi_0(M)_\CC$, and $W(M,\mathcal{O}):= \{s\in W(G|M): s.\mathcal{O}=\mathcal{O}\}$. Let $\mathrm{Irr}(G)$ denote the isomorphism classes of irreducible smooth representations of $G(F)$. There is a finite-to-one surjection
\begin{align*}
    \mathrm{Irr}(G)\longrightarrow \Temp_{\Ind,0}(G)_\CC
\end{align*}
given as follows: For each $\pi\in\mathrm{Irr}(G)$ there exists a unique pair $(M,\mathcal{O})$ such that $\pi$ is a subquotient of $\Ind_P^G\sigma$ for some $\sigma\in \mathcal{O}$. The choice of $\sigma$ is unique up to $W(M,\calO),$ so the assignment $\pi\mapsto (M,\sigma\in \mathcal{O})$ is well-defined.

\begin{definition}\label{def:cs}
    Let $f\in \mathcal{C}(G(F)).$ We say $f$ has \textbf{compatible spectrum} if the following holds:
    \begin{enumerate}[label=(CS\arabic*), ref=CS\arabic*]
    \item \label{CS1} Suppose for each $(M,\sigma)\in \widetilde{\mathrm{Temp}}_{\Ind}(G)$ and $(v,w)\in \mathrm{Ind}_{K_P}^K \sigma|_{K_M}\times \mathrm{Ind}_{K_P}^K \sigma^\lor|_{K_M},$  the local zeta function $Z_0(\mathrm{Ind}_P^G\sigma_\chi,f,v,w),$ originally defined on $\Re(\chi)=0,$ extends to a holomorphic function in $\chi\in C\times \Im\Lambda_M$ for some positive cone $C\subseteq \Re\Lambda_M$ containing $0$.
    \item \label{CS2} Let $(M,\sigma)\in\widetilde{\Temp}_\Ind(G)$ and $(M',\sigma')\in\widetilde{\Temp}_{\Ind,0}(M)$ such that $\sigma\xhookrightarrow{\iota} \mathrm{Ind}_{P'\cap M}^{M} \sigma'_{\lambda}$ for some $\lambda\in \Lambda_{M'}$. Then for $(v,w)\in \Ind_{K_P}^K\sigma\vert_{K_M}\times \Ind_{K_P}^K\sigma^\vee\vert_{K_M}$ and any lift $\tilde{w}\in \mathrm{Ind}_{K_{P'}}^{K}\sigma'^\lor|_{K_{M'}}$ of $w$ under the dual map $\iota^\lor,$
    \begin{align*}
            Z_0(\mathrm{Ind}_P^G \sigma_\chi, f,v,w)=Z_0(\mathrm{Ind}_{P'}^G \sigma'_{\lambda\chi}, f,v,\tilde{w}).
    \end{align*}
    for $\chi\in C'\times \Im\Lambda_{M},$ where $C'\subseteq \Re\Lambda_{M}$ is a positive cone on which both terms are defined via \eqref{CS1}.
    \end{enumerate}
\end{definition}
Let
\begin{align*}
    \mathcal{C}_{\mathrm{cs}}(G(F)) \subseteq\mathcal{C}(G(F))
\end{align*}
denote the subspace consisting of functions with compatible spectrum.  Every function in $\mathcal{S}_{\rho}^{\mathrm{as}}(G(F))$ satisfies \eqref{CS1} but does not necessarily satisfies \eqref{CS2}. Indeed, the pathological example in \cite[Example 6.2]{DRS} is a function $f\in \mathcal{S}_{\rho}^{\mathrm{as}}(G(F))$ such that \eqref{CS2} fails while $\HP(f)\in \mathrm{Poly}(\Temp_{\Ind}(G)).$

\begin{lemma}\label{lem:comp+poly} Let $f\in \mathcal{C}_{\mathrm{cs}}(G(F))$. The following are equivalent.
\begin{enumerate}
    \item $f\in C^\infty_c(G(F))$.
    \item $\mathrm{HP}(f)\in \mathrm{Poly}(\Temp_{\Ind}(G)).$
    \item $\mathrm{HP}(f)\vert_{\Temp_{\Ind,0}(G)}\in  \mathrm{Poly}(\Temp_{\Ind,0}(G)).$
\end{enumerate}
\end{lemma}

\begin{proof}
    The implications $(1)\Rightarrow (2)\Rightarrow (3)$ are clear. Suppose $\mathrm{HP}(f)\vert_{\Temp_{\Ind,0}(G)}\in  \mathrm{Poly}(\Temp_{\Ind,0}(G)).$ By \hyperref[BH:PW]{Theorem \ref{BH:PW}} there is a function $f'\in C^\infty_c(G(F))$ such that $\mathrm{HP}(f-f')=0$ on $\Temp_{\Ind,0}(G)$. We claim that   $\mathrm{HP}(f-f')(M,\sigma)=0$ for all $(M,\sigma)\in \widetilde{\mathrm{Temp}}_{\Ind}(G).$  Then $f=f'$ by \hyperref[HCPlan]{Harish-Chandra Plancherel theorem}, so $(3)\Rightarrow(1)$.

    Say $\sigma\le \mathrm{Ind}_{P'\cap M}^{M} \sigma'_{\lambda}$ for some $(M',\sigma')\in \widetilde{\mathrm{Temp}}_{\mathrm{Ind,0}}(M)$ and $\lambda\in \Lambda_{M'}.$ By \eqref{CS2} for any $(v,w)\in \Ind_{K_P}^K\sigma\vert_{K_M}\times \Ind_{K_P}^K\sigma^\vee\vert_{K_M}$ and any lift $\tilde{w}\in \mathrm{Ind}_{K_{P'}}^{K}\sigma'^\lor|_{K_{M'}}$ of $w,$ there is a positive cone $C$ in $\Re\Lambda_M$ such that
    \begin{align*}
         Z_0(\mathrm{Ind}_P^G \sigma_{\chi}, f-f',v,w)=Z_0(\mathrm{Ind}_{P'}^G \sigma'_{\lambda\chi}, f-f',v,\tilde{w})=0
    \end{align*}
     for  $\chi\in C\times \Im \Lambda_M$. Here both terms are defined via \eqref{CS1}. By \eqref{CS1} and the identity principle,  $Z_0(\mathrm{Ind}_P^G \sigma_{\chi}, f-f',v,w)=0$ for $\Re(\chi)=0$ for any $v,w$, so $\mathrm{HP}(f-f')(M,\sigma)=0.$
\end{proof}

By \eqref{eq:restricts:coset} we have

\begin{lemma}\label{Ccs:onealpha:stable} For $\alpha\in\Omega_G$, one has $\mathcal{C}_{\mathrm{cs}}(G(F))\mathbf{1}_\alpha\subseteq \mathcal{C}_{\mathrm{cs}}(G(F))$.\qed
\end{lemma}

Define the space of \textbf{almost compactly supported functions} 
\begin{align}\label{eq:ac}
    C_{\mathrm{ac}}^\infty(G(F)) := \left\{f\in C_u^\infty(G(F)) : f\mathbf{1}_\alpha\in C_c^\infty(G(F))\text{ for all }\alpha\in \Omega_G\right\}.
\end{align}

\begin{corollary}\label{Sas+csisac} One has $\mathcal{S}_\rho^{\mathrm{as}}(G(F))\cap \mathcal{C}_{\mathrm{cs}}(G(F))\subseteq C_{\mathrm{ac}}^\infty(G(F))$.
\end{corollary}
\begin{proof} Let $f\in \mathcal{S}_\rho^{\mathrm{as}}(G(F))\cap \mathcal{C}_{\mathrm{cs}}(G(F))$. By \hyperref[alpha:supp:1]{Lemma \ref{alpha:supp:1}} and \hyperref[Ccs:onealpha:stable]{Lemma \ref{Ccs:onealpha:stable}},  $f\mathbf{1}_\alpha$ has compatible spectrum and $\HP(f\mathbf{1}_\alpha)$ is a polynomial section. The corollary follows by \hyperref[lem:comp+poly]{Lemma \ref{lem:comp+poly}}.
\end{proof}

\section{Analytic behaviors of Schwartz functions}\label{sec:Schwartz:analytic}

In this section, we give an analytic definition of the space $\mathcal{S}_\rho^{\mathrm{as}}(G(F))\cap \mathcal{C}_{\mathrm{cs}}(G(F)).$

\subsection{Preparatory analysis}\label{ssec:prep}

\begin{lemma}\label{(i):equivalent:key} Let $f\in C^\infty_u(G(F))$, $P=MN\in \mathcal{P}$, and let $\sigma$ be a smooth representation of $M(F).$ The following are equivalent: 
\begin{enumerate}
    \item For all $v\in \Ind_{K_P }^K\sigma\vert_{K_M }$ and $w\in \Ind_{K_P }^K\sigma^\vee\vert_{K_M }$
    \begin{align*}
        g\mapsto f(g)\langle(\Ind_P^G\sigma)(g)v,w\rangle \in L^1(G(F)).
    \end{align*}
    \item For all $v\in \Ind_{K_P }^K\sigma\vert_{K_M }$ and $w\in \Ind_{K_P }^K\sigma^\vee\vert_{K_M }$
    \begin{align*}
        (g,k)\mapsto f(k^{-1}g)\langle v(g),w(k)\rangle \in L^1(G(F)\times K).
    \end{align*}
\end{enumerate}
\end{lemma}

\begin{proof} One has
\begin{align*}
    \int_{G(F)} \left|f(g) \langle (\Ind_P^G\sigma)(g)v,w\rangle\right| dg &= \int_{G(F)} \left|f(g) \int_K \langle v(kg),w(k)\rangle dk\right| dg\\
    &\leq \int_{G(F)\times K} \left|f(k^{-1}g) \langle v(g),w(k)\rangle\right| dk dg.
\end{align*}
This shows the implication (2)$\,\Rightarrow$(1). Now assume (1). Choose a compact open subgroup $K_0\le K$ such that $f\in C^\infty(G(F)/\!/K_0)$ and $v,w$ are $K_0$-invariant. We claim for $k_0\in K_0$
\begin{align}\label{eq:easyKinvariance}
    \int_{G(F)} f(g) \int_K \langle v(kk_0g),w(k)\rangle dk dg &= \int_{G(F)} f(g)\int_K  \langle v(kg),w(k)\rangle dkdg.
\end{align}
To see this, by the Iwasawa decomposition 
the integral on the left is
\begin{align*}
    &\int_{P(F)\times K} f(k'p) \int_{K} \langle v(kk_0k'p),w(k)\rangle dkdk'd_r p.
\end{align*}
For a given $p$, the integral over $K\times K$ is absolutely convergent, so by the Fubini-Tonelli theorem it is
\begin{align*}
    &\int_{P(F)}\left(\int_{K\times K} f(k'p)  \langle v(kk_0k'p),w(k)\rangle dk dk\right)d_r p\\
    &=\int_{P(F)}\left(\int_{K\times K} f(k'p)  \langle v(kk'p),w(k)\rangle dk dk\right)d_r p
\end{align*}
by changing variables $k'\mapsto k_0^{-1}k'$. This is equal to the integral on the right-hand side of \eqref{eq:easyKinvariance}.

It follows that we can write
\begin{align*}
    &\frac{1}{\vol(K_0,dk)}\int_{G(F)} f(g)\int_K  \langle v(kg),w(k)\rangle dkdg\\
    &=   \int_{G(F)} f(g) \sum_{k\in K/K_0}  \langle v(kg),w(k)\rangle dg\\
    &=\int_{G(F)} f(g) \sum_{k'\in K_P\backslash  K/K_0} \sum_{k\in K_Pk'K_0/K_0} \langle v(kg),w(k)\rangle dg\\
    &=\int_{G(F)} f(g) \sum_{k'\in K_P\backslash K/K_0} \#(K_Pk'K_0/K_0) \langle v(k'g),w(k')\rangle dg.
\end{align*}
For $k'\in K_P\backslash  K/K_0$, define $w_{k'}:K\to \sigma^\vee$ by
\begin{align*}
    w_{k'}(k) = \left\{\begin{array}{ll}
        w(k) &\textrm{if } k\in K_Pk'K_0,  \\
        -w(k) &\textrm{if } k\not\in K_Pk'K_0. 
    \end{array}\right.
\end{align*}
One easily checks this is well-defined and $w_{k'}\in (\Ind_{K_P }^K\sigma^\vee\vert_{K_M })^{K_0}$. Then
\begin{align*}
     &\int_{G(F)} \left|f(g)\langle(\mathrm{Ind}_P^G \sigma)(g)v,w+w_{k'} \rangle\right| dg\\
     &=\int_{G(F)} |f(g)| \left|\int_K \langle v(kg),w(k)+w_{k'}(k)\rangle dk\right|dg \\
     &= 2\#(K_Pk'K_0/K_0)\mathrm{vol}(K_0,dk)\int_{G(F)}|f(g)  \langle v(k'g),w(k')\rangle| dg.
\end{align*}
Therefore, (1) implies (2) by changing variables $g\mapsto k^{\prime-1}g$ and varying $k'\in K_P\backslash  K/K_0$. 
\end{proof}

\begin{lemma} \label{lem:exchange}
Let $f\in C^\infty_u(G(F)).$ Suppose the equivalent conditions in \hyperref[(i):equivalent:key]{Lemma \ref{(i):equivalent:key}} hold for $(P,\sigma)$. Then for any $(v,w)\in \Ind_{K_P }^K\sigma\vert_{K_M }\times\Ind_{K_P }^K\sigma^\vee\vert_{K_M }$ one has
\begin{align*}
    \sup_{\chi\in\Im\Lambda_M}\int_{G(F)}|f(g)\langle (\Ind_P^G\sigma_{\chi})(g)v,w\rangle| dg<\infty.
\end{align*}
In particular, the equivalent conditions in \hyperref[(i):equivalent:key]{Lemma \ref{(i):equivalent:key}} hold for $(P,\sigma_\chi)$ for all $\chi\in \Im \Lambda_M.$
\end{lemma}
\begin{proof}
 One has
\begin{align*}
    \left|\langle v_{\chi}(kg),w(k)\rangle \right| &= \left|\chi(\textbf{a}_P(kg))\langle v(kg),w(k)\rangle \right|= \left|\langle v(kg),w(k)\rangle \right|.
\end{align*}
Therefore,
\begin{align*}
    \int_{G(F)} \left|f(g)\right|\left|\langle (\Ind_P^G\sigma_{\chi})(g)v,w\rangle\right| dg &\leq  \int_{G(F)\times K} |f(g)| \left|\langle v_{\chi}(kg),w(k)\rangle \right|dk dg\\
    &=\int_{G(F)\times K} |f(g)| \left|\langle v(kg),w(k)\rangle \right|dk dg.
\end{align*}
By \hyperref[(i):equivalent:key]{Lemma \ref{(i):equivalent:key}} the last integral is finite.
\end{proof}

Recall the notations in \S\ref{subsec:intertwining:recall} for the next lemma.
\begin{lemma}\label{intertwine:W}
Let $f\in C^\infty_{u}(G(F))$. Suppose the equivalent conditions in \hyperref[(i):equivalent:key]{Lemma \ref{(i):equivalent:key}} hold for $(P,\sigma)$. Then they hold for $(s.P,s.\sigma)$ for all $s\in W(G,A_0)$, in which case we have
\begin{align*}
     \int_{G(F)} f(g)\langle (\mathrm{Ind}_P^G\sigma )(g)v,w\rangle dg=\int_{G(F)} f(g)\langle (\Ind_{s.P}^Gs.\sigma)(g)v_s,w_{s}\rangle dg
\end{align*}
for any $(v,w)\in \Ind_{K_P }^K\sigma\vert_{K_M }\times\Ind_{K_P }^K\sigma^\vee\vert_{K_M }$.
\end{lemma}

\begin{proof} Let $s\in W(G,A_0)$ and realize it as an element in $K$.
    Formally, we have
    \begin{align*}
        \int_{G(F)} f(g)\langle (\Ind_{s.P}^Gs.\sigma)(g)v_s,w_{s}\rangle dg&=\int_{G(F)\times K} f(g)\langle v(s^{-1}kg),w(s^{-1}k)\rangle dkdg\\
        &=\int_{G(F)\times K} f(g)\langle v(kg),w(k)\rangle dkdg\\
        &=\int_{G(F)} f(g)\langle (\mathrm{Ind}_P^G\sigma )(g)v,w\rangle dg.
    \end{align*}
     The rest follows from \hyperref[(i):equivalent:key]{Lemma \ref{(i):equivalent:key}}.
\end{proof}


\begin{lemma}\label{lem:zeta:holomorphic} Let $f\in C^\infty_u(G(F))$, $P=MN\in \mathcal{P}$, and let $\sigma$ be a smooth representation of $M(F).$ Suppose there is an open (connected) subset $C$ of $\Re\Lambda_M$ such that the equivalent conditions in \hyperref[(i):equivalent:key]{Lemma \ref{(i):equivalent:key}} hold for $(P,\sigma_\chi)$ for all $\chi\in C\times \Im\Lambda_M.$ Then for any $(v,w)\in \Ind_{K_P }^K\sigma\vert_{K_M }\times\Ind_{K_P }^K\sigma^\vee\vert_{K_M },$
\begin{align*}
            \int_{G(F)\times K}f(k^{-1}g)\langle v_{\chi}(g),w(k)\rangle dgdk
\end{align*}
is holomorphic as a function in $\chi$ on $C\times \Im\Lambda_M.$
\end{lemma}

\begin{proof} 
Fix $\chi\in C\times \Im\Lambda_M.$ Choose a point $(z_1,\ldots,z_n)\in \CC^n\cong X^\ast(M)_\CC$ that maps to $\chi$ under the map \eqref{XG->LambdaG:surj}. By the equivalent definitions of holomorphicity in \cite[\S 0.2]{Complexanalysis}, it suffices to show holomorphicity in each $z_i$. Fix an $i$, and let $U$ be a small open neighborhood of $z_i$ in $\CC$ such that the image of $(z_1,\ldots,z_{i-1},z,z_{i+1},\ldots, z_n)$ in $\Lambda_M,$ denoted as $\chi_z,$ lies in $C\times\Im \Lambda_M$ for all $z\in U$. We first show that
\begin{align*}
    h(z):=\int_{G(F)\times K}f(k^{-1}g)\langle v_{\chi_z}(g),w(k)\rangle dgdk
\end{align*}
is continuous as a function of $z$ in $U$. Let $U'\subseteq U$ be a compact subset. Choose $z',z''\in U'$ such that $\mathrm{Re}(z')\le \mathrm{Re}(z)\le \mathrm{Re}(z'')$ for all $z\in U'.$ Then for $z\in U',$ 
\begin{align*}
     |\langle v_{\chi_z}(g),w(k)\rangle|\le |\langle v_{\chi_{z'}}(g),w(k)\rangle|+|\langle v_{\chi_{z''}}(g),w(k)\rangle| 
\end{align*}
for all $(g,k)\in G(F)\times K.$ It follows that if $(w_j)_{j>0}$ is a sequence in $U$ that converges to $z\in U$, then by the Lebesgue's dominated convergence theorem
\begin{align*}
    \lim_{j\to \infty} h(w_j)= h(z).
\end{align*}

 By Morera's theorem, to prove the lemma it suffices to show
\begin{align*}
    \int_{\gamma} h(z)dz=0
\end{align*}
for any simple closed piecewise $C^1$-curve $\gamma$ in $U$. By the argument above, we have
\begin{align*}
    \int_{\gamma}\int_{G(F)\times K} |f(k^{-1}g)\langle v_{\chi_z}(g),w(k)\rangle| dgdkdz<\infty.
\end{align*}
Therefore, by the Fubini-Tonelli theorem 
\begin{align*}
    \int_{\gamma}h(z)dz=\int_{G(F)\times K}\int_\gamma f(k^{-1}g)\langle v_{\chi_z}(g),w(k)\rangle dz dg dk=0
\end{align*}
since the integrand is holomorphic. This completes the proof.

\end{proof}

\subsection{Absolutely convergent zeta integrals}\label{subsec:L1zeta}
Recall the definition of positive cones in \hyperref[def:positivecone]{Definition \ref{def:positivecone}}. Let $\mathcal{S}^{\infty}(G(F))\subseteq C^\infty_{u}(G(F))$ consist of functions $f$ such that
\begin{enumerate}[label=(S\text{\arabic*}), ref=S\text{\arabic*}]
    \item\label{S1} for every $(M,\sigma)\in\widetilde{\mathrm{Temp}}_\Ind(G)$ there exists a positive cone $C$ in $\Re\Lambda_M$  such that for all $(v,w)\in\Ind_{K_P }^K\sigma\vert_{K_M }\times \Ind_{K_P }^K\sigma^\vee\vert_{K_M }$, 
    \begin{align*}
        g\mapsto f(g)\langle(\Ind_P^G\sigma_\chi)(g)v,w\rangle \in L^1(G(F))
    \end{align*}
    for $\chi\in C$. 
\end{enumerate}
By \S \ref{ssec:prep}, if $f\in \mathcal{S}^\infty(G(F)),$ then for any $P=MN\in \mathcal{P},$  $\sigma\in \Pi_2(M),$ there is a positive cone $C$ in $\Re \Lambda_M$ such that for all $(v,w)\in\Ind_{K_P }^K\sigma\vert_{K_M }\times \Ind_{K_P }^K\sigma^\vee\vert_{K_M }$, the integral
\begin{align*}
    \int_{G(F)} f(g)\langle(\Ind_P^G\sigma_\chi)(g)v,w\rangle dg
\end{align*}
is absolutely convergent and defines a holomorphic function in $\chi$ on $C\times \Im \Lambda_M$.

\begin{lemma}\label{Sinftymodule}
    The space $\mathcal{S}^\infty(G(F))$ is a $G(F)\times G(F)$-module.  
\end{lemma}

\begin{proof}
    Let $(M,\sigma)\in\widetilde{\mathrm{Temp}}_\Ind(G).$ For $f\in \mathcal{S}^\infty(G(F)),$ there is a positive cone $C$ in $\Re\Lambda_M$ such that
    \begin{align*}
        g\mapsto f(g)\langle(\Ind_P^G\sigma_\chi)(g)v,w\rangle \in L^1(G(F))
    \end{align*}
    for all $\chi\in C$ and  $(v,w)\in \Ind_{K_P }^K\sigma\vert_{K_M }\times \Ind_{K_P }^K\sigma^\vee\vert_{K_M }$. 
    Let $h\in G(F).$ Then
    \begin{align*}
        \langle v_\chi(gh),w(k)\rangle=\chi(\mathbf{a}_P(gh))\langle v(gh),w(k)\rangle.
    \end{align*}
    Since $\textbf{a}_P(g)^{-1}\textbf{a}_P(gh)\in K_M\mathbf{a}_P(\mathbf{k}_P(g)h)$ and $Kh$ is compact, one has 
    \begin{align*}
        \sup_{g\in G(F)}\left|\frac{\chi(\mathbf{a}_P(gh))}{\chi(\mathbf{a}_P(g))}\right|<\infty.
    \end{align*}
     Now viewing $v'(\cdot):=v(\cdot h)$ as an element in $\Ind_{K_{P}}^K \sigma|_{K_M},$ we have
    \begin{align*}
        \int_{G(F)} \left|f(gh^{-1})\langle (\Ind_{P}^G \sigma_\chi)(g)v,w\rangle\right| dg&\le \int_{K\times G(F)} \left|f(k^{-1}g)\langle v_\chi(gh),w(k)\rangle\right| dgdk\\
        & \ll \int_{K\times G(F)} \left|f(k^{-1}g)\chi(\mathbf{a}(g))\langle v'(g),w(k)\rangle\right|dgdk\\
        &= \int_{K\times G(F)} \left|f(k^{-1}g)\langle v'_\chi(g),w(k)\rangle\right|dgdk.
    \end{align*}
By \hyperref[(i):equivalent:key]{Lemma \ref{(i):equivalent:key}} the integral is finite for $\chi\in C$. This shows $\mathcal{S}^\infty(G(F))$ is stable under right translation by $G(F)$. The proof is similar for left translation, by noting $\langle(\Ind_P^G\sigma_\chi)(h)v,w\rangle=\langle v,(\Ind_P^G(\sigma_\chi)^\lor)(h^{-1})w\rangle$.
\end{proof}

It follows from the proof above that the positive cone $C$ only depends on $\sigma$ and the $G(F)\times G(F)$-submodule generated by $f$. Recall the zonal spherical function $\Xi_\chi$ \eqref{defn:zonal:spherical}. For $f\in C^\infty_u(G(F)),$ let
\begin{align*}
    C_f:=\left\{\chi \in \Re\Lambda_{M_0}: f\Xi_\chi\in L^1(G(F))\right\}.
\end{align*}
It follows from 
\hyperref[(i):equivalent:key]{Lemma \ref{(i):equivalent:key}} and the identity
    \begin{align*}
       |\langle (v_0)_\chi(g),v_0(k)\rangle|=\chi\delta_{P_0}^{1/2}(\mathbf{a}(g))
    \end{align*}
that
\begin{align*}
        C_f&=\left\{\chi\in \Re\Lambda_{M_0}: (g,k)\mapsto f(k^{-1}g)\chi\delta_{P_0}^{1/2}(\textbf{a}(g)) \in L^1(G(F)\times K)\right\}.
    \end{align*}

\begin{lemma}\label{lem:rapiddecay}
    Let $f\in C^\infty_u(G(F)).$ Then for any $(M,\sigma)\in \widetilde{\mathrm{Temp}}_{\Ind,0}(G)$ and $(v,w)\in \Ind_{K_P}^K\sigma\vert_{K_M}\times \Ind_{K_P}^K\sigma^\vee\vert_{K_M}$, one has \begin{align*}
    g\mapsto f(g)\langle(\Ind_P^G\sigma_\chi)(g)v,w\rangle\in L^1(G(F))
    \end{align*} 
    for $\chi\in C_f\cap \Re\Lambda_M$.
\end{lemma}

\begin{proof}
 For $\chi\in \Re\Lambda_M$
    \begin{align*}
        \int_{G\times K} |f|(k^{-1}g)|\langle v_\chi(g),w(k)\rangle|dkdg=\int_{G\times K} |f|(k^{-1}g)\chi\delta_{P}^{1/2}(\mathbf{a}(g))|\langle \sigma\big(\mathbf{a}(g)\widetilde{\mathbf{n}}(g)\big)v(\mathbf{k}(g)),w(k)\rangle |dkdg,
    \end{align*}
    where $\widetilde{\mathbf{n}}$ is the composite $G(F)\xrightarrow{\mathbf{n}} N_0(F)\to N_0(F)/N_P(F).$
    Since a matrix coefficient of $\sigma$ has compact support modulo center, and $\delta_P$ and $\delta_{P_0}$ agree on $Z_M(F)$, the integral above converges if and only if
    \begin{align*}
       \int_{G\times K} |f|(k^{-1}g)\chi\delta_{P_0}^{1/2}(\mathbf{a}(g))|\langle \sigma(\mathbf{a}(g)\widetilde{\mathbf{n}}(g))v(\mathbf{k}(g),w(k)\rangle |dkdg<\infty.
    \end{align*}
    This integral is bounded above by
    \begin{align*}
         \left(\int_{G\times K} |f|(k^{-1}g)\chi\delta_{P_0}^{1/2}(\mathbf{a}(g))dkdg\right) \sup_{(k,g)\in K\times G(F)}|\langle \sigma(\mathbf{a}(g)\widetilde{\mathbf{n}}(g))v(\mathbf{k}(g)),w(k)\rangle|,
    \end{align*}
    which is finite if $\chi\in C_f.$    
\end{proof}

\begin{proposition}\label{Cfpositive} Let $f\in C^\infty_u(G(F)).$ Then $f\in \mathcal{S}^\infty(G(F))$ if and only if $C_f$ contains a positive cone in $\Re\Lambda_{M_0}$. 
\end{proposition}

\begin{proof}
    If $f\in \mathcal{S}^\infty(G(F))$, by \eqref{S1} $C_f$ contains a positive cone in $\Re\Lambda_{M_0}$. For the converse, suppose $C_f$ contains a positive cone in $\Re\Lambda_{M_0}$. Let $P=MN\in \mathcal{P}^{\mathrm{std}}$ and $\sigma\in \Pi_2(M).$ There exist a standard parabolic subgroup $P'\le P,$ $\sigma'\in \Pi_0(M')$ and $\lambda\in \Lambda_{M'}$ such that $\sigma$ is a subrepresentation of $\Ind_{P'\cap M}^{M}\sigma'_\lambda$, and hence $\mathrm{Ind}_{P}^G\sigma_\chi$ is a subrepresentation of $\Ind_{P'}^G\sigma'_{\lambda\chi}.$ Since $C_f$ contains a positive cone, $\lambda|\nu|^r\in C_f$ for $r\in \RR$ large. By \hyperref[lem:rapiddecay]{Lemma \ref{lem:rapiddecay}}, we have for all $(v',w')\in\Ind_{K_{P'} }^K\sigma'\vert_{K_{M'} }\times\Ind_{K_{P'} }^K\sigma'^\vee\vert_{K_{M'}}$ that
    \begin{align*}
    g\mapsto f(g)\langle(\Ind_{P'}^G\sigma'_{\lambda\chi})(g)v',w'\rangle\in L^1(G(F))
    \end{align*} 
    for $\lambda\chi\in C_f\cap \Re\Lambda_{M'}$. In particular, for all $(v,w)\in\Ind_{K_P }^K\sigma\vert_{K_M }\times\Ind_{K_P }^K\sigma^\vee\vert_{K_M}$ we have
    \begin{align*}
    g\mapsto f(g)\langle(\Ind_{P}^G\sigma_{\chi})(g)v,w\rangle\in L^1(G(F))
    \end{align*} 
    for $\chi$ in a positive cone in $\Re \Lambda_M$.
\end{proof}

Recall for a compact open subgroup $K_0\leq K,$
\begin{align*}
    \Theta^{K_0} = \{(M,\sigma)\in\widetilde{\Temp}_\Ind(G): (\Ind_P^G\sigma)^{K_0}\neq 0\}
\end{align*}
consists of finitely many connected components. This fact together with \hyperref[Cfpositive]{Proposition \ref{Cfpositive}} implies the following.

\begin{corollary}\label{cor:Cfpositive}
    Let $f\in \mathcal{S}^\infty(G(F))$. For any compact open subgroup $K_0\leq K,$ there exists a positive cone $C$ in $\Re\Lambda_{M_0}$ such that for every $M\in\mathcal{M}^{\mathrm{std}}$, the intersection $C\cap \Re\Lambda_M$ is a positive cone, and for $(M,\sigma)\in \Theta^{K_0}$ and $(v,w)\in\Ind_{K_P }^K\sigma\vert_{K_M }\times\Ind_{K_P }^K\sigma^\vee\vert_{K_M }$,
\begin{align*}
        g\mapsto f(g)\langle(\Ind_P^G\sigma_\chi)(g)v,w\rangle \in L^1(G(F))
    \end{align*}
for $\chi\in C\cap \Re\Lambda_{M}$. \qed
\end{corollary}

\begin{lemma}\label{lemma:constant} Let $f\in \mathcal{S}^\infty(G(F))$ and $P=MN\in \mathcal{P}$.
For each $m\in M(F)$, one has
\begin{align*}
    n\mapsto f(mn) \in L^1(N(F)).
\end{align*}    
Furthermore, the map $(-)^P$ induces a morphism
\begin{align*}
    (-)^{(P)}:\mathcal{S}^\infty(G(F))\longrightarrow\mathcal{S}^\infty(M(F)).
\end{align*}
\end{lemma}

\begin{proof} 
  By \hyperref[lem:exchange]{Lemma \ref{lem:exchange}} we may assume $P$ is standard. Since $|\nu|^r$ is a subrepresentation of $\mathrm{Ind}_{P_0}^G |\nu|^r\lambda$ for some $\lambda\in \Lambda_{M_0}$ for all $r\in \RR,$ by \eqref{S1} for $r$ large
  \begin{align*}
      \int_{G(F)} |f(g)||\nu|^r(g) dg <\infty.
  \end{align*}
  By the Iwasawa decomposition and the Fubini-Tonelli theorem, for $(m,k)\in M(F)\times K$ outside a measure zero set $U,$
  \begin{align*}
      n\mapsto f(mnk)\in L^1(N(F)).
  \end{align*}
  Since $f\in C^\infty_u(G(F)),$ if $U$ is nonempty, then it must have positive measure. This proves the first assertion.

  For $\chi\in C_f,$ the integral 
    \begin{align*}
        \int_{G(F)\times K_M} f(k^{-1}g)\chi\delta_{P_0}^{1/2}(\mathbf{a}(g))dgdk
    \end{align*}
    converges absolutely. By the Iwasawa decomposition, the integral equals
    \begin{align*}
        &\int_{M(F)\times K_M}\chi\delta_{P_0}^{1/2}(\mathbf{a}(m)) \int_K\int_{N(F)} f(k^{-1}mnk')dndk'dmdk\\
        &=\int_{K}\int_{M(F)\times K_M} (R(1,k')f)^{(P)}(k^{-1}m)\chi\delta_{P_0\cap M}^{1/2}(\mathbf{a}(m)) dmdkdk'.
    \end{align*}
    In particular, 
    \begin{align*}
        \int_{M(F)\times K_M}f^{(P)}(k^{-1}m)\chi\delta_{P_0\cap M}^{1/2}(\mathbf{a}(m))dmdk
    \end{align*}
    converges absolutely for $\chi\in C_f.$ The second assertion follows from \hyperref[Cfpositive]{Proposition \ref{Cfpositive}}.
\end{proof}

The following is a straightforward computation using the Iwasawa decomposition.

\begin{lemma}\label{constant:zeta:identity} Let $f\in \mathcal{S}^\infty(G(F))$. Let $K_0\leq K$ be a compact open subgroup such that $f$ is bi-$K_0$-invariant.  For $(M,\sigma)\in\widetilde{\Temp}_\Ind(G)$ and $(v,w)\in( \Ind_{K_P}^K\sigma\vert_{K_M})^{K_0}\times (\Ind_{K_P}^K\sigma^\vee\vert_{K_M})^{K_0}$ one has
\begin{align}\label{eq:descendzeta}
    Z(\Ind_P^G\sigma_\chi,f,v,w) = \vol(K_0,dg)^2\sum_{k,k'\in K/K_0} Z(\sigma_\chi,(R(k,k')f)^{(P)},v(k'),w(k))
\end{align}
for $\Re(\chi)$ in a positive cone in $\Re\Lambda_M$.\qed
\end{lemma}

\subsection{Schwartz spaces}\label{defn:Schwartz}

Consider functions $f\in \mathcal{S}^\infty(G(F))$ that satisfy in addition the following properties: 
\begin{enumerate}[label=(S\text{\arabic*}), ref=S\text{\arabic*}] \setcounter{enumi}{1}
    \item\label{S2} For every $(M,\sigma)\in\widetilde{\mathrm{Temp}}_\Ind(G)$
 and $(v,w)\in\Ind_{K_P }^K\sigma\vert_{K_M }\times \Ind_{K_P }^K\sigma^\vee\vert_{K_M },$ the integral
    \begin{align*}
        Z(\mathrm{Ind}_P^G \sigma_\chi,f,v,w):= \int_{G(F)} f(g)\langle(\Ind_P^G\sigma_\chi)(g)v,w\rangle  dg,
    \end{align*}
    originally defined for $\mathrm{Re}(\chi)$ in some positive cone in $\Re\Lambda_M$, extends to a meromorphic function on $\Lambda_M$ such that 
    \begin{align*}
        \chi\mapsto \dfrac{Z(\mathrm{Ind}_P^G \sigma_\chi,f,v,w)}{L(0,\sigma_{\chi},\rho_M)} \in \mathbb{C}[\Lambda_M].
    \end{align*}
\end{enumerate}
By the definition of positive cones, the intersection of a positive cone in $\Re \Lambda_M$ and $C_{\rho,M}$ is a nonempty open set. Hence the statement regarding extension makes sense.

\begin{lemma} \label{lem:constant}
Let $f\in \mathcal{S}^\infty(G(F))$ satisfy \eqref{S2}. Let $P=MN\in \mathcal{P}^{\mathrm{std}}.$ Then $f^{(P)}\in \mathcal{S}^\infty(M(F))$ also satisfies \eqref{S2}. Furthermore, the identity \eqref{eq:descendzeta} holds as meromorphic functions in $\chi\in \Lambda_M.$
\end{lemma}
\begin{proof} By an analogous argument as in \hyperref[(i):equivalent:key]{Lemma \ref{(i):equivalent:key}}, we can extract a single term on the right-hand side of \eqref{eq:descendzeta} by writing
\begin{align*}
    Z(\sigma_\chi, (R(k,k')f)^{(P)}, v(k'), w(k))=\sum_{i\in I} c_iZ(\mathrm{Ind}_P^G \sigma_\chi,f,v_i,w_i)  
\end{align*}
as a finite sum for some $c_i\in \CC$ and $(v_i,w_i)\in  \Ind_{K_P}^K\sigma\vert_{K_M}\times \Ind_{K_P}^K\sigma^\vee\vert_{K_M}$ depending on $k,k'.$ The lemma follows by the definition of \eqref{S2}.
\end{proof}

Define the \textbf{$\rho$-Schwartz space} 
\begin{align*}
    \mathcal{S}_\rho(G(F)):=\{ f\in \mathcal{S}^\infty(G(F)): f \textrm{ satisfies } \eqref{S2} \textrm{ and } |\nu|^rf\in \mathcal{C}(G(F)) \,\,\forall r>0 \textrm{ sufficiently large}\}.
\end{align*}

\begin{proposition} \label{Snur:HCSchwartz}
One has the equality
\begin{align*}
    |\nu|^{1/2}\mathcal{S}_\rho(G(F)) = \mathcal{S}_\rho^{\mathrm{as}}(G(F))\cap \mathcal{S}^\infty(G(F)).
\end{align*}
In addition, for any $(M,\sigma)\in \widetilde{\Temp}_{\Ind}(G)$ and $(v,w)\in  \Ind_{K_P}^K\sigma\vert_{K_M}\times \Ind_{K_P}^K\sigma^\vee\vert_{K_M},$ we have an identity of meromorphic functions 
     \begin{align}\label{eq:Z=Z0}
         Z(\mathrm{Ind}_P^G\sigma,|\nu|^{1/2}f,v,w)=Z_0(\mathrm{Ind}_P^G\sigma, |\nu|^{1/2}f,v,w).
     \end{align}
\end{proposition}

\begin{proof} Let $f\in \mathcal{S}^\infty(G(F))$ and let $K_0\leq K$ be a compact open subgroup such that $f\in C^\infty(G(F)/\!/K_0)$. Note that both terms in \eqref{eq:Z=Z0} are the zero function if $(M,\sigma)\not\in \Theta^{K_0}$ whenever they are defined. Thus in the proof below we only need to discuss $(M,\sigma)\in \Theta^{K_0}.$ By \hyperref[cor:Cfpositive]{Corollary \ref{cor:Cfpositive}} there is a positive cone $C$ in $\Re\Lambda_{M_0}$ such that for all $(M,\sigma)\in\Theta^{K_0}$,
\begin{align*}
    g\mapsto f(g)\langle(\Ind_P^G\sigma_\chi)(g)v,w\rangle \in L^1(G(F))
\end{align*}
for $\chi\in (C\cap \Re\Lambda_M)\times\Im\Lambda_M$ and $(v,w)\in  \Ind_{K_P}^K\sigma\vert_{K_M}\times \Ind_{K_P}^K\sigma^\vee\vert_{K_M}$. 

Suppose $f$ satisfies \eqref{S2} and there is $r_0>0$ such that $|\nu|^rf\in \mathcal{C}(G(F))$ for $r\geq r_0$. By enlarging $r_0,$ we may assume $|\nu|^r\in C$ for $r\ge r_0.$ We claim $|\nu|^s f \in \mathcal{C}(G(F))$ for $s\in \CC$ with $\Re(s)>0$. By the \hyperref[HCPlan]{Harish-Chandra Plancherel theorem}, for $\Re(s)\geq r_0$
\begin{align*}
    |\nu|^s f = \HP^{-1}(\HP(|\nu|^s f)).
\end{align*}
 Since $|\nu|^s f$ is holomorphic in $s\in \CC$ and $\HP$ is a topological isomorphism, by the identity principle it suffices to show $\HP(|\nu|^s f)$ extends to a holomorphic function on $\CC_{>0}:=\{s\in \CC: \Re(s)>0\}$. That is, we need to show for $(M,\sigma)\in \Theta^{K_0}$ and $(v,w)\in \Ind_{K_P}^K\sigma\vert_{K_M}\times \Ind_{K_P}^K\sigma^\vee\vert_{K_M},$ the function
\begin{align*}
  (s,\chi)\mapsto\langle\HP(|\nu|^s f)(M,\sigma)v,w\rangle_{\Ind_P^G\sigma_\chi}
\end{align*}
extends to a smooth function on $\CC_{>0}\times \Im\Lambda_M$ that is holomorphic in $s$.  For $\mathrm{Re}(s)>r_0$, one has
\begin{align*}
    \langle\HP(|\nu|^s f)(M,\sigma_\chi)v,w\rangle_{\Ind_P^G\sigma_\chi} = Z(\mathrm{Ind}_P^G\sigma_\chi,|\nu|^{s}f,v,w) = Z(\mathrm{Ind}_P^G\sigma_{|\nu|^s\chi},f,v,w).
\end{align*}
By \eqref{S2} $\chi\mapsto Z(\mathrm{Ind}_P^G\sigma_{\chi},f,v,w)$ extends to a holomorphic function on $C_{\rho,M}\times\Im\Lambda_M$. By \eqref{G} we have $|\nu|^{s}\in C_{\rho,M}$ for $\Re(s)>0$. This proves the claim, and \eqref{eq:Z=Z0} holds by the identity principle. The inclusion $|\nu|^{\frac{1}{2}}\mathcal{S}_\rho(G(F))\subseteq \mathcal{S}_\rho^{\mathrm{as}}(G(F))\cap \mathcal{S}^\infty(G(F))$ follows by the definition of \eqref{S2}. 

For the containment $\supseteq$, suppose in addition $f\in \mathcal{S}_\rho^{\mathrm{as}}(G(F))$. By \hyperref[Sas:twist:HC]{Lemma \ref{Sas:twist:HC}} we have $|\nu|^rf\in\mathcal{C}(G(F))$ for $r>-\frac{1}{2}$, so it remains to show $|\nu|^{-\frac{1}{2}}f$ satisfies \eqref{S2}. Let $(M,\sigma)\in \Theta^{K_0}$ and $(v,w)\in \Ind_{K_P}^K\sigma\vert_{K_M}\times \Ind_{K_P}^K\sigma^\vee\vert_{K_M}.$ For $r\ge r_0$ large, one has for $\chi\in \Im \Lambda_M$  
\begin{align*}
    Z_0(\Ind_P^G\sigma_\chi, (|\nu|^{-\frac{1}{2}}f)|\nu|^r,v,w)=\left\langle\HP((|\nu|^{-\frac{1}{2}}f)|\nu|^r)(M,\sigma) v,w\right\rangle_{\Ind_P^G\sigma_\chi} = Z(\Ind_P^G\sigma_{|\nu|^r\chi},|\nu|^{-\frac{1}{2}}f,v,w).
\end{align*}
On the other hand, by \hyperref[Sas:twist:HC]{Lemma \ref{Sas:twist:HC}} we have an identity of meromorphic functions $$Z_0(\Ind_P^G\sigma_\chi, (|\nu|^{-\frac{1}{2}}f)|\nu|^r,v,w)=Z_0(\Ind_P^G\sigma_{|\nu|^r\chi}, |\nu|^{-\frac{1}{2}}f,v,w).$$
It follows by the identity principle that $ Z(\Ind_P^G\sigma_{\chi},|\nu|^{-\frac{1}{2}}f,v,w),$ originally a holomorphic function on $(C\cap \Re \Lambda_M)\times \Im \Lambda_M,$ extends meromorphically to the function $Z_0(\Ind_P^G\sigma_{\chi},|\nu|^{-\frac{1}{2}}f,v,w)$. By the definition of  $\mathcal{S}_\rho^{\mathrm{as}}(G(F)),$ we have
\begin{align*}
    \chi\longmapsto \frac{Z(\Ind_P^G\sigma, |\nu|^{-\frac{1}{2}}f,v,w)}{L(0,\sigma_{\chi},\rho_M)}=\frac{Z_0(\Ind_P^G\sigma, |\nu|^{-\frac{1}{2}}f,v,w)}{L(0,\sigma_{\chi},\rho_M)}\in \CC[\Lambda_M].
\end{align*}
This completes the proof.
\end{proof}

\begin{corollary}\label{Srho:in:Sas} One has $\mathcal{S}_\rho(G(F))\chi \subseteq \mathcal{C}(G(F))$ for $\Re(\chi)\in C_{\rho,G}$. 
\end{corollary}

\begin{proof} This follows from \hyperref[Sas:twist:HC]{Lemma \ref{Sas:twist:HC}} and \hyperref[Snur:HCSchwartz]{Proposition \ref{Snur:HCSchwartz}}.
\end{proof}

\begin{corollary}\label{Srhomodule}
    The space $\mathcal{S}_\rho(G(F))$ is a $G(F)\times G(F)$-module.
\end{corollary}
\begin{proof}
    This follows from \hyperref[Sas:FT+constant]{Lemma \ref{Sas:FT+constant}}, \hyperref[Sinftymodule]{Lemma \ref{Sinftymodule}} and \hyperref[Snur:HCSchwartz]{Proposition \ref{Snur:HCSchwartz}}.
\end{proof}

\begin{corollary}
    For $P=MN\in \mathcal{P},$ the map $(-)^P$ induces a morphism
\begin{align*}
    (-)^{(P)}:\mathcal{S}_\rho(G(F))\longrightarrow\mathcal{S}_{\rho_M}(M(F)).
\end{align*}
\end{corollary}

\begin{proof}
    This follows from \hyperref[Sas:FT+constant]{Lemma \ref{Sas:FT+constant}}, \hyperref[lemma:constant]{Lemma \ref{lemma:constant}} and \hyperref[Snur:HCSchwartz]{Proposition \ref{Snur:HCSchwartz}}.
\end{proof}

\begin{corollary}\label{Sinfty:comp}
    We have $|\nu|^{1/2}\mathcal{S}_\rho(G(F))\subseteq \mathcal{S}_\rho^{\mathrm{as}}(G(F))\cap \mathcal{C}_{\mathrm{cs}}(G(F)).$
\end{corollary}
\begin{proof}
    This follows from  \hyperref[Snur:HCSchwartz]{Proposition \ref{Snur:HCSchwartz}} and the proof of \hyperref[Cfpositive]{Proposition \ref{Cfpositive}} (see also \hyperref[lem:cs:subquo]{Lemma \ref{lem:cs:subquo}} below).
\end{proof}

Choose a finite $W(G,A_0)$-stable generating set $\calB_0\subset X^\ast(M_0)$ of the cone $\overline{C}_{\rho,M_0}$ that contains $\calB_G.$ For $r\in \RR,$ define
\begin{align*}
    S_0(r):=K\left\{ a\in M_0(F): |\tau|(a)\le q^r \textrm{ for all } \tau\in \calB_0\right\}K.
\end{align*}

\begin{lemma}\label{supporttrunc}
    Let $f\in \mathcal{C}(G(F)).$ For any $r\in \RR,$ $\mathbf{1}_{S_0(r)}f\in \mathcal{S}^\infty(G(F)).$ 
\end{lemma}

\begin{proof}
    By \hyperref[Cfpositive]{Proposition \ref{Cfpositive}} it suffices to show for $d>0$ large
    \begin{align*}
        \int_{S_0(r)} \Xi(g)\Xi_\chi(g)\sigma(g)^{-d}dg
    \end{align*}
    is convergent for $\chi$ in a positive cone in $\Re\Lambda_{M_0}.$ Assume $\chi\in X^\ast(M_0)^+_\RR$ so that it is regular, i.e., $s.\chi\neq \chi$ for all $1\neq s\in W(G,A_0).$ Let $c(\chi):=\prod_{\substack{\alpha\in \Sigma}} c_\alpha(\chi),$ where $c_\alpha(\chi)$ is the Harish-Chandra $c$-function
\begin{align*}
    c_\alpha(\chi):=\frac{\big(1-q_{\alpha/2}^{-1/2}q_\alpha^{-1}q^{-\langle \chi,\alpha^\lor\rangle}\big)\big(1+q^{-1/2}_{\alpha/2}q^{-\langle \chi,\alpha^\lor\rangle}\big)}{1-q^{-2\langle \chi,\alpha^\lor\rangle}}.
\end{align*}
We refer one to \cite[\S 1]{Casselman:unramifiedI} for the definition of $q_{\alpha/2}$ and $q_{\alpha}.$ By \cite[Theorem 4.2]{Casselman:unramifiedI}, for $m\in \Omega^+$
\begin{align*}
    \Xi_{\chi}(m)=\frac{1}{\gamma(G|M_0)}q^{-\langle \eta_G,m\rangle}\sum_{s\in W(G,A_0)} c(s.\chi)q^{-\langle s.\chi,m\rangle}.
\end{align*}
For $\chi\in X^\ast(M_0)_\RR^+,$ we have
\begin{align*}
    \Xi_\chi(m)\ll  q^{-\langle\eta_G,m\rangle}\max_{s\in W(G,A_0)}q^{-\langle s.\chi,m\rangle}.
\end{align*}

 By \cite[Lemme II.1.1]{Waldspurger:Plancherel}, there is $d'>0$ such that for all $m\in \Omega^+,$ $\Xi(m)\ll \delta^{1/2}_{P_0}(m)\sigma(m)^{d'}.$ Using the Cartan decomposition, we have
    \begin{align*}
         \int_{S_0(r)} \Xi(g)\Xi_\chi(g)\sigma(g)^{-d}dg&\ll \sum_{\substack{m\in \Omega^+\\
        \langle \tau,m\rangle\ge -r, \,\forall\tau\in \calB_0}} \max_{s\in W(G,A_0)}q^{-\langle s.\chi,m\rangle}\sigma(m)^{-d+d'}\\
        &\le \sum_{\substack{m\in \Omega\\
        \langle \tau,m\rangle\ge -r, \,\forall\tau\in \calB_0}} q^{-\langle \chi,m\rangle}\sigma(m)^{-d+d'}.
    \end{align*}
     Since $\calB_0$ generates $\overline{C}_{\rho,M_0},$ for $d>d'$ large, the sum converges for $\chi\in C_{\rho,M_0}$. Note that $C_{\rho,M_0}\cap X^\ast(M_0)_\RR^+$ is nonempty by \hyperref[intertwining:posCone]{Lemma \ref{intertwining:posCone}}. As both $C_{\rho,M_0}$ and $X^\ast(M_0)_\RR^+$ are open cones that are positive, we conclude that $C_{\rho,M_0}\cap X^\ast(M_0)_\RR^+$ is an open cone that is positive. 
\end{proof}

\begin{theorem}\label{Schwartz=as+cs}
   We have 
\begin{align*}|\nu|^{1/2}\mathcal{S}_\rho(G(F))=\mathcal{S}_{\rho}^{\mathrm{as}}(G(F))\cap  \mathcal{C}_{\mathrm{cs}}(G(F))=\mathcal{S}_\rho^{\mathrm{as}}(G(F))\cap C^\infty_{\mathrm{ac}}(G(F)).
   \end{align*}
   Furthermore, for $f\in \mathcal{S}_\rho(G(F)),$ $\mathrm{supp}(f)\subseteq S_0(r)$ for some $r.$
\end{theorem}

The inclusions 
\begin{align*}
|\nu|^{1/2}\mathcal{S}_\rho(G(F))\subseteq\mathcal{S}_{\rho}^{\mathrm{as}}(G(F))\cap  \mathcal{C}_{\mathrm{cs}}(G(F))\subseteq\mathcal{S}_\rho^{\mathrm{as}}(G(F))\cap C^\infty_{\mathrm{ac}}(G(F))
\end{align*}
follow from \hyperref[Sinfty:comp]{Corollary \ref{Sinfty:comp}} and \hyperref[Sas+csisac]{Corollary \ref{Sas+csisac}}, respectively. We will prove \hyperref[Schwartz=as+cs]{Theorem \ref{Schwartz=as+cs}} by showing that if $f\in\mathcal{S}_{\rho}^{\mathrm{as}}(G(F))\cap  C^\infty_{\mathrm{ac}}(G(F)),$ then $\mathrm{supp}(f)\subseteq S_0(r)$ for some $r$. It then follows from \hyperref[supporttrunc]{Lemma \ref{supporttrunc}} that $f\in \mathcal{S}^\infty(G(F))$ and hence $|\nu|^{-1/2}f\in \mathcal{S}_\rho(G(F))$ by \hyperref[Snur:HCSchwartz]{Proposition \ref{Snur:HCSchwartz}}. We will defer the proof to \S \ref{ssec:Supportproof}.

\begin{corollary}
   Suppose $G$ is unramified. Then $|\nu|^{-\frac{1}{2}}b_\rho\in\mathcal{S}_\rho(G(F))$. 
\end{corollary}

\begin{proof}
    This follows from \cite[Proposition 6.16]{DRS} which shows  $b_\rho\in \mathcal{S}_\rho^{\mathrm{as}}(G(F))\cap C^\infty_{\mathrm{ac}}(G(F)).$ Alternatively $b_\rho\in \mathcal{S}_\rho^{\mathrm{as}}(G(F))\cap \mathcal{C}_{\mathrm{cs}}(G(F))$ by the classification of irreducible unramified representations of $G(F)$. 
\end{proof}

\begin{remark}\label{remark:topology}
    While it is common not to equip Schwartz spaces with any topology for non-Archimedean local fields\footnote{It is also customary to simply give Schwartz spaces the discrete topology so that every topological statement regarding continuity becomes trivial. However, discrete topology is not the correct topology in the functional analysis perspective.}, we point out that there is a natural topology on $\mathcal{S}_{\rho}^{\mathrm{as}}(G(F))\cap  \mathcal{C}_{\mathrm{cs}}(G(F))$  given as follows.
    
    First, we equip $\mathcal{S}_\rho^{\mathrm{as}}(G(F))$ with a topology. Let $K_0\le K$ be a compact open subgroup. For a function $f\in \mathcal{S}_\rho^{\mathrm{as}}(G(F)/\!/K_0),$ $\mathrm{HP}(f)$ is supported on $\Theta^{K_0}.$  Choose a point $(M,\sigma)$ in each connected component $\mathcal{O}$ of $\Theta^{K_0}.$ By \eqref{LambdaG:poly:supp} we may write
    \begin{align*}
        \frac{Z_0(\Ind^G_P \sigma_\chi, f,v,w)}{L(\frac{1}{2},\sigma_\chi,\rho_M)} = \sum_{\alpha\in S} c_\alpha(\sigma) q^{-\langle\chi,\alpha\rangle} 
    \end{align*}
    for some finite subset $S\subseteq \Omega_{M}$. In this case we say $\frac{Z_0(\Ind^G_P \sigma_\chi, f,v,w)}{L(\frac{1}{2},\sigma_\chi,\rho_M)}$ is supported on $S$. Then
    we can write $\mathcal{S}_\rho^{\mathrm{as}}(G(F)/\!/K_0)$ as a countable union of finite-dimensional vector spaces
    \begin{align*}
        \bigcup_{\mathcal{O}\subseteq \Theta^{K_0}}\bigcup_{\substack{S\subseteq \Omega_{M}\\\#S<\infty}} \bigcup_{v,w}\left\{ f\in \mathcal{S}_\rho^{\mathrm{as}}(G(F)/\!/K_0):  \frac{Z_0(\Ind^G_P \sigma_\chi, f,v,w)}{L(\frac{1}{2},\sigma_\chi,\rho_M)}\in \CC[\Lambda_M] \textrm{ is supported on $S$}\right\},
    \end{align*}
    where $v,w$ range over fixed (finite) bases of $\left(\mathrm{Ind}_{K_P}^K \sigma|_{K_M}\right)^{K_0}$ and $\left(\mathrm{Ind}_{K_P}^K \sigma^\lor|_{K_M}\right)^{K_0}$ respectively.
Therefore, we can equip $\mathcal{S}_\rho^{\mathrm{as}}(G(F)/\!/K_0)$ with the inductive limit topology. This topology is independent of the choices of the base point of each orbit. The space $\mathcal{S}_\rho^{\mathrm{as}}(G(F)/\!/K_0)$ is an algebraic topological space in the sense of \cite[Definition 4.1.1]{Li:zeta}.

Therefore, $\mathcal{S}_{\rho}^{\mathrm{as}}(G(F))=\bigcup_{K_0\le K} \mathcal{S}_\rho^{\mathrm{as}}(G(F)/\!/K_0)$ with the inductive limit topology is also an algebraic topological space. We equip $\mathcal{S}_\rho^{\mathrm{as}}(G(F))\cap \mathcal{C}_{\mathrm{cs}}(G(F))$ with the subspace topology inherited from $\mathcal{S}_{\rho}^{\mathrm{as}}(G(F))$. It is a closed subspace of $\mathcal{S}_{\rho}^{\mathrm{as}}(G(F))$. Note that this topology is strictly finer than the subspace topology inherited from $\mathcal{C}(G(F)).$ For instance, $C^\infty_c(G(F))$ is a closed algebraic topological subspace of $\mathcal{S}_\rho^{\mathrm{as}}(G(F))\cap \mathcal{C}_{\mathrm{cs}}(G(F)).$
\end{remark}

Due to \hyperref[Snur:HCSchwartz]{Proposition \ref{Snur:HCSchwartz}} and \hyperref[Schwartz=as+cs]{Theorem \ref{Schwartz=as+cs}}, in the rest of the paper, except for \S \ref{ssec:Supportproof}, we will constantly identify $Z_0$ with $Z$ for $f\in \mathcal{S}_{\rho}^{\mathrm{as}}(G(F))\cap  \mathcal{C}_{\mathrm{cs}}(G(F))=\mathcal{S}_\rho^{\mathrm{as}}(G(F))\cap C^\infty_{\mathrm{ac}}(G(F))$.

\subsection{Fourier transforms}

In this paper, we do not prove the Schwartz space $\mathcal{S}_\rho^{\mathrm{as}}(G(F))\cap C^\infty_{\mathrm{ac}}(G(F))$ is stable under the Fourier transform $\mathcal{F}_\rho$. Instead, we show that being stable under $\mathcal{F}_\rho$ is equivalent to the following natural assumption on $\gamma$-factors of non-supercuspidal discrete series and $\gamma$-factors of their supercuspidal supports.

\begin{enumerate}[label=(LLC\text{\arabic*}), ref=LLC\text{\arabic*}] \setcounter{enumi}{6}
    \item\label{LLC:gamma} For $P=MN\in\mathcal{P}$, $\sigma\in\Pi_0(M)$ and $\pi\in\Pi_2(G)$ such that $\pi\leq \Ind_P^G(\sigma_\lambda)$ for some $\lambda\in \Lambda_M$, one has $\Lambda_M\ni\chi\mapsto \gamma(\tfrac{1}{2},\sigma_\chi,\rho_M,\psi)$ is holomorphic at $\chi = \lambda$ and 
    \begin{align*}
        \gamma(\tfrac{1}{2},\pi,\rho,\psi) = \gamma(\tfrac{1}{2},\sigma_\chi,\rho_M,\psi)\vert_{\chi = \lambda}.
    \end{align*}
\end{enumerate}

\begin{lemma}[{\cite[Proposition 3.4]{DRS}}] \label{lem:LLC7:subquo}   Assume \eqref{LLC:gamma} holds for $G$. Then for $P=MN\in\mathcal{P},\,\sigma\in\Pi_0(M)$  and $\pi\in\Pi_2(G)$ such that $\pi$ is a subquotient of $\Ind_P^G\sigma_\lambda$ for some $\lambda\in \Lambda_M$, one has $\Lambda_M\ni\chi\mapsto \gamma(\tfrac{1}{2},\sigma_\chi,\rho_M,\psi)$ is holomorphic at $\chi = \lambda$ and 
    \begin{align*}
        \gamma(\tfrac{1}{2},\pi,\rho,\psi) = \gamma(\tfrac{1}{2},\sigma_\chi,\rho_M,\psi)\vert_{\chi = \lambda}.
    \end{align*}
\end{lemma}
\begin{proof}  By \cite[Corollaire VI.5.4]{Representation-p-adic}, there exists $Q\in \mathcal{P}(M)$ such that $\pi$ is a subrepresentation of $\Ind_{Q}^G \sigma_\lambda$. Hence the claim follows from \eqref{LLC:gamma}.
\end{proof}

\begin{lemma}\label{lem:cs:subquo} Let $f\in \mathcal{S}_\rho(G(F))$. Let $M'\le M\in\mathcal{M}^{\mathrm{std}}$ and $(\sigma',\sigma)\in\Pi_0(M')\times \Pi_2(M)$ such that $\sigma$ is a subquotient of $\mathrm{Ind}_{P'\cap M}^{M} \sigma'_{\lambda}$ for some $\lambda\in \Lambda_{M'}$. Let $E$ be any $M(F)$-submodule of $ \mathrm{Ind}_{P'\cap M}^{M} \sigma_{\lambda}'$ such that there is a quotient map $E\to \sigma.$ View $E$ as a $K_M$-submodule of $\mathrm{Ind}_{K_{P'\cap M}}^{K_M} \sigma'|_{K_{M'}}.$ Then for $(v,w)\in \Ind_{K_P}^K\sigma\vert_{K_M}\times \Ind_{K_P}^K\sigma^\vee\vert_{K_M},$ and any lift $(v',w')\in \mathrm{Ind}_{K_P}^{K} E\times \mathrm{Ind}_{K_{P'}}^K \sigma'^\vee|_{K_M'}$ of $(v,w)$, one has an identity of meromorphic functions
    \begin{align*}
            Z(\mathrm{Ind}_P^G \sigma_\chi, f,v,w)=Z(\mathrm{Ind}_{P'}^G \sigma'_{\lambda\chi}, f,v',w').
    \end{align*}
\end{lemma}
\begin{proof} 

Let $v'\in \mathrm{Ind}_{K_P}^K E$ be a lift of $v$ and $w'\in \mathrm{Ind}_{K_{P'}}^K \sigma'^\vee|_{K_M'}$ be a lift of $w$ via
\begin{align*}
    \Ind_{K_{P'}}^K\sigma'^\lor|_{K_{M'}}\longrightarrow (\mathrm{Ind}_{K_P}^K E)^\lor \hookleftarrow \mathrm{Ind}_{K_P}^K \sigma^\lor|_{K_M}.
\end{align*}
Note that the natural pairing between $\mathrm{Ind}_{K_{P'}}^K\sigma'|_{K_{M'}}$ and $\mathrm{Ind}_{K_{P'}}^K\sigma'^\lor|_{K_{M'}}$ is perfect. Therefore
\begin{align*}
    \langle (\mathrm{Ind}_P^G\sigma_\chi)(g) v,w\rangle=\langle (\mathrm{Ind}_{P'}^G\sigma'_{\chi\lambda'})(g)v',w'\rangle
\end{align*}
for all $g\in G(F),$ so
   \begin{align*}
            Z(\mathrm{Ind}_P^G \sigma_\chi, f,v,w)=Z(\mathrm{Ind}_{P'}^G \sigma'_{\lambda\chi}, f,v',w')
    \end{align*}
for $\mathrm{Re}(\chi)$ in a positive cone $C$ in $\Re\Lambda_M$, where both integrals are absolutely convergent. The assertion then follows from the identity principle.
\end{proof}

\begin{proposition}\label{LLC7=Fourierstable}
The following are equivalent.
\begin{enumerate}
    \item \eqref{LLC:gamma} holds for all semi-standard Levi subgroups $M$ of $G$.
    \item $\mathcal{S}^{\mathrm{as}}_{\rho}(G(F))\cap \mathcal{C}_{\mathrm{cs}}(G(F))$ is stable under $\mathcal{F}_\rho$.
\end{enumerate}
\end{proposition}
\begin{proof} By \hyperref[Schwartz=as+cs]{Theorem \ref{Schwartz=as+cs}} $|\nu|^{1/2}\mathcal{S}_\rho(G(F))=\mathcal{S}^{\mathrm{as}}_{\rho}(G(F))\cap \mathcal{C}_{\mathrm{cs}}(G(F)),$ so we can apply \hyperref[lem:cs:subquo]{Lemma \ref{lem:cs:subquo}} to functions in $\mathcal{S}^{\mathrm{as}}_{\rho}(G(F))\cap \mathcal{C}_{\mathrm{cs}}(G(F))$. We also identify the functionals $Z$ and $Z_0$.

Assume (1). Let $f\in \mathcal{S}^{\mathrm{as}}_{\rho}(G(F))\cap \mathcal{C}_{\mathrm{cs}}(G(F)).$ Let $(M,\sigma)\in\widetilde{\Temp}_\Ind(G)$ and $(M',\sigma')\in\widetilde{\Temp}_{\Ind,0}(M)$ such that $\sigma\le \mathrm{Ind}_{P'\cap M}^{M} \sigma'_{\lambda}$ for some $\lambda\in \Lambda_{M'}$. By \hyperref[Sas:FE:near0]{Corollary \ref{Sas:FE:near0}}, we have identities of meromorphic functions on $\Lambda_M$
\begin{align*}
    Z(\mathrm{Ind}_P^G \sigma_\chi, \mathcal{F}_\rho f,v,w) &= \gamma(\tfrac{1}{2},(\sigma_\chi)^\vee,\rho_M,\psi)Z(\mathrm{Ind}_P^G (\sigma_\chi)^\vee, f,w,v)\\
    &= \gamma(\tfrac{1}{2},(\sigma^\vee)_{\chi^{-1}},\rho_M,\psi)Z(\mathrm{Ind}_P^G (\sigma^\vee)_{\chi^{-1}}, f,w,v)\\
    &=\gamma(\tfrac{1}{2},(\sigma'^\vee)_{\chi^{-1}\lambda^{-1}},\rho_{M'},\psi)Z(\mathrm{Ind}_{P'}^G (\sigma'^\vee)_{\chi^{-1}\lambda^{-1}}, f,w',v)\\
    &=\gamma(\tfrac{1}{2},(\sigma'_{\chi\lambda})^\vee,\rho_{M'},\psi)Z(\mathrm{Ind}_{P'}^G (\sigma'_{\chi\lambda})^\vee, f,w',v)\\
    &= Z(\Ind_{P'}^G\sigma'_{\chi\lambda},\mathcal{F}_\rho f,v,w').
\end{align*}
Here in the third equality we apply \hyperref[lem:LLC7:subquo]{Lemma \ref{lem:LLC7:subquo}} and \hyperref[lem:cs:subquo]{Lemma \ref{lem:cs:subquo}} to the contragredient $\sigma^\vee$ realized as a quotient of $\Ind_{P'\cap M}^M(\sigma'^\vee)_{\lambda^{-1}}$. This shows $\mathcal{F}_\rho f$ satisfies \eqref{CS2}, so that $\mathcal{F}_\rho f\in \mathcal{S}^{\mathrm{as}}_{\rho}(G(F))\cap \mathcal{C}_{\mathrm{cs}}(G(F))$. This proves (2).

Conversely, assume (2) holds. Let $M\in \mathcal{M}^{\mathrm{std}},$ $\sigma\in \Pi_2(M)$ and $(M',\sigma')\in\widetilde{\Temp}_\Ind(M)$ such that $\sigma\leq \Ind_{P'\cap M}^M\sigma'_\lambda$ for some $\lambda\in\Lambda_{M'}$. Pick $f\in C_c^\infty(G(F))$ such that $(\Ind_P^G\sigma)(f)\neq 0.$ Then there exists $v,w$ such that $Z(\Ind_P^G\sigma_\chi,f,v,w)$ is a nonzero function. Then by \eqref{CS2}, $Z(\Ind_{P'}^G\sigma'_{\chi\lambda},f,v,w')$ is a nonzero function for any lift $w'$ of $w$. By assumption $\mathcal{F}_\rho f\in\mathcal{S}^{\mathrm{as}}_{\rho}(G(F))\cap \mathcal{C}_{\mathrm{cs}}(G(F))$, so we have identities of meromorphic functions on $\Lambda_M$
\begin{align*}
    Z(\Ind_P^G(\sigma_\chi)^\vee,\mathcal{F}_\rho f,w,v) = Z(\Ind_{P'}^G(\sigma'_{\chi\lambda})^\vee,\mathcal{F}_\rho f,w',v)=\gamma(\tfrac{1}{2},\sigma'_{\chi\lambda},\rho_{M'},\psi)Z(\Ind_{P'}^G\sigma'_{\chi\lambda},f,v,w'),
\end{align*}
and
\begin{align*}
    Z(\Ind_P^G(\sigma_\chi)^\vee,\mathcal{F}_\rho f,w,v) = \gamma(\tfrac{1}{2},\sigma_\chi,\rho_M,\psi)Z(\Ind_P^G\sigma_\chi,f,v,w) = \gamma(\tfrac{1}{2},\sigma_\chi,\rho_M,\psi)Z(\Ind_{P'}^G\sigma'_{\chi\lambda},f,v,w')
\end{align*}
by \hyperref[Sas:FE:near0]{Corollary \ref{Sas:FE:near0}} and \hyperref[lem:cs:subquo]{Lemma \ref{lem:cs:subquo}}. We conclude that $\gamma(\tfrac{1}{2},\sigma_\chi,\rho_M,\psi) = \gamma(\tfrac{1}{2},\sigma'_{\chi\lambda},\rho_{M'},\psi)$ as meromorphic functions, and they are holomorphic in a neighborhood of $\chi=1$. This implies \eqref{LLC:gamma} using the fact that for general $M\in \mathcal{M},$ $\sigma\in \Pi_2(M)$ and $s\in W(G,A_0),$ $\gamma(\tfrac{1}{2},\sigma_\chi,\rho_M,\psi)=\gamma(\tfrac{1}{2},s.\sigma_\chi,\rho_{s.M},\psi)$ by \eqref{LLC:para}.
\end{proof}

Using the Bernstein-Zelevinsky classification, it is shown in \cite[Proposition 3.3]{DRS} that \eqref{LLC:gamma} holds for general linear groups. One can probably establish \eqref{LLC:gamma} for quasi-split classical groups by the study of cuspidal supports of discrete series in \cite{Moeglin:classicaldiscrete, Moeglin-Tadic,Mok:quasisplitunitary,Xu:quasisplitclassical}. We leave this for future work.

\begin{remark}
    Assume \eqref{LLC:gamma}. It follows from \hyperref[Sas:FE:near0]{Corollary \ref{Sas:FE:near0}} that $\mathcal{F}_\rho$ is a continuous automorphism of $\mathcal{S}^{\mathrm{as}}_{\rho}(G(F))\cap \mathcal{C}_{\mathrm{cs}}(G(F))$ with respect to the topology introduced in \hyperref[remark:topology]{Remark \ref{remark:topology}}.
\end{remark}

\subsection{Proof of Theorem \ref{Schwartz=as+cs}}\label{ssec:Supportproof}

Let $f\in\mathcal{S}_{\rho}^{\mathrm{as}}(G(F))\cap  C^\infty_{\mathrm{ac}}(G(F)).$ Let $K_0\le K$ be a compact open subgroup such that $f$ is bi-$K_0$-invariant. Recall 
\begin{align*}
    \Theta^{K_0} = \{(M,\sigma)\in\widetilde{\Temp}_\Ind(G): (\Ind_P^G\sigma)^{K_0}\neq 0\}
\end{align*}
has finitely many connected components. Let $\mathcal{B}_0\subset X^*(M_0)$ be a finite $W(G,A_0)$-invariant generating set of the cone $\overline{C}_{\rho,M_0}$. Recall in \S\ref{subsec:Cartan} we view $\Omega=M_0(F)/M_0(F)^1$ as a subset of $\Hom_\ZZ(X^*(M_0),\ZZ)$. To prove \hyperref[Schwartz=as+cs]{Theorem \ref{Schwartz=as+cs}}, we need to show

\begin{proposition}\label{prop:key} There exists $r\in \RR$ such that for $\alpha\in\Omega_G$ and $g\in KmK$ with $m\in\Omega$, one has $f\mathbf{1}_\alpha(g)\neq 0$ only if
\begin{align*}
   \langle \tau ,m\rangle\ge -r\quad \forall\tau\in \calB_0.
\end{align*}
\end{proposition}

The proof of the proposition will occupy this subsection. Our proof is a modification of the proof of \cite[Proposition 2.1]{Heiermann:HeckePlan}. Note that we have a priori bound for $\alpha$: By \hyperref[Sas:f1alpha:polynomial]{Corollary \ref{Sas:f1alpha:polynomial}} and the fact that $p_{M_0\leq G}(\mathcal{B}_0)\subseteq \overline{C}_{\rho,G}$, there exists a constant $h_0$ such that $f\mathbf{1}_\alpha\neq 0$ only if
\begin{align}\label{support:alpha:bound}
    \langle\tau,\alpha\rangle\geq h_0\,\,\,\,\forall \tau\in\mathcal{B}_0.
\end{align}

To pin down notations, we start with a review of \cite{Heiermann:HeckePlan}. For $P=MN\in \mathcal{P},$ a $\Lambda_M$-orbit $\mathcal{O}$ in $\Pi_0(M)_\CC$ and $\sigma\in \mathcal{O},$ we put
\begin{align*}
    \varphi_{\alpha,P,\calO}(\sigma):=(\mathrm{Ind}_P^G\sigma) (f\mathbf{1}_\alpha).
\end{align*}
Then $\varphi_{\alpha,P,\calO}$ is a polynomial section on $\mathcal{O}$ with coefficients in $\mathrm{End}_\CC( \Ind_{K_P}^K \sigma|_{K_M})$. 
By \cite[Part B]{Heiermann:HeckePlan}, specifically the proof of Lemma 4.3 and \S 5 in loc.~cit., we can identify $\varphi_{\alpha, P,\mathcal{O}}$ as a polynomial section with coefficients in $\mathrm{Ind}_{K_{\overline{P}}}^K \sigma|_{K_M}\otimes \mathrm{Ind}_{K_{P}}^K \sigma^\lor |_{K_M}.$ Explicitly, choose $K'\le K_0$ that admits an Iwahori factorization with respect to $P(F)$. Then there exists a polynomial section $\xi_0$ on $\calO$ with coefficients in $\mathrm{Ind}_{K_{\overline{P}}}^K \sigma|_{K_M}\otimes \mathrm{Ind}_{K_{P}}^K \sigma^\lor |_{K_M}$ such that
\begin{align*}
    \varphi_{\alpha, P,\mathcal{O}}(\sigma_\chi)=\sum_{s\in W(M,\mathcal{O})} \bigg((J_{P|\overline{s.P}}(\sigma_\chi)\circ \lambda(s))\otimes (J_{P|s. P}(\sigma_\chi^\lor)\circ \lambda(s))\bigg)\bigg((\mathrm{Ind}_{\overline{P}}^G (s^{-1}.\sigma_\chi))(f\mathbf{1}_\alpha)\otimes \mathrm{id}\bigg)\xi_0(s^{-1}.\sigma_\chi),
\end{align*}
where $W(M,\mathcal{O})=\{s\in W(G|M):s.\mathcal{O}=\mathcal{O}\}$. The section $\xi_0$ only depends on $\calO$ and $K'.$ 

Let 
\begin{align*}
\xi(\sigma_\chi):=\bigg((\mathrm{Ind}_{\overline{P}}^G \sigma_\chi)(f\mathbf{1}_\alpha)\otimes \mathrm{id}\bigg)\xi_0(\sigma_\chi).
\end{align*}
Define a linear map
\begin{align*}
E   ^{G}_{P,\sigma_\chi}:\mathrm{Ind}_{K_P}^K \sigma|_{K_M}\otimes \mathrm{Ind}_{K_{P}}^K \sigma^\lor |_{K_M}&\longrightarrow C^\infty(G(F))\\
    v\otimes w&\longmapsto \left(g\mapsto \langle (\mathrm{Ind}_P^G\sigma_\chi)(g)v,w\rangle\right).
\end{align*}
Consider the function on $G(F)$ given by
\begin{align}\label{eq:falphaPO}
    f_{\varphi_{\alpha,P,\calO}}(g):=\int_{\Re(\chi)=\mu\gg_P \,0} E^{G}_{P,\sigma_\chi}\big((J_{\overline{P}|P}(\sigma_\chi)^{-1}\otimes \mathrm{id})\xi(\sigma_\chi)\big)(g^{-1}) d\Im\chi.
\end{align}
Here $\mu\gg_P 0$ indicates that we have chosen $\mu\in X^\ast(M)_\RR^+$ sufficiently positive with respect to $P$ so that $\langle \mu, \gamma^\lor\rangle>R$ for some large $R$ for all $\gamma\in \Sigma(P).$ By \hyperref[lem:J]{Lemma \ref{lem:J}} and the residue theorem, the integral is independent of the choice of $\mu,$ so the function $f_{\varphi_{\alpha,P,\calO}}$ is well-defined. It is shown in \cite{Heiermann:HeckePlan} that this integral defines a function in $C^\infty_c(G(F))$ and only depends on $\varphi_{\alpha,P,\mathcal{O}}$. Let $[\mathcal{O}]:=\bigcup_{s\in W(G,A_0)}s.\mathcal{O}.$ There exist explicit constants $c([\mathcal{O}]),$ defined in \cite[\S 3.2]{Heiermann:HeckePlan}, such that
\begin{align*}
    f\mathbf{1}_\alpha=\sum_{[\mathcal{O}]}c([\mathcal{O}])\sum_{(P',\mathcal{O'})}f_{\varphi_{\alpha,P',\calO'}}.
\end{align*}
Here the inner sum is taken over all pairs $(P',\calO')$ such that $ \mathcal{O'}=s.\mathcal{O}$ and $P\in \mathcal{P}(s.M)$ for some $s\in W(G,A_0).$ Note that $f_{\varphi_\alpha,P',\mathcal{O'}} = 0$ for all but finitely many $[\mathcal{O}]$, so the sum is essentially finite. 

To prove \hyperref[prop:key]{Proposition \ref{prop:key}}, it suffices to prove the claimed support for each function $f_{\varphi_\alpha,P',\mathcal{O}'}$. By \hyperref[alpha:supp:1]{Lemma \ref{alpha:supp:1}.(ii)}, we may assume that
\begin{align*}
    \xi(\sigma_\chi)=p(\chi)\bar{v}\otimes w
\end{align*}
for some $(\bar{v},w)\in (\mathrm{Ind}_{K_{\overline{P}}}^K\sigma|_{K_M})^{K'}\times (\mathrm{Ind}_{K_P}^K\sigma^\lor |_{K_M})^{K'},$ and
\begin{align}\label{eq:p:expansion}
    p(\chi)=\sum_{\beta\in \Omega_M} c_\beta q^{-\langle \chi,\beta\rangle}
\end{align}
is a polynomial with $c_\beta\neq 0$ only if $\beta\vert_{X^*(G)} = \alpha$ and $  \langle \tau,\beta\rangle\geq -c$ for all $\tau\in \calB_0$ for some $c\in \ZZ$. 

 We need to show that there exists a constant $h$, independent of $\alpha$, such that the support of the integral \eqref{eq:falphaPO} in  $\alpha G(F)^1$ is contained in
\begin{align*}
    K \{m\in M_0(F)\cap \alpha G (F)^1: \langle \tau,H_{M_0}(m)\rangle\ge h\quad  \forall \tau\in \calB_0 \}K.
\end{align*}
For ease of notation, we write $H_0$ for $H_{M_0}$ in below. Note that $A_0(F)M_0(F)^1\le M_0(F)$ and $G(F)^1Z_G(F)\le G(F)$ are of finite index. Using this fact together with $K$-finiteness and the Cartan decomposition, it suffices to consider $\alpha\in Z_G(F)$ and show the support of the function defined on $A_0^+:=A_0(F)\cap M_0(F)^+\cap G(F)^1$ by 
\begin{align}\label{eq:cptsupport1}
    a\mapsto \int_{\mathrm{{Re}}(\chi)= \mu\gg_P \, 0} \chi(\alpha)^{-1}p(\chi)\langle (\Ind_P^G\sigma_\chi)(a)J_{\overline{P}|P}(\sigma_\chi)^{-1}\bar{v},w\rangle d\Im\chi
\end{align}
is contained in 
 \begin{align}\label{eq:criterion}
     \{a\in A_0^+: \langle \tau,\alpha +H_0(w_0.a^{-1})\rangle\ge h\quad  \forall \tau\in \calB_0 \},
 \end{align}
where $w_0\in W(G,A_0)$ is the long Weyl element.

For $\theta\subseteq \Delta$ and $t,t'\in \ZZ_{\ge 0},$ define
\begin{align*}
    A_0^+(\theta,t,t'):=\left\{ a\in A_0^+: \substack{\langle \gamma,H_0(a)\rangle\le t \quad\forall\gamma\in\theta,\quad\\ \langle \gamma,H_0(a)\rangle >t'\quad\forall\gamma\in \Delta-\theta}\right\}.
\end{align*}
Let $A_0^+(\theta,t):=A_0^+(\theta,t,t).$ Then we have $A_0^+=\bigcup\limits_{\theta\subseteq\Delta}A_0^+(\theta,t)$ for any $t\ge 0$. 


\begin{lemma}\label{lemma:support} For each $\theta\subseteq\Delta$, there exists a function $h_\theta:\mathbb{Z}_{\geq 0}\to \ZZ_{\ge 0}$ such that
\begin{enumerate}
    \item $h_\theta(t)\geq t$ for $t\in\mathbb{Z}_{\geq 0}$, and 
    \item there exists a function $h_\theta':\mathbb{Z}_{\geq 0}\to\mathbb{R}$ such that if $a\in A_0^+(\theta,t,h_\theta(t))$, then \eqref{eq:cptsupport1} is nonzero only if 
\begin{align*}
    \langle \tau,\alpha+ H_0(w_0.a^{-1})\rangle\ge h'_{\theta}(t)\quad  \forall \tau\in \calB_0.
\end{align*}
\end{enumerate}
Furthermore, the functions $h_\theta$ and $h_\theta'$ are independent of $\alpha$.
\end{lemma}

Assuming the lemma, we continue to finish the proof of \hyperref[prop:key]{Proposition \ref{prop:key}}. Let $t\in\mathbb{Z}_{\geq 0}$ be fixed. If $a\in A_0^+(\theta,t)$ such that \eqref{eq:cptsupport1} is nonzero, then either
\begin{itemize}
    \item $a\in A_0^+(\theta,t,h_\theta(t))$, in which case by \hyperref[lemma:support]{Lemma \ref{lemma:support}.(2)} we have  $\langle\tau,\alpha+H_0(w_0.a^{-1})\rangle\geq h_\theta'(t)$ for all $\tau\in\mathcal{B}_0$, or
    \item $a\in A_0^+(\theta,t) - A_0^+(\theta,t,h_\theta(t))$, in which case the set $\theta':=\{\gamma\in\Delta : \langle\gamma,H_0(a)\rangle\leq h_\theta(t)\}$ contains $\theta$ as a proper subset, and $a\in A_0^+(\theta',h_\theta(t))$ by \hyperref[lemma:support]{Lemma \ref{lemma:support}.(1)}.
\end{itemize}
Therefore,
\begin{align*}
    \supp\eqref{eq:cptsupport1}\subseteq \left(\bigcup_{\theta\subseteq\Delta}\bigcup_{\substack{\theta'\subseteq\Delta\\\theta\subsetneq\theta'}}A_0^+(\theta',h_\theta(t))\right) \cup \{a\in A_0^+: \langle \tau,\alpha+H_0(w_0.a^{-1})\rangle\ge \min_\theta h'_{\theta}(t)\quad  \forall \tau\in \calB_0 \}.
\end{align*}
Repeating the argument, we find constants $t_1,h_1$, independent of $\alpha$, such that
\begin{align*}
    \supp\eqref{eq:cptsupport1} \subseteq A_0^+(\Delta,t_1) \cup \{a\in A_0^+: \langle \tau,\alpha+H_0(w_0.a^{-1})\rangle\ge h_1\quad  \forall \tau\in \calB_0 \}.
\end{align*}
Since $\mathcal{B}_0$ is finite and $A_0^+(\Delta,t_1)$ is compact, we can find $h_2\in\RR$ such that
\begin{align*}
    A_0^+(\Delta,t_1) &\subseteq \{a\in A_0^+ : \langle\tau,H_0(w_0.a^{-1})\rangle \geq h_2\,\,\,\,\forall\tau\in\mathcal{B}_0\}\\
    &\!\!\!\!\overset{\eqref{support:alpha:bound}}{\subseteq}\{a\in A_0^+ : \langle\tau,\alpha+H_0(w_0.a^{-1})\rangle \geq h_0+h_2\,\,\,\,\forall\tau\in\mathcal{B}_0\}.
\end{align*}
In conclusion
\begin{align*}
    \supp\eqref{eq:cptsupport1} \subseteq \{a\in A_0^+ : \langle\tau,\alpha+H_0(w_0.a^{-1})\rangle \geq \min\{h_0+h_2,h_1\}\,\,\,\,\forall\tau\in\mathcal{B}_0\}.
\end{align*}
This justifies \eqref{eq:criterion}, whence finishing the proof of \hyperref[prop:key]{Proposition \ref{prop:key}}.\qed

\begin{proof}[Proof of Lemma \ref{lemma:support}] Fix $\theta$. Let $(P',M'):=(P_\theta,M_\theta)$ so that $\Delta_{M_\theta}=\theta$ and $P_\theta$ is standard. We can choose a compact subset $U_\theta^{-1}\subset A_0^+$ such that $A_0^+\subset A_{M'}(F)(A_{0}\cap M'^{\mathrm{der}})(F)U_\theta;$ we will write for $a\in A_0^+,$ $a=a_\theta a'c_a.$ Note that 
\begin{align*}
    \langle \gamma,H_0(a')\rangle\le 0\quad \forall \gamma\in \Delta-\theta,
\end{align*}
and if $a\in A_0^+(\theta,t,t')$ then $a_\theta\in A_{M'}(F)\cap A_0^+(\theta,t,t').$ For a given $t\ge 0$ and  $a\in A_0^+(\theta,t)$, for $\gamma\in \theta$
\begin{align*}
    \langle\gamma, H_0(a'c_a)\rangle=\langle\gamma,H_0(aa_\theta^{-1})\rangle=\langle \gamma,H_0(a)\rangle\in [0,t].
\end{align*}
Therefore, the set $\{ a'c_a:a\in A_0^+(\theta,t)\}$ is compact.

Let 
\begin{align*}
   W:=W(M',A_0)\backslash \{s\in W(G,A_0): s. M\le M'\}.
\end{align*}
We identity $W$ as a subset of $W(G,A_0)$ by choosing for each coset the representative of minimal length. For $s\in W$ (if nonempty), define $P'_s,\widetilde{P}'_s\in \mathcal{P}(M)$ by
\begin{align*}
    P'_s:=((s^{-1}. M')\cap P)(s^{-1}.N'),\quad\quad \widetilde{P}'_s:=((s^{-1}.M')\cap P)(s^{-1}.\overline{N}').
\end{align*}
For $(v,w)\in \mathrm{Ind}_{K_{P}}^K\sigma|_{K_M}\times \mathrm{Ind}_{K_P}^K\sigma^\lor |_{K_M}$ and $a\in A_0(F),$ define the function on $\chi\in \Lambda_M$
\begin{align*}
    c_{P'|P}(\sigma_\chi,s)(v\otimes w)(a):=\bigg\langle (\mathrm{Ind}_{s.P\cap M'}^{M'}s.\sigma_\chi)(a)\bigg(\lambda(s)J_{P'_{s}|P}(\sigma_\chi)v\bigg)\bigg|_{M'},\bigg(\lambda(s)J_{\widetilde{P}'_s|P}((\sigma_\chi)^\lor)w\bigg)\bigg|_{M'}\bigg\rangle.
\end{align*}
By \hyperref[lem:J]{Lemma \ref{lem:J}}, this is a rational function on $\Lambda_M.$ Furthermore, there are finitely many $R_i\in\CC$ such that for any $a,$ poles of $c_{P'|P}(\sigma_\chi,s)(v\otimes w)(a)$ are contained in the image of the hyperplanes $\langle \chi,\gamma_i^\lor\rangle =R_i$ where $\gamma_i^\lor\in \Sigma(P)\cap \Sigma(P'_s).$ By \cite[Th\'eor\`eme 1.3.1]{Heiermann:HeckePlan}, for $t'=h_\theta(t)$ sufficiently large and $a\in A_0^+(\theta,t,h_\theta(t))$
\begin{align*}
    &\langle (\Ind_P^G\sigma_\chi)(a)J_{\overline{P}|P}(\sigma_\chi)^{-1}\bar{v},w\rangle=\gamma(G|M')^{-1}\delta_{P'}^{1/2}(a)\sum_{s\in W} c_{P'|P}(\sigma_\chi,s)(J_{\overline{P}|P}(\sigma_\chi)^{-1}\bar{v}\otimes w)(a).
\end{align*}

Let $a\in A_0^+(\theta,t,h_\theta(t))$.  The identity implies  \eqref{eq:cptsupport1} is zero if $W$ is empty. Hence we assume $W$ is nonempty and fix $s\in W.$ Then we are reduced to study the support of the function
\begin{align}\label{eq:cptsupport2}
    a\mapsto \int_{\mathrm{{Re}}(\chi)= \mu\gg_P0} \chi(\alpha)^{-1}p(\chi)c_{P'|P}(\sigma_\chi,s)(J_{\overline{P}|P}(\sigma_\chi)^{-1}\bar{v}\otimes w)(a)d\Im\chi.
\end{align}
Define a function $r_a$ on $\Lambda_M$ by
\begin{align*}
    r_a(\chi):=c_{P'|P}(\sigma_\chi,s)(J_{\overline{P}|P}(\sigma_\chi)^{-1}\bar{v}\otimes w)(a'c_a).
\end{align*}
Since the set $\{ a'c_a:a\in A_0^+(\theta,t)\}$ is compact, the set $\{r_a:a\in  A_0^+(\theta,t,h_\theta(t))\}$ is a finite set of rational functions on $\Lambda_M$. It follows from  \hyperref[lem:J]{Lemma \ref{lem:J}} and \cite[Lemme 1.3.2]{Heiermann:HeckePlan} that poles of $r_a(\chi)$ lie on the image of finitely many hyperplanes of the form $\langle \chi,\gamma^\lor\rangle=R$ where $\gamma\in \Sigma(P)\cap \Sigma(P'_s).$ 

Let $\chi_\sigma$ be the central character of $\sigma$. By a change of variable $\chi\mapsto s^{-1}.\chi$,  we can write \eqref{eq:cptsupport2} as
\begin{align}\label{eq:cptsupport3}
    a\mapsto (s.\chi_{\sigma})(a_\theta)\int_{\Re(\chi)=s.\mu} \chi(a_\theta)\chi(\alpha)^{-1}p(s^{-1}.\chi)r_a(s^{-1}.\chi) d\Im\chi.
\end{align}
Consider the decomposition $\Re\Lambda_{s.M}=\Re\Lambda_{M'}\times\Re\Lambda^{M'}_{s. M}.$ Let $\mu'\in X^\ast(M')_\RR^+$. Let $\gamma\in \Sigma(s.P)\cap \Sigma(s. P'_s)$. Then either $\gamma\in \Sigma((s. P)\cap M')$ or $\gamma|_{A_{M'}}\in \Sigma(P').$ In the former case, $\langle \mu',\gamma^\lor\rangle=0;$ in the latter case, $\langle \mu',\gamma^\lor\rangle>0.$ Thus we can shift the contour, and the value \eqref{eq:cptsupport3} remains the same for $a\in A_0^+(\theta,t,h_\theta(t))$ if one replaces $s.\mu$ with $s.\mu+\mu'$ (c.f. \cite[\S1.2.2]{Heiermann:HeckePlan}).

Let $\mathrm{conv}(\mathcal{B}_0)\subseteq X^*(M_0)_\RR$ denote the convex hull of $\mathcal{B}_0$. Let $\tilde{\tau}\in \mathrm{conv}(\mathcal{B}_0)\cap X^*(M')_\RR^+$. For $\beta\in \Omega_M$ in the support of \eqref{eq:p:expansion}, one has $ \langle \tilde{\tau},s.\beta-\alpha\rangle\ge -c-\langle \tilde{\tau},\alpha\rangle.$ Therefore, if $\langle \tilde{\tau},H_0(a_\theta)\rangle>  \langle \tilde{\tau},\alpha\rangle+c,$ then $\langle \tilde{\tau},H_0(a_\theta)+s.\beta-\alpha\rangle>0.$ In this case, 
\begin{align*}
    \lim_{n\to \infty}\sup_{\Re(\chi)=s.\mu} \left|\chi \tilde{\tau}^n(a_\theta)\chi\tilde{\tau}^n(\alpha)^{-1}p\big(s^{-1}.(\chi\tilde{\tau}^n)\big)\right|=0.
\end{align*}
Since $r_a$ is independent of $a_\theta,$ \eqref{eq:cptsupport3} is equal to
\begin{align*}
    \lim_{n\to \infty}(s.\chi_\sigma)(a_\theta)\int_{\Re(\chi)=s.\mu+n\tilde{\tau}} \chi(a_\theta)\chi(\alpha)^{-1}p(s^{-1}.\chi)r_a(s^{-1}.\chi) d\Im\chi=0.
\end{align*}
Thus for $a\in A_0^+(\theta,t,h_
\theta(t)),$ \eqref{eq:cptsupport3} is nonvanishing only if for any such $\tilde{\tau},$ $c\geq \langle \tilde{\tau},-\alpha+H_0(a_\theta)\rangle,$ i.e.,
\begin{align*}
    c\geq \sup_{\substack{\tilde{\tau}\in\mathop{\mathrm{conv}}(\mathcal{B}_0)\\\tilde{\tau}\in X^*(M')^+_\mathbb{R}}}\langle \tilde{\tau},-\alpha+H_0(a_\theta)\rangle.
\end{align*}
We claim there is a constant $c'>0$ depending only on $\theta$ and $\mathcal{B}_0$ such that
\begin{align*}
     \sup_{\substack{\tilde{\tau}\in\mathop{\mathrm{conv}}(\mathcal{B}_0)\\\tilde{\tau}\in X^*(M')^+_\mathbb{R}}}\langle \tilde{\tau},-\alpha+H_0(a_\theta)\rangle\ge c'\sup_{\tilde{\tau}\in\mathop{\mathrm{conv}}(\mathcal{B}_0)}\langle \tilde{\tau},-\alpha+H_0(a_\theta)\rangle.
\end{align*}
To see this, since $a_\theta\in A_{M'}\cap A_0^+(\theta,t,h_\theta(t))$, one has 
\begin{align*}
    \sup_{\tilde{\tau}\in\mathop{\mathrm{conv}}(\mathcal{B}_0)}\langle \tilde{\tau},-\alpha+H_0(a_\theta)\rangle = \sup_{\substack{\tilde{\tau}\in\mathop{\mathrm{conv}}(\mathcal{B}_0)\\\tilde{\tau}\vert_{M'}\in X^*(M')^+_\RR}}\langle \tilde{\tau},-\alpha+H_0(a_\theta)\rangle.
\end{align*}

\begin{lemma} Let $V$ be a finite dimensional real inner product space. Let $W\leq V$ be a linear subspace, and let $\mathrm{pr}:V\to W$ be the orthogonal projection. Let $U\subseteq V$ be a compact neighborhood of $0.$ Suppose $W^+\subseteq W$ is a subset closed under multiplication by $\RR_{>0}.$ Then there exists a constant $c'>0$ such that
\begin{align*}
    \sup_{x\in U\cap W^+} f(x) \geq c' \sup_{x\in U,\,\mathrm{pr}(x)\in W^+}f(x)
\end{align*}
for any $f\in \Hom_\RR(V,\RR)$ such that $f(x) = f(\mathrm{pr}(x))$ for all $x\in V$. 
\end{lemma}
\begin{proof} 
Let $f\in \Hom_\RR(V,\RR)$ such that $f(x) = f(\mathrm{pr}(x))$ for all $x\in V$. Let $\norm{\cdot}$ be the norm on $V$ induced by the inner product. By assumption, there are $R>\delta>0$ such that $\{\norm{x}<\delta\}\subseteq U\subseteq \{\norm{x}<R\}$. Let $c' = \frac{\delta}{R}$. For $x\in U$ with $\mathrm{pr}(x)\in W^+$, put $x':=c'\mathrm{pr}(x)\in W^+$. Since $\norm{\mathrm{pr}(x)}\leq\norm{x}$, this implies $x'\in U$. Thus
\begin{align*}
    c'f(x) = f(c'\mathrm{pr}(x)) =  f(x') \leq  \sup_{z\in U\cap W^+} f(z),
\end{align*}
and the lemma follows.
\end{proof}

Apply the lemma to $(V,U,W,W^+) = (X^*(M_0)_\RR,\mathrm{conv}(\mathcal{B}_0),X^*(M')_\RR,X^*(M')^+_\RR).$ Let $c'$ be a constant satisfying the assertion of the lemma. Since the functional $\langle\cdot,-\alpha+H_0(a_\theta)\rangle$ is trivial on $X_{M_0}^{M'},$ it follows that
\begin{align*}
     \sup_{\substack{\tilde{\tau}\in\mathop{\mathrm{conv}}(\mathcal{B}_0)\\\tilde{\tau}\in X^*(M')^+_\mathbb{R}}}\langle \tilde{\tau},-\alpha+H_0(a_\theta)\rangle\ge c'\sup_{\tilde{\tau}\in\mathop{\mathrm{conv}}(\mathcal{B}_0)}\langle \tilde{\tau},-\alpha+H_0(a_\theta)\rangle
\end{align*}
as claimed. As 
\begin{align*}
    \sup_{\tilde{\tau}\in\mathop{\mathrm{conv}}(\mathcal{B}_0)}\langle \tilde{\tau},-\alpha+H_0(a_\theta)\rangle =  \sup_{\tilde{\tau}\in\partial(\mathop{\mathrm{conv}}(\mathcal{B}_0))}\langle \tilde{\tau},-\alpha+H_0(a_\theta)\rangle = \max_{\tilde{\tau}\in\mathcal{B}_0}\langle \tilde{\tau},-\alpha+H_0(a_\theta)\rangle, 
\end{align*}
we conclude for $a\in A_0^+(\theta,t,h_
\theta(t))$ that \eqref{eq:cptsupport3} is nonvanishing only if 
\begin{align*}
    cc'^{-1}\geq \max_{\tilde{\tau}\in\mathcal{B}_0}\langle \tilde{\tau},-\alpha+H_0(a_\theta)\rangle,
\end{align*}
or equivalently
\begin{align*}
    -cc'^{-1}\le  \min_{\tau\in \calB_0} \langle \tau,\alpha+H_0( w_0.a_\theta^{-1})\rangle.
\end{align*}
Since $\{ a'c_a:a\in A_0^+(\theta,t)\}$ is compact, we have \eqref{eq:cptsupport3} is nonvanishing only if
\begin{align*}
    h_{\theta}'(t)\le\min_{\tau\in \calB_0} \langle \tau,\alpha+H_0( w_0.a^{-1})\rangle
\end{align*}
for some $h_{\theta}'$ independent of $\alpha$. This completes the proof.
\end{proof}

\section{General functional equations and gcd}\label{sec:FunctEq}

For $M\in \mathcal{M}^{\mathrm{std}}$ and $\sigma\in \Pi_2(M),$ define
    \begin{align*}
         I_{(M,\sigma)}:=\mathrm{Span}_\CC\left\{\frac{Z(\mathrm{Ind}_P^G \sigma_\chi,f,v,w)}{L(0,\sigma_\chi,\rho_M)}: f\in \mathcal{S}_\rho(G(F)),(v,w)\in \Ind_{K_P}^K\sigma\vert_{K_M}\times \Ind_{K_P}^K\sigma^\vee\vert_{K_M}\right\}.
    \end{align*}
It is by definition a vector subspace of $\CC[\Lambda_M].$  The section is devoted to proving 
\begin{theorem}\label{multigcd}
    Let $M\in \mathcal{M}^{\mathrm{std}}$ and $\sigma\in \Pi_0(M).$ We have
    \begin{align*}
        I_{(M,\sigma)}=\CC[\Lambda_M].
    \end{align*}
    If in addition, \eqref{LLC:gamma} holds for all $M\in \mathcal{M}$, then the above holds for $\sigma\in \Pi_2(M)$. 
\end{theorem}

We start with an easy observation.
\begin{lemma}\label{lem:ideal}
      Let $M\in \mathcal{M}^{\mathrm{std}}$ and $\sigma\in \Pi_2(M).$ The set $I_{(M,\sigma)}$ is an ideal of $\CC[\Lambda_M],$ and it is equal to
    \begin{align}\label{eq:ideal:alternative}
        \mathrm{Span}_\CC\left\{\frac{Z(\sigma_\chi,f^{(P)},v,w) }{L(0,\sigma_\chi,\rho_M)}: f\in \mathcal{S}_\rho(G(F)),(v,w)\in \sigma \times \sigma^\vee\right\}.
    \end{align}
\end{lemma}

\begin{proof}
   As explained in the proof of \hyperref[lem:constant]{Lemma \ref{lem:constant}}, by an analogous argument as in \hyperref[(i):equivalent:key]{Lemma \ref{(i):equivalent:key}}, one has $I_{(M,\sigma)}$ is equal to \eqref{eq:ideal:alternative}. For $m_0\in M(F)$
    \begin{align*}
        &\chi(m_0)Z(\sigma_\chi,f^{(P)},v,w)\\
        &=\int_{M(F)}R(m_0,1)(f)^{(P)}(m)\langle \sigma_\chi(m)v,\sigma^\lor(m_0)w\rangle dm\\
        &=\delta_{P}^{1/2}(m_0)\int_{M(F)}(R(m_0,1)f)^{(P)}(m)\langle \sigma_\chi(m)v,\sigma^\lor(m_0)w\rangle dm\\
        &=\delta_{P}^{1/2}(m_0)Z(\sigma_\chi,(R(m_0,1)f)^{(P)},v,\sigma^\lor(m_0)w).
    \end{align*}
    Since $\mathcal{S}_\rho(G(F))$ is a $G(F)\times G(F)$-module by \hyperref[Srhomodule]{Corollary \ref{Srhomodule}}, this implies $I_{(M,\sigma)}$ is an ideal. 
\end{proof}


By \hyperref[lem:ideal]{Lemma \ref{lem:ideal}} to prove \hyperref[multigcd]{Theorem \ref{multigcd}}, it suffices to construct elements in $I_{(M,\sigma)}$ that generate the unit ideal. Thanks to \hyperref[Schwartz=as+cs]{Theorem \ref{Schwartz=as+cs}}, we can construct $f\in \mathcal{S}_\rho(G(F))$ by specifying sections $\HP(|\nu|^{-1/2}f).$ The technicality is that we need to choose the sections in a compatible manner.

To orient the reader, we outline the structure of this section. The proof of \hyperref[multigcd]{Theorem \ref{multigcd}} is divided into the supercuspidal and non-supercuspidal discrete series cases, which are treated in \S \ref{ssec:multigcd:cuspidal} and \S \ref{ssec:multigcd:noncuspidal}, respectively. The proofs of the two cases are largely the same, except that when $\sigma\in \Pi_2(M)-\Pi_0(M)$, we require the additional input of \eqref{LLC:gamma} to understand the relation between $L(s,\sigma,\rho_M)$ and the $L$-function of its supercuspidal support (see \hyperref[lem:divides]{Lemma \ref{lem:divides}}).

We study the consequences of \eqref{LLC:gamma} in \S \ref{sec:LLC7}. The technical result that will be used crucially in \S \ref{ssec:multigcd:noncuspidal} is \hyperref[lem:auxL]{Lemma \ref{lem:auxL}}. We recommend the reader to assume \hyperref[lem:auxL]{Lemma \ref{lem:auxL}} and start from \S \ref{ssec:rem:discrete}, in which we review Casselman's criterion for square-integrability and the Geometric lemma. We apply these and the matrical Paley-Wiener theorem (\hyperref[BH:PW]{Theorem \ref{BH:PW}}) to carefully construct functions $f\in \mathcal{S}_\rho^{\mathrm{as}}(G(F))\cap C^\infty_{\mathrm{ac}}(G(F))$ such that $Z(\mathrm{Ind}_P^G \sigma_\chi,f,v,w)$ is a nonzero function for some $v,w$ and $f$ has the smallest possible spectral supports in both \S \ref{ssec:multigcd:cuspidal} and \S \ref{ssec:multigcd:noncuspidal}. Finally, in \S \ref{ssec:Tate} we state the complete version of \hyperref[BKN:conj]{Theorem \ref{BKN:conj}}, whose proof is a simple modification of that of \hyperref[multigcd]{Theorem \ref{multigcd}}.

\subsection{Consequences of \eqref{LLC:gamma}}\label{sec:LLC7}
In this subsection assume \eqref{LLC:gamma} holds for all $M\in \mathcal{M}$. By \hyperref[LLC7=Fourierstable]{Proposition \ref{LLC7=Fourierstable}}, the $\rho$-Fourier transform then defines an automorphism
\begin{align*}
    \mathcal{F}_\rho: \mathcal{S}_\rho^{\mathrm{as}}(G(F))\cap C^{\infty}_{\mathrm{ac}}(G(F))\longrightarrow \mathcal{S}_\rho^{\mathrm{as}}(G(F))\cap C^{\infty}_{\mathrm{ac}}(G(F)).
\end{align*}

We first prove general functional equations for zeta integrals of an irreducible smooth representation $\pi$ of $G(F)$. Recall that a Langlands triple $(P,\sigma,\chi)$ for $G(F)$ consists of 
\begin{enumerate}
    \item a standard parabolic $P=MN\in \mathcal{P}^{\mathrm{std}}$,
    \item an irreducible tempered representation $\sigma$ of $M(F),$ and
    \item a quasi-character $\chi\in \Lambda_M$ with $\Re(\chi)\in X^\ast(M)^+_\RR.$  
\end{enumerate}
By the Langlands classification \cite{Silb:classification, BW:book, Representation-p-adic}, for a Langlands triple $(P,\sigma,\chi)$ 
\begin{align*}\Hom_{G(F)}(\mathrm{Ind}_P^G \sigma_\chi, \mathrm{Ind}_{\overline{P}}^G\sigma_\chi)=\CC J_{\overline{P}|P}(\sigma_\chi).
\end{align*}
Furthermore, $\mathrm{Ind}_P^G \sigma_\chi$ admits a unique irreducible quotient $J(P,\sigma_\chi),$ which is isomorphic to the image of $J_{\overline{P}|P}(\sigma_\chi).$ Conversely, every irreducible smooth  representation $\pi$ of $G(F)$ is isomorphic to $J(P,\sigma_\chi)$ for some Langlands triple $(P,\sigma, \chi),$ and $J(P,\sigma_\chi)\cong J(P',\sigma'_{\chi'})$ if and only if $P=P'$ and $\sigma_\chi=\sigma'_{\chi'}.$ Therefore, if we restrict to $\chi\in X^\ast(M)_\RR^+,$ we have a one-to-one correspondence between $\mathrm{Irr}(G)$ and Langlands triples.

\begin{proposition}\label{prop:generalFE}
     Let $f\in \mathcal{S}_\rho^{\mathrm{as}}(G(F))\cap C^{\infty}_{\mathrm{ac}}(G(F))$. Let $\pi\in \mathrm{Irr}(G)$ with Langlands triple $(P,\sigma,\lambda)$. For any $(v,w)\in \pi\times \pi^\lor,$ the integral 
     \begin{align*}
         Z(\pi_\chi,f,v,w):=\int_{G(F)} f(g)\langle \pi_\chi(g)v,w\rangle dg
     \end{align*}
     is absolutely convergent for $\Re(\chi)$ in a positive cone in $\Re\Lambda_G$. Moreover, one has a functional equation
        \begin{align*}
            Z((\pi_\chi)^\lor,\mathcal{F}_\rho f,w,v)=\gamma(\tfrac{1}{2},\sigma_{\lambda\chi},\rho_M,\psi)Z(\pi_{\chi},f,v,w).
        \end{align*}
        as rational functions on $\Lambda_G.$
\end{proposition}

\begin{proof}
    We can assume $\pi$ has unitary central character. When $\pi\in\Temp(G)\cup\Temp_\Ind(G)$, the assertion follows from \hyperref[Sas:twist:HC]{Lemma \ref{Sas:twist:HC}} and \hyperref[Sas:FE:near0]{Corollary \ref{Sas:FE:near0}}. Note \eqref{LLC:para} is used here so that $\gamma(\tfrac{1}{2},\pi,\rho,\psi) = \gamma(\tfrac{1}{2},\sigma,\rho_M,\psi)$ if $(M,\sigma)\in\widetilde{\Temp}_\Ind(G)$ and $\pi$ is a subrepresentation of $\Ind_P^G\sigma$. 
    

 For general $\pi\in\mathrm{Irr}(G)$,  we identify $\pi^\lor$ as the unique irreducible subrepresentation of $\mathrm{Ind}_P^G(\sigma_\lambda)^\lor$. Choose $\tilde{v}$ in $\mathrm{Ind}_P^G \sigma_\lambda$ whose image in $\pi$ is $v$. Then $\langle \pi_\chi(g)v,w\rangle=\langle (\Ind_{P}^G \sigma_{\lambda \chi})(g)\tilde{v},w\rangle$ for all $g\in G(F)$ and $\chi\in \Lambda_G$. By \eqref{S1}
    \begin{align*}
        g\mapsto  f(g)\langle(\Ind_P^G\sigma_{\lambda\chi})(g)\tilde{v},w\rangle \in L^1(G(F))
    \end{align*}
    for $\Re(\chi)$ in a positive cone in $\Re \Lambda_G$. This proves the first assertion. In addition, by the previous case we have
    \begin{align*}
        Z((\pi_\chi)^\lor,\mathcal{F}_\rho f,w,v) &= Z(\Ind_P^G(\sigma_{\lambda\chi})^\vee,\mathcal{F}_\rho f,w,\tilde{v})\\
        &=\gamma(\tfrac{1}{2},\sigma_{\lambda\chi},\rho_M,\psi)Z(\Ind_P^G\sigma_{\lambda\chi},f,\tilde{v},w)\\
        &=\gamma(\tfrac{1}{2},\sigma_{\lambda\chi},\rho_M,\psi)Z(\pi_\chi,f,v,w).
    \end{align*}
\end{proof}

\begin{lemma}\label{lem:naivegen}
    Let $(M,\sigma)\in \widetilde{\Temp}_\Ind(G).$ Then $I_{(M,\sigma)}$ contains $L(0,\sigma_\chi,\rho_M)^{-1}$ and  $L(1,(\sigma_{\chi})^\lor,\rho_M)^{-1}.$ The vanishing locus of $I_{(M,\sigma)}$ in $\Lambda_M$ has codimension at least $2$. In particular, if $\dim_\CC\Lambda_M=1,$ then $I_{(M,\sigma)}=\CC[\Lambda_M].$
\end{lemma}

\begin{proof}
     Choose $(v_0,w_0)\in \Ind_{K_P}^K \sigma|_{K_M}\times \Ind_{K_P}^K \sigma^\lor|_{K_M}$ such that $\langle v_0,w_0\rangle=1.$ Let $K_0$ be a compact open subgroup such that $v_0$ is fixed by $K_0.$ Then $Z(\mathrm{Ind}_P^G \sigma_\chi,\mathbf{1}_{K_0},v_0,w_0)=\mathrm{vol}(K_0)$ and thus by \hyperref[Sas:FE:near0]{Corollary \ref{Sas:FE:near0}}
    \begin{align*}
        Z(\mathrm{Ind}_P^G (\sigma_\chi)^\lor,|\nu|^{-1/2}\mathcal{F}_\rho |\nu|^{1/2}\mathbf{1}_{K_0},w_0,v_0)=\mathrm{vol}(K_0)\gamma(0,\sigma_\chi,\rho_M,\psi).
    \end{align*}
    Therefore, $I_{(M,\sigma)}$ contains $L(0,\sigma_\chi,\rho_M)^{-1}$ and  $L(1,(\sigma_{\chi})^\lor,\rho_M)^{-1}$. 
    
    Retain the notation in the proof of \hyperref[Lfunction:expand:cone]{Lemma \ref{Lfunction:expand:cone}}. Then
\begin{align*}
    L(0,\sigma_\chi,\rho_M)^{-1} = \prod_{j=1}^n\prod_{k=0}^R \det\left(I - q^{-\frac{k}{2}-\langle \chi,\widetilde{\mu}_j\rangle}\phi_{j,k}(\sigma)(\mathrm{Frob})\big|_{V_{j,k}^{I_F}}\right).
\end{align*}
Since eigenvalues of $\phi_{j,k}(\sigma)(\mathrm{Frob})\big|_{V_{j,k}^{I_F}}$ have norm $1$, it follows that
\begin{align*}
    L(0,\sigma_\chi,\rho_M)^{-1} = \prod_{j=1}^n\prod_{k=0}^R \prod_{l=1}^{\dim V_{j,k}^{I_F}} (1-q^{-\frac{k}{2}-\langle \chi,\widetilde{\mu}_j\rangle+\lambda_{j,k,l}})
\end{align*}
for some $\lambda_{j,k,l}\in i\RR$. Similarly,
\begin{align*}
    L(1,(\sigma_\chi)^\vee,\rho_M)^{-1} = \prod_{j=1}^n\prod_{k=0}^R \prod_{l=1}^{\dim V_{j,k}^{I_F}} (1-q^{-\frac{k}{2}-1+\langle \chi,\widetilde{\mu}_j\rangle-\lambda_{j,k,l}}).
\end{align*}
We claim for all $(j,k,\ell),(j',k',\ell')$, the intersection
\begin{align*}
    \{\chi\in X^*(M)_\CC:1 - q^{-\frac{k'}{2}-\langle \chi,\widetilde{\mu}_{j'}\rangle+\lambda_{j',k',l'}}=0\} \cap \{\chi\in X^*(M)_\CC :1 - q^{-\frac{k}{2}-1+\langle \chi,\widetilde{\mu}_j\rangle-\lambda_{j,k,l}}= 0\}.
\end{align*}
must have codimension at least $2$. This is clear if $\widetilde{\mu}_{j'}$ and $\widetilde{\mu}_j$ are not parallel. If $\widetilde{\mu}_{j'} = c \widetilde{\mu}_{j}$ for some $c\in \CC^\times$, then by \eqref{G}
\begin{align*}
    1 = \langle\nu,\widetilde{\mu}_{j'}\rangle = c\langle\nu,\widetilde{\mu}_{j}\rangle = c,
\end{align*}
so $\widetilde{\mu}_{j'} = \widetilde{\mu}_{j}$. But $\frac{k}{2} + 1 +\lambda_{j,k,l} \not\equiv -\frac{k'}{2} - \lambda_{j',k',l'}\pmod{\frac{2\pi i}{\log q}\ZZ}$, so the claim follows.
In particular, the common zeros of $L(0,\sigma_\chi,\rho_M)^{-1}$ and $L(1,(\sigma_{\chi})^\lor,\rho_M)^{-1}$ in $\Lambda_M$ have codimension at least $2$.
\end{proof}

\begin{lemma}\label{lem:divides}
    Let $\pi\in \Pi_2(G).$ Suppose $\pi$ is a subquotient of $\mathrm{Ind}_P^G\sigma_\lambda$ for some $(M,\sigma)\in\widetilde{\Temp}_{\Ind}(G)$ and $\lambda\in \Lambda_M$. Then 
    \begin{align*}
        \chi\mapsto \frac{L(0,\pi_\chi,\rho)}{L(0,\sigma_{\lambda\chi},\rho_M)}\in \CC[\Lambda_G].
    \end{align*}
\end{lemma}

\begin{proof}
By the proof of \hyperref[lem:naivegen]{Lemma \ref{lem:naivegen}}, the set of common zeros of $L(0,\pi_\chi,\rho)^{-1}$ and  $L(1,(\pi_{\chi})^\lor,\rho)^{-1},$ denoted by $Y$, is a codimension $2$ algebraic subscheme of $\Lambda_G$. When $\dim_\CC\Lambda_G=1,$ $Y$ is empty.

    Let $C$ be a smooth irreducible affine algebraic curve in $\Lambda_G-Y$. Then $L(0,\pi_{\chi},\rho)^{-1}|_C$ and $L(1,(\pi_{\chi})^\lor,\rho)^{-1}|_C$ as functions on $C$ do not have common zeros, so they generate $\Gamma(C,\mathcal{O}_C).$ Therefore, by \hyperref[lem:naivegen]{Lemma \ref{lem:naivegen}} there exist finitely many $f_i,v_i,w_i$ such that
    \begin{align*}
\sum_{i}Z(\pi_{\chi},f_i,v_i,w_i)|_C=L(0,\pi_{\chi},\rho)|_C.
    \end{align*}
    On the other hand, since $\pi$ is a subquotient of $\mathrm{Ind}_P^G\sigma_\lambda,$ as in \hyperref[lem:cs:subquo]{Lemma \ref{lem:cs:subquo}} for a proper lift $(v_i',w_i')$ of $(v_i,w_i)$
    \begin{align*}
        \sum_{i}Z(\pi_{\chi},f_i,v_i,w_i)=\sum_{i}Z(\mathrm{Ind}_P^G \sigma_{\lambda\chi},f_i,v_i',w_i')\in L(0,\sigma_{\lambda\chi},\rho_M)\CC[\Lambda_G].
    \end{align*}
   It follows that
    \begin{align*}
        \frac{L(0,\pi_\chi,\rho)}{L(0,\sigma_{\lambda\chi},\rho_M)}\bigg|_C\in \Gamma(C,\mathcal{O}_C).
    \end{align*}

    When $\dim_\CC \Lambda_G=1,$ we can take $C=\Lambda_G$ and we are done. Suppose $\dim_\CC \Lambda_G\ge 2.$ Suppose on the contrary that $\tfrac{L(0,\pi_\chi,\rho)}{L(0,\sigma_{\lambda\chi},\rho_M)}$ is not entire. Since $Y$ has codimension $2$, by Bertini's theorem \cite[Theorem 6.3, Remark 6.3.18]{Jean-Pierre:Bertini} we can find a $C\subset \Lambda_G-Y$ such that $\tfrac{L(0,\pi_\chi,\rho)}{L(0,\sigma_{\lambda\chi},\rho_M)}\big|_C$ is not holomorphic, a contradiction.
\end{proof}

Let $M\in \mathcal{M}^{\mathrm{std}}$ and $\sigma\in \Pi_0(M)_\CC.$ Recall $U_\sigma=\{\tau\in \Im \Lambda_M: \sigma\cong \sigma_\tau\}$ is a finite group, and the $\Lambda_M$-orbit of $\sigma$ is a complex algebraic torus $\mathcal{O}_{\sigma\CC}:=\Lambda_M\sigma\cong \Lambda_M/U_\sigma.$ Let $W(M,\mathcal{O}_{\sigma\CC})=\{s\in W(G|M):s.\mathcal{O}_{\sigma\CC}=\mathcal{O}_{\sigma\CC}\}.$ Thus $W(M,\mathcal{O}_{\sigma\CC})$ naturally acts on $\mathcal{O}_{\sigma\CC}.$

\begin{lemma}\label{lem:auxL}
Let $\pi\in \Pi_2(G)-\Pi_0(G)$ be a subrepresentation of $\mathrm{Ind}_{P}^G \sigma$ for some $\sigma\in \Pi_0(M)_\CC,$ where $M\in \mathcal{M}^{\mathrm{std}}.$  There exists a rational function $L\in \CC(\mathcal{O}_{\sigma\CC})^{W(M,\mathcal{O}_{\sigma\CC})}$ such that the following holds.
\begin{enumerate}
        \item Suppose $P'=M'N'\in \mathcal{P}^{\mathrm{std}}$ and $P'\ge P$ such that $\pi$ is a subquotient of $\Ind_{P'}^G \sigma'$ for some $\sigma'\in \Pi_2(M')_\CC.$ Then there exists $\chi'\in \Lambda_M$ such that $\sigma'$ is a subquotient of $\mathrm{Ind}_{P\cap M'}^{M'} \mathrm{\sigma}_{\chi'}.$ Furthermore,  
        \begin{align*}
                \Lambda_{M'}\ni \chi \mapsto \frac{L(\sigma_{\chi'\chi})}{L(0,\sigma'_\chi,\rho_{M'})} \in \CC[\Lambda_{M'}].
        \end{align*}
        \item In addition, 
       \begin{align*}
                \Lambda_{G}\ni \chi \mapsto \frac{L(\sigma_{\chi})}{L(0,\pi_\chi,\rho)} \in \CC[\Lambda_{G}]^\times.
        \end{align*}
\end{enumerate}

\end{lemma}
\begin{proof} 
Let $P'=M'N'\in \mathcal{P}$ and $M'\ge M$ such that $\pi$ is a subquotient of $\Ind_{P'}^G \sigma'$ for some $\sigma'\in \Pi_2(M')_\CC.$ By \cite[Th\'eor\`eme VI.5.4]{Representation-p-adic} there exists a parabolic subgroup $Q$ of $M'$ with Levi component $M$ and $\chi'\in \Lambda_M$ such that $\sigma'$ is a subquotient of $\mathrm{Ind}_{Q}^{M'} \mathrm{\sigma}_{\chi'}.$ Furthermore, there is $s\in W(M,\mathcal{O}_{\sigma\CC})$ such that $s.\sigma_{\chi'}\cong \sigma.$ Consider the (finite) set $S$ of all such triples $(P'=M'N',\sigma',\chi').$ 

The set $S$ admits an action of $U_\sigma:$ for $\tau\in U_\sigma$ and $(P',\sigma',\chi')\in S$
\begin{align*}
   \tau. (P',\sigma',\chi'):=(P',\sigma',\tau\chi').
\end{align*}
The quotient $S/U_\sigma$ admits an action of $W(M,\mathcal{O}_{\sigma\CC}):$ for $s\in W(M,\mathcal{O}_{\sigma\CC})$ and $(P',\sigma',U_\sigma\chi')\in S/U_\sigma$
\begin{align*}
   s. (P',\sigma',U_\sigma\chi'):=(s.P',s.\sigma',U_\sigma s.\chi').
\end{align*}

 Let $(P',\sigma',\chi')\in S.$ By \hyperref[lem:divides]{Lemma \ref{lem:divides}}, we have
\begin{align}\label{eq:intermediate}
     \Lambda_{M'}\ni\chi \mapsto f_{P',\sigma',\chi'}(\chi):=\frac{L(0,\sigma'_\chi,\rho_{M'})}{L(0,\sigma_{\chi'\chi},\rho_M)}\in\CC[\Lambda_{M'}],
\end{align}
and
\begin{align}\label{eq:intermediate2}
    \Lambda_{G}\ni\chi \mapsto h_{P',\sigma',\chi'}(\chi):=\frac{L(0,\pi_{\chi},\rho)}{L(0,\sigma'_\chi,\rho_{M'})} \in \CC[\Lambda_{G}].
\end{align}
Note that $\mathcal{O}_{\sigma\CC}\ni \sigma_\chi\mapsto L(0,\sigma_\chi,\rho_M)\in \CC(\mathcal{O}_{\sigma\CC})^{W(M,\mathcal{O}_{\sigma\CC})}.$ Therefore, as functions on $\Lambda_G$
\begin{align*}
    f_{G,\pi,1}(\chi) = h_{P',\sigma',\chi'}(\chi)f_{P',\sigma',\chi'}(\chi).
\end{align*}

For $(P',\sigma',\chi')\in S$, we let $Z_{(P',\sigma',\chi')}\subseteq \Lambda_{M'}$ denote the zero locus of $L(0,\sigma'_\chi,\rho_{M'})^{-1}$ in  $\Lambda_{M'}$. Let $S':=\{ (P',\sigma',\chi')\in S: M\lneq P'\lneq  G\},$ and consider the following subsets of $\Lambda_M:$
\begin{align*}
    Z_\sigma&:=Z_{(P,\sigma,1)},\\
    Z &:= \bigcup\limits_{(P',\sigma',\chi')\in S'} \chi 'Z_{(P',\sigma',\chi')}\subseteq \Lambda_{M},\\
    Z_{\pi} &:= \bigcup _{(G,\pi,\chi')\in S} \chi' Z_{(G,\pi,\chi')}=W(M,\mathcal{O}_{\sigma\CC}).(U_\sigma Z_{(G,\pi,1)}).
\end{align*}
By \eqref{eq:intermediate} $Z\subseteq Z_\sigma.$ In addition, if $Z$ is nonempty, then $Z_{\pi}\subseteq Z$  by \eqref{eq:intermediate2}.
Note that $Z_\sigma, Z,Z_\pi$ are stable under $U_\sigma,$ and their quotients by $U_\sigma$ are stable under $W(M,\mathcal{O}_{\sigma\CC}).$ In particular, $U_\sigma$ permutes irreducible components of $Z_\sigma,Z,Z_\pi,$ and $W(M,\mathcal{O}_{\sigma\CC})$ permutes the $U_\sigma$-orbits of their irreducible components.

Let $T_M:=X^*(M)\otimes_\ZZ \frac{\CC}{\frac{2\pi i}{\log q}\ZZ}$, and let $p_M:T_M\to \Lambda_M$ be the map induced by \eqref{XG->LambdaG:surj}. Let $\widetilde{D}:=\ker p_M,$ which is a finite subgroup of $X^*(M)\otimes_\ZZ \frac{i\RR}{\frac{2\pi i}{\log q}\ZZ}$. For $(P',\sigma',\chi')\in S$, let $A_{(P',\sigma',\chi')}$ denote the set of irreducible components of $ \chi'Z_{(P',\sigma',\chi')}.$ For $X\in A_{(P',\sigma',\chi')},$ let $\widetilde{X}$ be an irreducible/connected component of $p_M^{-1}(X).$  Then $\pi_0(p_M^{-1}(X))=\widetilde{D}\widetilde{X}.$ By the proof of \hyperref[lem:naivegen]{Lemma \ref{lem:naivegen}},  $\widetilde{X}$ is a coset of a subtorus of $T_M$ of corank $\dim \Lambda_M-\dim\Lambda_{M'}+1$. Note that $T_{M'}\to T_M$ is injective, and we choose $\widetilde{X}$ that is contained in the image of $T_{M'}.$

    Let $A := \bigcup\limits_{(P',\sigma',\chi')\in S'} A_{(P',\sigma',\chi')},$ $A_\emptyset:=\{ X\in A: X\cap Z_\pi= \emptyset\}$ and $A_\pi:=\bigcup_{s\in W(M,\mathcal{O}_{\sigma\CC})} s.U_\sigma A_{(G,\pi,1)}.$ These sets are $U_\sigma$-invariant, and their quotients by $U_\sigma$ are $W(M,\mathcal{O}_{\sigma\CC})$-invariant. For each $X\in A\cup A_\pi,$ choose a coset $H_{X}$ of a subtorus of $T_M$ of corank $1$ such that $\widetilde{X} \subseteq H_{X}$ and the following holds:
    \begin{itemize}
        \item $\bigcup\limits_{X\in A}\widetilde{D} H_X$ and $\bigcup\limits_{X\in A_\pi}\widetilde{D} H_X$  are stable under $U_\sigma.$
        \item $\bigcup\limits_{X\in A}\widetilde{D} H_X/U_\sigma$ and $\bigcup\limits_{X\in A_\pi}\widetilde{D} H_X/U_\sigma$ are stable under $W(M,\mathcal{O}_{\sigma\CC}).$
    \end{itemize}
    Write $H_{X} = \{\chi\in T_M:1 - c_Xq^{-\langle\chi,\alpha_X\rangle}=0\}$ for some $c_X\in \CC^\times$ and $\alpha_X\in \Hom_\ZZ(X^*(M),\ZZ)$ that is primitive. Choose $\chi'_X\in T_M$ such that $p_M(\chi'_X)=\chi',$ and let
    \begin{align*}
        h_X(\chi) := \prod_{\tau\in \widetilde{D}}\left(1 - c_Xq^{-\langle\tau\chi'^{-1}_X\chi,\alpha_X\rangle}\right) \in \CC[T_M]^{\widetilde{D}} = \CC[\Lambda_M].
    \end{align*}
    Then $h_X$ is independent of the choice of $\chi'_X,$ and the zero locus of $h_X$ is $\widetilde{D}\chi'_X H_X.$ We may and do choose $\{(c_X,\alpha_X):X\in A\cup A_\pi\}$ so that for any $X\in A\cup A_\pi$
    \begin{align*}
        \prod_{X'\in W(M,\mathcal{O}_{\sigma\CC}).(U_\sigma X)} h_{X'}(\chi) \in \CC[\mathcal{O}_{\sigma\CC}]^{W(M,\mathcal{O}_{\sigma\CC})}.
    \end{align*}

    Let $X\in A-A_\emptyset.$ Put
    \begin{align*}
         m_X:=  \max_{\substack{(P',\sigma',\chi')\\ X\in Z_{(P',\sigma',\chi')}}} \text{ multiplicity of $\widetilde{X}$ in $\left.\dfrac{L(0,\sigma'_\chi,\rho_{M'})}{L(0,\sigma_{\chi'\chi},\rho_M)}\right|_{\chi\in T_{M'}}$ as a divisor},
    \end{align*}
    If $X$ is a maximal subset among $A-A_\emptyset,$ let $\widetilde{m}_{X}:=m_X$. Otherwise, let
    \begin{align*}
        \widetilde{m}_{X}:=\max_{X\subsetneq X'\in A-A_\emptyset}(0,m_X-m_{X'}).
    \end{align*}

    Let $X\in A_\pi.$ There is $s\in W(M,\mathcal{O}_{\sigma\CC})$ and $\tau\in U_\sigma$ such that $s.X\subseteq \tau^{-1} Z_{G,\pi,1}.$ Let 
    \begin{align*}
        m_X:=m_{s.X}:=\text{ multiplicity of $\widetilde{\tau(s.X)}$ in $\left.\dfrac{L(0,\pi_{\chi},\rho)}{L(0,\sigma_{\chi},\rho_M)}\right|_{\chi\in T_{G}}$ as a divisor}.
    \end{align*}
    This is well-defined, i.e., independent of the choices of $s$ and $\tau$ because $U_\sigma\cap \Lambda_G \subseteq U_\pi$. By \eqref{eq:intermediate2}, we have
    \begin{align*}
        \widetilde{m}_X:=m_X-\max_{X\subsetneq X'\in A-A_\emptyset} m_{X'}\ge 0.
    \end{align*}
    In other words,
    \begin{align}\label{eq:sumequal}
        m_X=\widetilde{m}_X+\sum_{X\subsetneq X'\in A-A_\emptyset} \widetilde{m}_{X'}.
    \end{align}
    Note that the function $A_\pi \cup A-A_\emptyset\ni X\mapsto \widetilde{m}_X$ is invariant under $U_\sigma$ and $W(M,\mathcal{O}_{\sigma\CC}).$ 
    
    Choose $n\in \ZZ_{\ge 0}$ and define
    \begin{align*}
        L(\sigma_\chi) := L(0,\sigma_\chi,\rho_M) \prod_{X\in A_\pi \cup A-A_\emptyset} h_X(\chi)^{\widetilde{m}_X} \prod_{X\in A_{\emptyset}} h_X(\chi)^{n}.
    \end{align*}
    Clearly $L\in \CC(\mathcal{O}_{\sigma\CC})^{W(M,\mathcal{O}_{\sigma\CC})}.$ By choosing $n$ large, condition (1) holds by the construction of $L.$ 

     Denote by $\mathfrak{L}$ the collection of all possible functions $L$ constructed above. For each $L\in \mathfrak{L}$ and $X\in A\cup A_\pi,$ we let $H_{L,X}$ be the coset chosen in the construction of $L$. Since each coset of positive corank is a finite intersection of cosets of corank $1$, for a fixed $\chi_0\in\Lambda_G$ we can choose $\{H_{L,X}:\,X\in A\cup A_\pi\}$ so that $\chi_0 \in X$ if and only if $\chi_0\in p_M(H_{L,X})$. For this choice, by \eqref{eq:intermediate2} and \eqref{eq:sumequal} we have
     \begin{align*}
         \left.\frac{L(\sigma_\chi)}{L(0,\pi_\chi,\rho)}\right|_{\chi=\chi_0} \neq 0.
     \end{align*} 
    Note $\Lambda_G\ni \chi\mapsto L(0,\pi_\chi,\rho)$ is invariant by $U_\pi$ and hence by $U_\sigma\cap \Lambda_G$. Thus by Hilbert's Nullstellensatz, the ideal in $\CC[\Lambda_G]^{U_\sigma\cap \Lambda_G}$ generated by
    \begin{align*}
        \left\{  \Lambda_G\ni\chi\mapsto \frac{L(\sigma_\chi)}{L(0,\pi_\chi,\rho)}: L\in \mathfrak{L}\right\}
    \end{align*}
    is unital. Since $\CC[\mathcal{O}_{\sigma\CC}]\longrightarrow \CC[\Lambda_G]^{U_\sigma\cap \Lambda_G}$ is surjective, and it factors through the averaging $\CC[\mathcal{O}_{\sigma\CC}]\to \CC[\mathcal{O}_{\sigma\CC}]^{W(M,\mathcal{O}_{\sigma\CC})}$, the lemma follows.
\end{proof}

\subsection{Some remarks on discrete series}\label{ssec:rem:discrete}

We review in this subsection some facts on (essentially) discrete series that will be repeatedly used later in \S \ref{ssec:multigcd:cuspidal} and \S \ref{ssec:multigcd:noncuspidal}.

Let $(\pi,V)$ be a smooth admissible representation of $G(F)$. For $\chi\in \Lambda_G,$ we let
\begin{align*}
     V_\chi:=\{v \in V : \textrm{there is } n \in \ZZ_{\ge 0} \textrm{ such that for all } a \in A_G(F), (\pi(a)-\chi(a))^n v = 0\}.
\end{align*}
If $V_\chi\neq 0,$ we say $\chi$ is an exponent of $V$. Let $\mathcal{E}xp(\pi)\subseteq \Lambda_G$ be the set of exponents of $V$. Then
\begin{align*}
    V=\bigoplus_{\chi\in \mathcal{E}xp(\pi)} V_\chi.
\end{align*}
Let $P=MN\in \mathcal{P}$ and $r^G_P$ be the normalized Jacquet functor. Then $r^G_P \pi$ is a smooth admissible representation of $M(F)$. If $\pi$ has finite length, $\mathcal{E}xp(\pi)$ is finite and $r^G_P \pi$ also has finite length.  
\begin{theorem}[Casselman's criterion for square-integrability {\cite[VII.1.2]{Representation-p-adic}}]\label{CasselmanCriterion}
    A smooth admissible representation $(\pi,V)$ of $G(F)$ with unitary central character is square-integrable if and only if for any $P=MN\in \mathcal{P},$ the real part of each exponent lies in $\RR_{>0}\Delta(P),$ i.e.,
    \begin{align*}
        \Re \mathcal{E}xp(r^G_P\pi)\subseteq \RR_{>0}\Delta(P)= \left\{\sum_{\alpha\in\Delta(P)} x_\alpha \alpha: x_\alpha\in \RR_{>0}\right\}.
    \end{align*}\qed
\end{theorem}

For an admissible representation $\pi$ of finite length, we let $\mathrm{s.s.}(\pi)$ be its semisimplification. For $\theta\subseteq \Delta,$ let $P_\theta=M_\theta N_\theta$ be the standard parabolic subgroup associated to $\theta$. For $\theta,\omega\subseteq \Delta,$ let 
\begin{align*}
    W^{M_\theta,M_\omega} := \{s\in W(G,A_0):s^{-1}\theta, s\omega\subseteq \Sigma\}.
\end{align*}
By \cite[V.4.6]{Representation-p-adic} this is a set of representatives of the double quotient $W(M_\theta,A_0)\backslash W(G,A_0)/W(M_\omega,A_0)$. Furthermore, 
\begin{align*}
    (s^{-1}.P_\theta\cap P_\omega)N_\omega=P_{s^{-1}\theta\cap \omega}
\end{align*}
whose unipotent radical is generated by $N_\omega$ and $s^{-1}N_\theta s\cap N_\emptyset,$ and its standard Levi subgroup is $s^{-1}.M_\theta\cap M_\omega.$ The group $s^{-1}.P_\theta\cap M_\omega$ is a parabolic subgroup of $M_\omega$ with unipotent radical $s^{-1}.N_\theta  \cap M_\omega$ and Levi subgroup $s^{-1}.M_\theta \cap M_\omega$.

\begin{theorem}[Geometric Lemma {\cite[VI.5.1]{Representation-p-adic}}]\label{geometriclemma:1}
    Let $\theta,\omega\subseteq \Delta.$ Let $\sigma$ be an admissible representation of $M_\theta(F)$ of finite length. Then 
    \begin{align*}
        \mathrm{s.s.}\Big(r^G_{P_\omega}\circ \mathrm{Ind}_{P_\theta}^G \sigma\Big)=\sum_{s\in W^{M_\theta,M_\omega}} \mathrm{s.s.} \left(\mathrm{Ind}^{M_\omega}_{s^{-1}.P_\theta\cap  M_\omega} \circ r^{s^{-1}.M_\theta}_{P_{s^{-1}\theta\cap \omega}\cap s^{-1}.M_{\theta}} s^{-1}.\sigma\right).
    \end{align*}
    \qed    
\end{theorem}

\begin{corollary}\label{geometriclemma:2} Let $\omega\subseteq \theta\subseteq\Delta.$ Suppose $\sigma$ is $M_\theta(F)$-subrepresentation of  $\mathrm{Ind}_{P_\omega\cap M_\theta}^{M_\theta} \sigma',$ where $\sigma'$ is a supercuspidal representation of $M_\omega(F)$. Then
    \begin{align*}
         \mathrm{s.s.}\left(r_{P_\omega}^G \circ \mathrm{Ind}_{P_\theta}^G \sigma\right)=\sum_{\substack{s\in W^{M_\theta,M_\omega}\\ M_\omega\leq s^{-1}.M_\theta}}  \mathrm{s.s.}\left( r^{s^{-1}.M_\theta}_{P_{s^{-1}\theta\cap \omega}\cap s^{-1}.M_\theta}s^{-1}.\sigma\right).
    \end{align*}
    In particular, for $P=MN\in \mathcal{P}^{\mathrm{std}}$ and $\sigma$ a supercuspidal representation of $M(F)$, 
    \begin{align*}
            \mathrm{s.s.}\left(r_P^G \circ \mathrm{Ind}_P^G \sigma\right)=\sum_{s\in W(G|M)} s.\sigma.
    \end{align*}
    \qed

\end{corollary}

\begin{corollary}\label{cor:notfull}
    Suppose $M\in \mathcal{M}^{\mathrm{std}}$ is proper. For an admissible representation $\sigma$ of finite length of $M(F),$ $\Ind_P^G\sigma$ is not essentially square-integrable.
\end{corollary}

\begin{proof}
    We may assume $\sigma$ is irreducible. Let $w_0$ be the long Weyl element in $W(G|M)$. By the \hyperref[geometriclemma:1]{Geometric lemma}, both $w_0.\sigma$ and $\sigma$ are subquotients of $r_P^G\circ \mathrm{Ind}_P^G\sigma$. Thus central characters of $\sigma$ and $w_0.\sigma$ are exponents of $r_P^G\circ \mathrm{Ind}_P^G\sigma.$ Since $P$ is proper, the real part of one of the two central characters must not lie in $\Re\Lambda_G+\RR_{>0}\Delta(P)$. The assertion then follows from \hyperref[CasselmanCriterion]{Casselman's criterion}.
\end{proof}

\subsection{Proof of \hyperref[multigcd]{Theorem \ref{multigcd}} for supercuspidal representations}\label{ssec:multigcd:cuspidal} For $(M,\sigma)\in \widetilde{\mathrm{Temp}}_{\mathrm{Ind}}(G),$ let $\mathcal{O}_\sigma$ be the $\Im\Lambda_M$-orbit in $\widetilde{\mathrm{Temp}}_{\mathrm{Ind}}(G)$. Fix $M$ and $\sigma\in \Pi_0(M)$.  Consider the set of pairs $(M_i,\mathcal{O}_{\sigma_i})$ where $M_i>M$ is a standard Levi subgroup and $\sigma_i\in \Pi_2(M_i)-\Pi_0(M_i)$ is isomorphic to a subrepresentation of $\Ind_{P\cap M_i}^{M_i} \sigma_{\chi_i}$ for some $\chi_i\in \Lambda_M.$ Denote this set of pairs by $J_{(M,\sigma)}$ indexed by $i\in Y_{(M,\sigma)}$. By \cite[Proposition 3.1]{BDK:TracePW}, $Y_{(M,\sigma)}$ is finite.  

Recall the notations in \S \ref{ssec:HCplancherel}. Let $U_\sigma=\{\tau\in \Im \Lambda_M: \sigma\cong \sigma_\tau\}.$ It is a finite group. We have a group action 
\begin{align}\label{eq:Usigmaaction}
\begin{split}
    U_\sigma\times \mathrm{End}_\CC (\Ind_{K_P}^K \sigma|_{K_M})&\longrightarrow \mathrm{End}_\CC (\Ind_{K_P}^K \sigma|_{K_M})\\
    (\tau,T)&\longmapsto  \tau.T.
\end{split}
\end{align}
For each $\tau\in U_\sigma,$ fix an $M(F)$-equivariant isomorphism $\varphi_\tau:\sigma\cong \sigma_\tau$. This induces an isomorphism $\sigma_\chi\cong \sigma_{\chi\tau}$ for all $\chi\in \Lambda_M.$ For $i\in Y_{(M,\sigma)}$ and $\tau\in U_\sigma,$ let $\sigma_{i,\tau}$ denote the maximal semisimple subrepresentation of $\Ind_{P\cap M_i}^{M_i} \sigma_{\tau\chi_i}$ whose irreducible subrepresentations are all isomorphic to $\sigma_i$. Let $V_{i,\tau}:=\left(\Ind_{P\cap M_i}^{M_i} \sigma_{\tau\chi_i}\right)/\sigma_{i,\tau}.$ It is nonzero by \hyperref[cor:notfull]{Corollary \ref{cor:notfull}}. For each $\tau,\tau'\in U_\sigma,$ $\varphi_\tau$ induces $M_i(F)$-equivariant isomorphisms
\begin{align}\label{eq:tauiso}
    &\sigma_{i,\tau'}\cong (\Ind^{M_i}_{P\cap M_i}\varphi_\tau)(\sigma_{i,\tau'})=\sigma_{i,\tau\tau'}, \quad V_{i,\tau'}\cong (\Ind^{M_i}_{P\cap M_i}\varphi_\tau)(V_{i,\tau'})=V_{i,\tau\tau'}.
\end{align} 
Identify $\mathrm{Ind}_{K_{P_i}}^K \sigma_{i,\tau}|_{K_{M_i}}$ as a subspace of $\mathrm{Ind}_{K_{P}}^{K} \sigma|_{K_M}$ via the inclusion $\Ind_{P_i}^G\sigma_{i,\tau}\le \Ind_{P}^G\sigma_{\tau\chi_i},$ and identify $\mathrm{Ind}_{K_{P_i}}^{K}V_{i,\tau}^\lor$ as a subspace of $\mathrm{Ind}_{K_P}^K \sigma^\lor|_{K_M}$ via the quotient map $\Ind_{P}^G \sigma_{\tau\chi_i}\to \Ind_{P_i}^G V_{i,\tau}.$ Note that if $\tau\in \Lambda_{M_i},$ then isomorphisms in \eqref{eq:tauiso} preserve the underlying spaces. In particular, in this case we have an identity of the underlying vector spaces
\begin{align*}
    \mathrm{Ind}_{K_{P_i}}^{K}V_{i,\tau'}^\lor=\mathrm{Ind}_{K_{P_i}}^{K}V_{i,\tau\tau'}^\lor.
\end{align*}

For each subset $\emptyset\neq   S\subseteq Y_{(M,\sigma)},$ 
put
\begin{align*}
    \kappa_S := \left\{U_\sigma \bigcap_{i\in S} \tau_{i}\chi_i\Lambda_{M_i}\subseteq\Lambda_M: \bigcap_{i\in S} \tau_{i}\chi_i\Lambda_{M_i}\neq \emptyset, \tau_i\in U_\sigma\right\},
\end{align*}
and let $\kappa_\emptyset:=\{\Lambda_M\}$. Note that if $C,C'$ are distinct sets in $\kappa_S,$ then $C\cap C'=\emptyset.$ For each $C\in \kappa_S$, fix a choice $\underline{\tau}_C=(\tau_{i,C})_{i\in S}\in U_\sigma^{\# S}$ such that $C=U_\sigma E$ with $E=\bigcap_{i\in S} \tau_{i,C}\chi_i\Lambda_{M_i}.$ 

\begin{lemma}\label{lem:cuspnonzero}
 Suppose $S$ is nonempty. For $\chi\in \bigcap_{i\in S}\tau_{i,C}\chi_i\Lambda_{M_i},$ $\Ind_P^G \sigma_\chi\ge \sum_{i\in S} \Ind_{P_i}^G (\sigma_{i,\tau_{i,C}})_{\chi(\tau_{i,C}\chi_i)^{-1}}$ is a proper containment. In particular,
\begin{align*}
    \bigcap_{i\in S} \Ind_{K_{P_i}}^K V_{i,\tau_{i,C}}^\vee \neq 0.
\end{align*}
\end{lemma}
\begin{proof} 
Write $\chi'_i:=\chi(\tau_{i,C}\chi_i)^{-1}\in \Lambda_{M_i}.$ Choose $s_0\in W(G|M)$ such that
    \begin{align*}
        \Re(s_0.\chi) \in - \overline{X^\ast(M)}^{+}_\RR =\mathbb{R}_{\leq0}\{\omega_\alpha: \alpha\in\Delta(P)\} + X^*(G)_\RR.
    \end{align*}
    Let $i\in S$ and $s\in W^{M_i,M}$ such that $M\le s^{-1}.M_i$.
    Then
    \begin{align*}
        \Re(s_0.\chi) \in \mathbb{R}_{\leq0}\{\omega_\alpha: \alpha\in\Delta(P\cap s^{-1}.M_i)\} + X^*(s^{-1}.M_i)_\RR \subseteq X^*(M)_\RR
    \end{align*}
    By \hyperref[CasselmanCriterion]{Casselman's criterion},
    \begin{align*}
        \Re\mathcal{E}xp  \,\left(r^{s^{-1}.M_i}_{P\cap s^{-1}.M_i}s^{-1}.(\sigma_{i,\tau_{i,C}})_{\chi_i'}\right)\subseteq X^*(s^{-1}.M_i)_\RR+\RR_{>0}\Delta(P\cap s^{-1}.M_i).
    \end{align*}
    It follows from the \hyperref[geometriclemma:2]{Geometric lemma} that $$\Re (s_0.\chi)\in \Re \mathcal{E}xp \left(r^G_P\circ \mathrm{Ind}_P^G \sigma_\chi\right)-\bigcup_{i\in S} \Re \mathcal{E}xp \left(r^G_P\circ \Ind_{P_i}^G (\sigma_{i,\tau_{i,C}})_{\chi_i'}\right),$$ 
    and the lemma follows.
\end{proof}

Recall the action \eqref{eq:Usigmaaction}. Let $U_\sigma$ act diagonally on $\CC[\Lambda_M]\otimes \mathrm{End}_\CC (\Ind_{K_P}^K \sigma|_{K_M})$.

\begin{lemma}\label{lem:varying}
    There exists a family of holomorphic sections 
    \begin{align*}
        \mathfrak{H}=\{T_{S,C}:S\subseteq Y_{(M,\sigma)}, C\in \kappa_S\}\subseteq \left(\CC[\Lambda_M]\otimes \mathrm{End}_\CC (\Ind_{K_P}^K \sigma|_{K_M})\right)^{U_\sigma}
    \end{align*}
    such that for each $i\in Y_{(M,\sigma)},  v\in \mathrm{Ind}_{K_{P_i}}^K \sigma_{i,1}|_{K_{M_i}}$ and $\chi\in \chi_i\Lambda_{M_i}$
    \begin{align*}
        T_{S,C}(\chi)(v)=0
    \end{align*}
    for all $(S,C).$ 
    
    Moreover, if we let $\mathcal{H}$ denote the collection of all such families, the ideal in $\CC[\Lambda_M]$ generated by the set 
    \begin{align*}
        \{\Lambda_{M}\ni \chi \mapsto \langle T_{S,C}(\chi)(v),w \rangle: (v,w)\in \mathrm{Ind}_{K_{P}}^K \sigma|_{K_M}\times \mathrm{Ind}_{K_{P}}^K \sigma^\lor|_{K_M}, T_{S,C}\in \mathfrak{H}, \mathfrak{H}\in \mathcal{H}\}
    \end{align*}
    is unital. 
\end{lemma}

\begin{proof}      
  Choose any $U_\sigma$-eigenvector $T\in \mathrm{End}_\CC (\Ind_{K_P}^K \sigma|_{K_M})$ and $h\in \CC[\Lambda_M]^\times$ such that $T'_{\emptyset,\Lambda_M}:=hT\in (\CC[\Lambda_M]\otimes \mathrm{End}_\CC (\Ind_{K_P}^K \sigma|_{K_M}))^{U_\sigma}.$ If $J_{(M,\sigma)}$ is empty, take $$\mathfrak{H}:=\{T_{\emptyset,\Lambda_M}:=T'_{\emptyset,\Lambda_M}\},$$ and we are done.
 
 Thus we assume $J_{(M,\sigma)}$ is nonempty. Let $S\neq \emptyset$ and $C\in \kappa_S$. Then $C=U_\sigma E_C,$ where $E_C:=\bigcap_{i\in S} \tau_{i,C}\chi_i\Lambda_{M_i}.$  
Let $U'_\sigma\le U_\sigma$ be the subgroup that stabilizes $E_C$. Then $U'_\sigma$ is a subgroup of $\bigcap_{i\in S}\Lambda_{M_i}.$  Therefore, for all $\tau\in U'_{\sigma}$, we have identities of underlying vector spaces
\begin{align*}
    \bigcap_{i\in S}\mathrm{Ind}_{K_{P_i}}^{K}V_{i,\tau_{i,C}}^\lor=\bigcap_{i\in S}\mathrm{Ind}_{K_{P_i}}^{K}V_{i,\tau\tau_{i,C}}^\lor.
\end{align*}
By \hyperref[lem:cuspnonzero]{Lemma \ref{lem:cuspnonzero}} we can choose $0\neq w_{S,C}\in \bigcap_{i\in S}\mathrm{Ind}_{K_{P_i}}^{K}V_{i,\tau_{i,C}}^\lor \subseteq\mathrm{Ind}_{K_{P}}^{K} \sigma^\lor|_{K_M}$ so that
\begin{align*}
\langle v,w_{S,C}\rangle=0 \quad\textrm{ for any } v\in \sum_{i\in S}\mathrm{Ind}_{K_{P_i}}^{K} \sigma_{i,\tau_{i,C}}|_{K_{M_i}}.
\end{align*}
Choose $v_{S,C}\in \Ind_{K_P}^K\sigma|_{K_M}$ such that $\langle v_{S,C},w_{S,C}\rangle=1$. Consider the morphism $T'''_{S,C}\in \mathrm{End}_\CC (\Ind_{K_P}^K \sigma|_{K_M})$ corresponding to $v_{S,C}\otimes w_{S,C}\in \mathrm{Ind}_{K_P}^K\sigma|_{K_M}\otimes \mathrm{Ind}_{K_P}^K\sigma^\lor|_{K_M}.$ Choose a $U_\sigma'$-eigenvector 
\begin{align*}
    T\in \mathrm{span}_\CC\{ \tau.T'''_{S,C}: \tau\in U_{\sigma}'\}
\end{align*}
and $h\in \CC[\Lambda_M]^\times$ such that $T''_{S,C}:=hT\in (\CC[\Lambda_M]\otimes \mathrm{End}_\CC (\Ind_{K_P}^K \sigma|_{K_M}))^{U_\sigma'}.$ Then for any $\chi\in \Lambda_M$
\begin{align*}
    T_{S,C}''(\chi)(v)=0\quad\textrm{ for any } v\in \sum_{i\in S}\mathrm{Ind}_{K_{P_i}}^{K} \sigma_{i,\tau_{i,C}}|_{K_{M_i}},
\end{align*}
and there exists $(v_0,w_0)\in \Ind_{K_{P}}^K \sigma|_{K_{M}}\times \Ind_{K_{P}}^K \sigma^\lor|_{K_{M}}$ such that $\langle T_{S,C}''(\chi)(v_0),w_0\rangle\neq 0$ for any $\chi\in \Lambda_M.$

Let $\eta_{S,C}$ be a function in $\CC[\Lambda_M]^{U_\sigma'}$ whose vanishing locus $V(\eta_{S,C})$ in $\Lambda_M$ contains $U_\sigma \chi_i\Lambda_{M_i}-\tau_{i,C}\chi_i\Lambda_{M_i}$ for all $i\in S,$ and $V(\eta_{S,C})\cap E_C=\emptyset.$ Note that $V(\eta_{S,C})\supseteq C-E_C.$ Put
\begin{align*}
    T_{S,C}':=\sum_{\tau\in U_\sigma/U_{\sigma}'} \tau. \bigg(\eta_{S,C} T''_{S,C}\bigg).
\end{align*}
Then $T'_{S,C}\in \left(\CC[\Lambda_M]\otimes \mathrm{End}_\CC (\Ind_{K_P}^K \sigma|_{K_M})\right)^{U_\sigma}$ satisfies the following properties: 
\begin{itemize}
    \item For any $i\in S$ and $\chi\in \chi_i\Lambda_{M_i}$
\begin{align*}
    T_{S,C}'(\chi)(v)=0\quad\textrm{ for any } v\in \mathrm{Ind}_{K_{P_i}}^{K} \sigma_{i,1}|_{K_{M_i}}.
\end{align*}
    \item One has 
    \begin{align*}
        \langle T_{S,C}'(\chi)(v_0),w_0\rangle\neq 0.
    \end{align*}
    for any $\chi\in C$.
\end{itemize}

For each $i\in Y_{(M,\sigma)},$ let $\{f_{i,1},\ldots,f_{i,p_i}\}\subseteq\CC[\Lambda_M]^{U_\sigma}$ be a generating set of the defining ideal of $U_\sigma\chi_i\Lambda_{M_i}$, viewed as a Zariski closed subspace of $U_\sigma\backslash \Lambda_M$. Fix a map $\xi:Y_{(M,\sigma)}\to \ZZ_{\geq 1}$ with $1\leq \xi(i)\leq p_i$. Define a partial order $\le$ on the finite set $\{(S,C): S\subseteq Y_{(M,\sigma)}, C\in \kappa_S\}$ by specifying $(S,C)\le (S',C')$ if and only if $C'\subseteq C$. Note that if $(S,C)\le (S',C'),$ then $C'\in \kappa_{S\cup S'}$ and $(S,C)\le (S\cup S', C').$

    If $(S,C)$ is not maximal (with respect to $\leq$) with $S\neq\emptyset$, let
    \begin{align*}
        h_{S,C,\xi}:=\prod_{\substack{i\in Y_{(M,\sigma)}\\ C\cap U_\sigma \chi_i\Lambda_{M_i}\subsetneq C}}f_{i,\xi(i)}=\prod_{\substack{i\notin S\\ C\cap U_\sigma \chi_i\Lambda_{M_i}\subsetneq C}} f_{i,\xi(i)}.
    \end{align*}
    For $(S,C)$ maximal, one has $C\cap U_\sigma\chi_i\Lambda_{M_i} = \emptyset$ for $i\not\in S$. We choose $h_{S,C,\xi}$ that is nonvanishing on $C$ and is zero on $U_\sigma\chi_i\Lambda_{M_i}$ for any $i\not\in S.$ 
    
    Let $\{f_{1},\ldots,f_{p_\emptyset}\}\subseteq \CC[\Lambda_M]^{U_\sigma}$ be a generating set of the defining ideal of $\bigcup_{i\in Y_{(M,\sigma)}}U_\sigma \chi_i\Lambda_{M_i}.$ Extend the domain of $\xi$ to  $\{\emptyset\} \cup Y_{(M,\sigma)}$ with $1\le \xi(\emptyset)\le p_{\emptyset}.$ Let $h_{\emptyset,\Lambda_M,\xi}:=f_{\xi(\emptyset)}.$ Define for any $(S,C),\xi$
    \begin{align*}
       T_{S,C,\xi}:=h_{S,C,\xi}T'_{S,C}.
    \end{align*}
     Consider the families indexed by $\xi$
    \begin{align*}
        \mathfrak{H}_{\xi}:=\{ T_{S,C,\xi}: S\subseteq Y_{(M,\sigma)}, C\in \kappa_S\}.
    \end{align*}
    It follows by construction and Hilbert's Nullstellensatz that the ideal in $\CC[\Lambda_M]$ generated by the set 
    \begin{align*}
        \{\Lambda_{M}\ni \chi \mapsto \langle T_{S,C,\xi}(\chi)(v),w \rangle: (v,w)\in \mathrm{Ind}_{K_{P}}^K \sigma|_{K_M}\times \mathrm{Ind}_{K_{P}}^K \sigma^\lor|_{K_M}, T_{S,C,\xi}\in \mathfrak{H}_\xi, \xi\}
    \end{align*}
    is unital. In addition, given $(S,C),\xi,$ for any $i\in Y_{(M,\sigma)}, v\in \mathrm{Ind}_{K_{P_i}}^K \sigma_{i,1}|_{K_{M_i}}$ and $\chi\in \chi_i\Lambda_{M_i},$ we have $$T_{S,C}'(\chi)(v)=\delta_{i\notin S}T_{S,C}'(\chi)(v).$$
     On the other hand, by the definition of $h_{S,C,\xi}$ we have  $h_{S,C,\xi}(\chi)=0$ for $\chi\in \chi_i\Lambda_{M_i}$ if $i\not\in S$. Thus 
    \begin{align*}
    T_{S,C,\xi}(\chi)(v)=h_{S,C,\xi}(\chi)T_{S,C}'(\chi)(v)=0.
    \end{align*}
    This proves the lemma.
\end{proof}

Let $\mathfrak{H}$ be a family satisfying the conclusion of \hyperref[lem:varying]{Lemma \ref{lem:varying}}. Let $T_{S,C}\in \mathfrak{H}.$ Note that $\chi\mapsto L(1/2,\sigma_\chi,\rho_M)$ is $U_\sigma$-invariant, so 
\begin{align*}
    \Im \Lambda_{M}\ni\chi\mapsto L(1/2,\sigma_\chi,\rho_M)T_{S,C}(\chi)\in C^\infty(\mathcal{O}_\sigma, P).
\end{align*}
By the 
\hyperref[HCPlan]{Harish-Chandra Plancherel theorem}, there is a unique function $f_{S,C}\in \mathcal{C}(G(F))$ such that $\mathrm{HP}(f_{S,C})$ is supported on $(M,\mathcal{O}_\sigma)$ and
\begin{align*}
    \mathrm{HP}(f_{S,C})(M,\sigma_\chi)=L(1/2,\sigma_\chi,\rho_M)T_{S,C}(\chi).
\end{align*}
Clearly, $f_{S,C}\in \mathcal{S}_\rho^{\mathrm{as}}(G(F)).$

\begin{lemma}\label{lem:cuspconst}
We have $f_{S,C}\in \mathcal{S}^{\mathrm{as}}_\rho(G(F))\cap C^\infty_{\mathrm{ac}}(G(F)).$
\end{lemma} 

\begin{proof} It remains to show $f_{S,C}\in C^\infty_{\mathrm{ac}}(G(F)).$ By \hyperref[Lfunction:expand:cone]{Lemma \ref{Lfunction:expand:cone}},  we can write 
\begin{align*}
       L(1/2,\sigma_\chi,\rho_M) = \sum_{\alpha\in\widetilde{\sigma}_{\rho,G}\cap \Omega_G} c_\alpha(\sigma_{|\nu|^{1/2}\chi})
\end{align*}
as a series that converges absolutely for $\mathrm{Re}(\chi)\in |\nu|^{-1/2}C_{\rho,M}.$ Let $\alpha\in \widetilde{\sigma}_{\rho,G}\cap \Omega_G.$ Note that $c_\alpha\in \CC[\mathcal{O}_{\sigma\CC}].$ Let $f_\alpha$ be the unique function in $\mathcal{S}^{\mathrm{as}}_\rho(G(F))$ such that $\mathrm{HP}(f_\alpha)$ is supported on $(M,\mathcal{O}_\sigma)$ and
\begin{align*}
    \mathrm{HP}(f_\alpha)(M,\sigma_\chi)=c_\alpha(\sigma_{|\nu|^{1/2}\chi})T_{S,C}(\chi).
\end{align*}
By \hyperref[BH:PW]{Theorem \ref{BH:PW}} there is a unique $f\in C^\infty_c(G(F))$ such that $\HP(f)|_{\Temp_{\Ind,0}(G)}=\HP(f_{\alpha}).$ We claim $\HP(f)=\HP(f_\alpha)$ and thus $f=f_\alpha$ by \hyperref[HCPlan]{Harish-Chandra Plancherel theorem}. Assuming the claim, since $T_{S,C}$ is a holomorphic section, for $\beta\in \Omega_G$
\begin{align*}
    e_\beta\HP(f_{S,C})(M,\sigma_\chi) = \sum_\alpha e_\beta \HP(f_\alpha)(M,\sigma_\chi)
\end{align*}
is a finite sum. By \eqref{HP:orthoproj:alpha} we conclude $f_{S,C}\mathbf{1}_\beta = \sum_\alpha f_\alpha\mathbf{1}_\beta\in C_c^\infty(G(F))$ for all $\beta\in \Omega_G$. This proves $f_{S,C}\in C^\infty_{\mathrm{ac}}(G(F))$.

    To prove the claim, observe that 
    \begin{align*}
        \supp \HP(f)\subseteq (M,\mathcal{O}_\sigma)\cup \bigcup_{i\in Y_{(M,\sigma)}}  (M_i,\mathcal{O}_{\sigma_i}).
    \end{align*}
    If $J_{(M,\sigma)}$ is empty, then we are done. Otherwise, we only need to show $\HP(f)$ vanishes on $(M_i,\mathcal{O}_{\sigma_i})$ for all $i$. For any $(v,w)\in \Ind_{K_{P}}^K \sigma|_{K_M}\times \Ind_{K_{P}}^K \sigma^\lor|_{K_M},$ we have
    \begin{align*}
        Z(\mathrm{\Ind}_{P}^G \sigma_\chi,f,v,w)=c_{\alpha}(\sigma_{|\nu|^{1/2}\chi})\langle T_{S,C}(\chi)(v),w\rangle.
    \end{align*}
    By \hyperref[lem:varying]{Lemma \ref{lem:varying}} this vanishes if $\chi\in \chi_i\Lambda_{M_i}$ and $v\in \mathrm{Ind}_{K_{P_i}}^K \sigma_{i,1}|_{K_{M_i}}.$ Since $f\in C^\infty_c(G(F)),$  by  \hyperref[lem:cs:subquo]{Lemma \ref{lem:cs:subquo}} $\HP(f)$ vanishes on $(M_i,\mathcal{O}_{\sigma_i})$. This completes the proof.
\end{proof}

    For $(v,w)\in \mathrm{Ind}_{K_P}^K  \sigma|_{K_M}\times \mathrm{Ind}_{K_P}^K \sigma^\lor|_{K_M},$
    \begin{align*}
    Z(\mathrm{Ind}_P^G\sigma_\chi,f_{S,C},v,w)=L(1/2,\sigma_\chi,\rho_{M})\langle T_{S,C}(\chi)(v),w\rangle.
    \end{align*}
    Therefore, $\langle T_{S,C}(|\nu|^{-1/2}\chi)(v),w\rangle\in I_{(M,\sigma)}$ for all $S,C$. Then \hyperref[multigcd]{Theorem \ref{multigcd}} for supercuspidal representations follows from \hyperref[lem:ideal]{Lemma \ref{lem:ideal}} and \hyperref[lem:varying]{Lemma \ref{lem:varying}} by varying $v,w,T_{S,C},\mathfrak{H}.$ \qed

\subsection{Proof of \hyperref[multigcd]{Theorem \ref{multigcd}} for non-supercuspidal discrete series}\label{ssec:multigcd:noncuspidal}

Assume \eqref{LLC:gamma} holds for all semi-standard Levi subgroups of $G$. The proof of \hyperref[multigcd]{Theorem \ref{multigcd}} for $\sigma'\in \Pi_2(M')-\Pi_0(M')$ is similar but slightly more involved. Retain the notations in \S \ref{ssec:multigcd:cuspidal}. Let $M< M'$ be the unique standard Levi subgroup such that $\sigma'$ is isomorphic to a subrepresentation of $\Ind_{P
\cap M'}^{M'}\sigma_{\lambda}$ for some $\sigma\in \Pi_0(M)$ and $\lambda\in \Lambda_{M}.$   Let $i_0\in Y_{(M,\sigma)}$ be the index such that $(M',\mathcal{O}_{\sigma'})=(M_{i_0},\mathcal{O}_{\sigma_{i_0}}),$ and we may assume $\lambda=\chi_{i_0}.$ Recall that $\sigma'_1$ is the maximal semisimple subrepresentation of $\Ind_{P\cap M'}^{M'} \sigma_{\lambda}$ whose irreducible subrepresentations are all isomorphic to $\sigma'$.

Let
\begin{align*}
    Y'_{(M,\sigma)} := \left\{i\in Y_{(M,\sigma)}:\begin{array}{l}
         \text{either $M_i\not\le M'$, or}  \\
          \text{$M_i\le M'$ and $\sigma'$ is not a subquotient of $\Ind^{M'}_{P_i\cap M'} (\sigma_i)_{\lambda_i}$ for any $\lambda_i\in \Lambda_{M_i}.$}
    \end{array}\right\}
\end{align*}
and $J_{(M,\sigma)}' := \{(M_i,\mathcal{O}_{\sigma_i})\in J_{(M,\sigma)}: i\in Y'_{(M,\sigma)} \}$. Let $S\subseteq Y'_{(M,\sigma)}$ and $C\in \kappa_{\{i_0\}\cup S}.$ Recall we have chosen $\underline{\tau}_C=(\tau_{i,C})_{i\in \{i_0\}\cup S}\in U_\sigma^{1+\#S}$ so that $C=U_{\sigma}E$ with $E=\bigcap_{i\in \{i_0\}\cup S} \tau_{i,C}\chi_i\Lambda_{M_i}.$ We normalize $\underline{\tau}_{C}$ so that  $\tau_{i_0,C}=1.$ 

\begin{lemma}\label{lem:discretenonzero}
Fix any $M'(F)$-equivariant embedding $\sigma'\to \sigma'_1$. Let $\emptyset\neq S\subseteq Y'_{(M,\sigma)}$ and $C\in \kappa_{\{i_0\}\cup S}.$ For $\chi\in \lambda^{-1}\bigcap_{i\in \{i_0\}\cup S} \tau_{i,C}\chi_i\Lambda_{M_i},$
\begin{align*}
    \mathrm{Ind}_{P'}^G \sigma_{\chi}'\not\leq \sum_{i\in S}\mathrm{Ind}_{P_i}^{G} (\sigma_{i,\tau_{i,C}})_{\chi\lambda(\tau_{i,C}\chi_i)^{-1}}
\end{align*}
as subspaces of $\mathrm{Ind}_{P}^G \sigma_{\lambda\chi}.$
\end{lemma}

\begin{proof}
     By the \hyperref[geometriclemma:2]{Geometric lemma}, $r^G_{P'}\circ \mathrm{Ind}_{P'}^G \sigma'_\chi$
     contains $s.\sigma'_\chi$ as subquotients for all $s\in W(G|M').$ Choose $s_0$ such that 
     \begin{align*}
         \mathrm{Re}(s_0.\sigma'_\chi)=\mathrm{Re}(s_0.\chi) \in -\overline{X^\ast(M')}_\RR^+= \mathbb{R}_{\leq 0}\{\omega_\alpha:\alpha\in \Delta(P')\} + X^*(G)_\RR.
     \end{align*}
     For $i\in S,$ write $\chi_i':=\chi\lambda(\tau_{i,C}\chi_i)^{-1}\in \Lambda_{M_i}.$ 

    Let $S':=\{ i\in S: M_i\not\le M'\}$. 
    For $i\in S'$, by \hyperref[CasselmanCriterion]{Casselman's criterion} and the \hyperref[geometriclemma:2]{Geometric lemma}, 
    \begin{align*}
        \Re \mathcal{E}xp\left( r_{P}^G\circ\mathrm{Ind}_{P_i}^G (\sigma_{i,\tau_{i,C}})_{\chi_i'}\right)\subseteq\bigcup_{\substack{s\in W^{M_i,M}\\ M\leq s^{-1}.M_i}} X^*(s^{-1}.M_i)_\RR+\RR_{>0}\Delta(P\cap s^{-1}.M_i).
    \end{align*} 
    By \hyperref[CasselmanCriterion]{Casselman's criterion} again, we have
    \begin{align*}
        \Re \mathcal{E}xp(r^{M'}_{P\cap M'} s_0.\sigma'_\chi) \subseteq \Re(s_0.\chi) + \mathbb{R}_{>0}\Delta(P\cap M').
    \end{align*}
    Since $M_i\not\le M',$ by the choice of $s_0$ we have
    \begin{align*}
    \Re \mathcal{E}xp(r^{M'}_{P\cap M'} s_0.\sigma'_\chi)\not\subseteq  \Re \mathcal{E}xp\left( r_{P}^G\circ\mathrm{Ind}_{P_i}^G (\sigma_{i,\tau_{i,C}})_{\chi_i'}\right).
    \end{align*}
    Hence, $s_0.\sigma'_\chi$ is not a subquotient of $r_{P'}^G\circ \Ind_{P_i}^G (\sigma_{i,\tau_{i,C}})_{\chi_i'}.$ Thus $\mathrm{s.s.}\left(\Ind_{P'}^G \sigma'_\chi\right)\not\le \mathrm{s.s.}\left(\sum_{i\in S'} \Ind_{P_i}^G (\sigma_{i,\tau_{i,C}})_{\chi_i'}\right).$
    
    For $i\in S-S'$ so that $M_i\le M',$ by definition $\sigma'$ is not a subquotient of $\Ind_{P_i\cap M'}^{M'} (\sigma_{i,\tau_{i,C}})_{\lambda_i}$ for any $\lambda_i\in \Lambda_{M_i}.$  As both $\sigma'_\chi$ and $\Ind_{P_i\cap M'}^{M'}(\sigma_{i,\tau_{i,C}})_{\chi_i'}$ are subpresentations of $\Ind_{P\cap M'}^{M'} \sigma_{\lambda\chi},$ we have $\sigma'_\chi\cap \sum_{i\in S-S'}\Ind_{P_i\cap M'}^{M'}(\sigma_{i,\tau_{i,C}})_{\chi_i'}=0.$  Now suppose on the contrary that $\mathrm{Ind}_{P'}^G \sigma'_\chi\le  \sum_{i\in S} \mathrm{Ind}_{P_i}^G (\sigma_{i,\tau_{i,C}})_{\chi_i'}.$ Then we have an injection
    \begin{align*}
        \mathrm{Ind}_{P'}^G \sigma'_\chi&\hookrightarrow{} \left( \sum_{i\in S} \mathrm{Ind}_{P_i}^G (\sigma_{i,\tau_{i,C}})_{\chi_i'}\right)\bigg/ \left(\sum_{i\in S-S'} \mathrm{Ind}_{P_i}^G (\sigma_{i,\tau_{i,C}})_{\chi_i'}\right)\\
        &\quad\cong \left( \sum_{i\in S'} \mathrm{Ind}_{P_i}^G (\sigma_{i,\tau_{i,C}})_{\chi_i'}\right)\bigg/ \left(\left(\sum_{i\in S-S'} \mathrm{Ind}_{P_i}^G (\sigma_{i,\tau_{i,C}})_{\chi_i'}\right)\cap \left(\sum_{i\in S'} \mathrm{Ind}_{P_i}^G (\sigma_{i,\tau_{i,C}})_{\chi_i'}\right)\right).
    \end{align*}
    This contradicts $\mathrm{s.s.}\left(\Ind_{P'}^G \sigma'_\chi\right)\not\le \mathrm{s.s.}\left(\sum_{i\in S'} \Ind_{P_i}^G (\sigma_{i,\tau_{i,C}})_{\chi_i'}\right),$ and the lemma is justified. 
\end{proof}

\begin{lemma}\label{lem:varying:discrete}
There exists a family of holomorphic sections 
    \begin{align*}
        \mathfrak{H}=\{T_{S,C}:S\subseteq Y_{(M,\sigma)}', C\in \kappa_{\{i_0\}\cup S}\}\subseteq (\CC[\Lambda_M]\otimes \mathrm{End}_\CC (\Ind_{K_P}^K \sigma|_{K_M}))^{U_\sigma}
    \end{align*}
    such that for each $i\in Y_{(M,\sigma)}',v\in \mathrm{Ind}_{K_{P_i}}^K \sigma_{i,1}|_{K_{M_i}}$ and $\chi\in \chi_i\Lambda_{M_i}$
    \begin{align*}
        T_{S,C}(\chi)(v)=0
    \end{align*}
    for all $(S,C).$ 
    
    Consider the collection $\mathcal{H}$ of all such families. The ideal in $\CC[\Lambda_{M'}]$ generated by the set 
    \begin{align*}
        \big\{\Lambda_{M'}\ni \chi\mapsto \langle T_{S,C}(\lambda\chi)(v),w \rangle: (v,w)\in \mathrm{Ind}_{K_{P}}^K \sigma'_{1}|_{K_M}\times \mathrm{Ind}_{K_{P}}^K \sigma^\lor|_{K_M}, T_{S,C}\in \mathfrak{H}, \mathfrak{H}\in \mathcal{H}\big\}
    \end{align*}
    is unital. 
\end{lemma}
\begin{proof}
    Let $S\subseteq Y'_{(M,\sigma)}, C\in \kappa_{\{i_0\}\cup S}$ be given. In the construction of $T'''_{\{i_0\}\cup S,C}$ in the proof of \hyperref[lem:varying]{Lemma \ref{lem:varying}},  by  \hyperref[lem:discretenonzero]{Lemma \ref{lem:discretenonzero}} we can choose instead $v_{\{i_0\}\cup S,C}\in \Ind_{K_{P'}}^{K} \sigma'_1|_{K_M'}$ and $w_{\{i_0\}\cup S,C}\in \bigcap_{i\in S} \Ind_{K_{P_i}}^{K}V_{i,\tau_{i,C}}^\lor$ such that $\langle v_{\{i_0\}\cup S ,C},w_{\{i_0\}\cup S,C}\rangle=1.$ The rest of the proof is similar and we leave the details to the reader.
\end{proof}

Let
\begin{align*}
    \widetilde{Y}_{(M,\sigma)} := \left\{i\in Y_{(M,\sigma)}: M_i<M'\text{ and $\sigma'$ is a subquotient of $\Ind_{P_i\cap M'}^{M'} (\sigma_i)_{\lambda_i}$ for some $\lambda_i\in \Lambda_{M_i}$}\right\}
\end{align*}
and $\widetilde{J}_{(M,\sigma)} = \{(M_i,\mathcal{O}_{\sigma_i}):i\in\widetilde{Y}_{(M,\sigma)} \}$. By definition, we have
\begin{align*}
    Y_{(M,\sigma)} &= \{i_0\}\cup Y'_{(M,\sigma)}\cup \widetilde{Y}_{(M,\sigma)},\\
    J_{(M,\sigma)} &= \{ (M',\mathcal{O}_{\sigma'})\}\sqcup J_{(M,\sigma)}'\sqcup  \widetilde{J}_{(M,\sigma)}.
\end{align*}

By \hyperref[lem:auxL]{Lemma \ref{lem:auxL}}, there exists a function $L\in \CC(\mathcal{O}_{\sigma\CC})^{W(M,\mathcal{O}_{\sigma\CC})}$ with the following properties:
\begin{itemize}
    \item $\Lambda_M\ni \chi \mapsto L(\sigma_\chi)\in L(1/2,\sigma_\chi,\rho_M)\CC[\Lambda_M].$
    \item $\Lambda_{M_i}\ni \chi \mapsto  L(\sigma_{\lambda_i\chi})\in L(1/2,(\sigma_i)_\chi,\rho_{M_i})\CC[\Lambda_{M_i}]$ for all $i\in \widetilde{Y}_{(M,\sigma)}.$
    \item $\Lambda_{M'}\ni \chi \mapsto L(\sigma_{\lambda\chi})\in L(1/2,\sigma'_\chi,\rho_{M'})\CC[\Lambda_{M'}]^\times$.
\end{itemize}
By \hyperref[Lfunction:expand:cone]{Lemma \ref{Lfunction:expand:cone}}, there is a finite set $X\subset \Omega_G$ such that
\begin{align*}
       L(\sigma_\chi) = \sum_{\alpha\in X + \widetilde{\sigma}_{\rho,G}\cap \Omega_G} c_\alpha(\chi),
\end{align*}
and the series converges absolutely for $\mathrm{Re}(\chi)$ in a positive cone in $\Lambda_{M}.$ Let $\mathfrak{H}$ be a family satisfying the conclusion of \hyperref[lem:varying:discrete]{Lemma \ref{lem:varying:discrete}}. Let $S\subseteq Y'_{(M,
\sigma)}$ and $C\in \kappa_{\{i_0\}\cup S}.$ Let $\alpha\in X+ \widetilde{\sigma}_{\rho,G}\cap \Omega_G.$ Note that $c_\alpha\in \CC[\Lambda_M]^{U_\sigma}=\CC[\mathcal{O}_{\sigma\CC}].$ By \hyperref[BH:PW]{Theorem \ref{BH:PW}} there is a unique $f_\alpha\in C^\infty_c(G(F))$ such that $\mathrm{HP}(f_\alpha)|_{\mathrm{Temp}_{\mathrm{Ind},0}(G)}$ is supported on $(M,\mathcal{O}_\sigma)$ and
\begin{align*}
    \mathrm{HP}(f_\alpha)(M,\sigma_{\chi})=c_\alpha(\chi)T_{S,C}(\chi).
\end{align*}
Then 
\begin{align*}
    \supp \mathrm{HP}(f_\alpha)\subseteq (M,\mathcal{O}_{\sigma})\cup (M',\mathcal{O}_{\sigma'})\cup \bigcup_{i\in  Y'_{(M,\sigma)}\sqcup\widetilde{Y}_{(M,\sigma)}} (M_i,\mathcal{O}_{\sigma_i}).
\end{align*}

For any $(v,w)\in \Ind_{K_{P}}^K \sigma|_{K_{M}}\times \Ind_{K_{P}}^K \sigma^\lor|_{K_{M}},$ we have for $\chi\in \Lambda_M$
    \begin{align*}
        Z(\mathrm{\Ind}_{P}^G \sigma_{\chi},f_\alpha,v,w)=c_{\alpha}(\chi)\langle T_{S,C}(\chi)(v),w\rangle.
    \end{align*}
    For $i\in Y_{(M,\sigma)}', v\in \mathrm{Ind}_{K_{P_i}}^K \sigma_{i,1}|_{K_{M_i}}$ and $\chi\in \chi_i\Lambda_{M_i},$ $ Z(\mathrm{\Ind}_{P}^G \sigma_{\chi},f_\alpha,v,w)=0$. Since $f_\alpha\in C^\infty_c(G(F)),$  by  \hyperref[lem:cs:subquo]{Lemma \ref{lem:cs:subquo}} $\HP(f_\alpha)$ vanishes on $(M_i,\mathcal{O}_{\sigma_i})$ for $i\in Y'_{(M,\sigma)}$. On the other hand, by the choice of $L$ and \hyperref[lem:cs:subquo]{Lemma \ref{lem:cs:subquo}}, we have
\begin{align*}
    f_{S,C}:=\sum_{\alpha\in X+\widetilde{\sigma}_{\rho,G}\cap \Omega_G} f_\alpha\in \mathcal{S}^{\mathrm{as}}_\rho(G(F))\cap C^\infty_{\mathrm{ac}}(G(F)).
\end{align*}
Identify $\sigma'$ as a subrepresentation of $\sigma'_1$. Let $(v,w)\in \mathrm{Ind}_{K_{P'}}^K  \sigma'|_{K_{M'}}\times \mathrm{Ind}_{K_{P}}^K \sigma^\lor|_{K_{M}}.$ For $\chi\in \Lambda_{M'}$
\begin{align*}
    Z(\mathrm{Ind}_P^G \sigma_{\lambda\chi},f_{S,C},v,w)= L(\sigma_{\lambda\chi})\langle T_{S,C}(\lambda\chi)(v),w\rangle.
\end{align*}
Let $\bar{w}$ be the image of $w$ in $\mathrm{Ind}_{K_{P'}}^G \sigma'^\lor|_{K_{M'}}.$ Then again by \hyperref[lem:cs:subquo]{Lemma \ref{lem:cs:subquo}}
\begin{align*}
    \Lambda_{M'}\ni\chi \mapsto Z(\mathrm{Ind}_{P'}^G\sigma'_\chi,f_{S,C},v,\bar{w})\in L(1/2,\sigma'_\chi,\rho_{M'})\langle T_{S,C}(\lambda\chi)(v),w\rangle\CC[\Lambda_{M'}]^\times.
\end{align*}
 Thus \hyperref[multigcd]{Theorem \ref{multigcd}} follows from \hyperref[lem:ideal]{Lemma \ref{lem:ideal}} and \hyperref[lem:varying:discrete]{Lemma \ref{lem:varying:discrete}} by varying $v,w, T_{S,C}, \mathfrak{H},$
\qed

\subsection{Tate-Godement-Jacquet theory}\label{ssec:Tate}


\begin{theorem}\label{thm:gcd}
    Suppose \eqref{LLC:gamma} holds for all $M\in \mathcal{M}$. Let $\pi$ be a smooth irreducible  representation of $G(F)$ with Langlands data $(P,\sigma,\lambda).$ Let $\mathcal{C}(\pi)\subseteq C^\infty(G(F))$ denote the space of matrix coefficients of $\pi$. 
    \begin{enumerate}
        \item Let $f\in \mathcal{S}_\rho(G(F))$ and $c\in\mathcal{C}(\pi).$ There is a positive cone $C$ in $\Re \Lambda_G$ such that for $\chi\in \Lambda_G,$ the integral
    \begin{align*}
        Z(\pi_\chi,f,c):= \int_{G(F)} f(g)c(g)\chi(g)dg
    \end{align*}
    converges absolutely if $\Re(\chi)\in C$. If $\pi\in \mathrm{Temp}(G),$ then one can take $C=C_{\rho,G}.$
        \item We have
        \begin{align*}
            I_\pi:=\mathrm{Span}_\CC \left\{\frac{Z(\pi_\chi,f,c)}{L(0,\sigma_\chi,\rho_M)}: (f,c)\in \mathcal{S}_\rho(G(F))\times\mathcal{C}(\pi)\right\}=\CC[\Lambda_G].
        \end{align*}
        \item For $(f,c)\in\mathcal{S}_\rho(G(F))\times\mathcal{C}(\pi)$, one has a functional equation
        \begin{align*}
            Z((\pi^\lor)_{|\nu|\chi^{-1}},|\nu|^{-1/2}\mathcal{F}_\rho (|\nu|^{1/2}f),c^\vee)=\gamma(0,\sigma_{\lambda\chi},\rho_M)Z(\pi_\chi,f,c).
        \end{align*}\qed
    \end{enumerate}
\end{theorem}

\begin{proof}
    Statements (1) and (3) follow from  \hyperref[Srho:in:Sas]{Corollary \ref{Srho:in:Sas}} and \hyperref[prop:generalFE]{Proposition \ref{prop:generalFE}}. Statement (2) essentially follows from the proof of \hyperref[multigcd]{Theorem \ref{multigcd}} as we shall explain. Retain the notations in \S \ref{ssec:multigcd:cuspidal}.

    For $\chi\in \Lambda_G,$ recall that $\pi_\chi$ is the unique irreducible quotient of $\mathrm{Ind}_P^G\sigma_{\lambda\chi}.$ Assume first $\sigma\in \Pi_0(M).$ For any nonempty subset $S\subseteq Y_{(M,\sigma)}$ and $C\in \kappa_S$ such that $\lambda\in C,$ we have $\pi^\lor\subseteq\bigcap_{i\in S} \mathrm{Ind}_{K_{P_i}}^K V_{i,\tau_{i,C}}^\lor.$ Thus in the proof of \hyperref[lem:varying]{Lemma \ref{lem:varying}}, we can choose $w_{S,C}\in \pi^\lor$ in the construction of $T_{S,C}'''.$  Then (2) follows from the fact that the ideal in $\CC[\Lambda_G]$ generated by the set
    \begin{align*}
    \left\{ \Lambda_G\ni \chi\mapsto \langle T_{S,C}(\lambda\chi)(v),w\rangle: v\in \Ind_{K_P}^K \sigma|_{K_M}, w\in \pi^\lor, T_{S,C}\in \mathfrak{H},  \mathfrak{H}\in \mathcal{H}\right\}
    \end{align*}
    is unital. The proof for $\sigma\in \Pi_2(M)-\Pi_0(M)$ is similar, which we leave to the reader.

    For general $\sigma\in \mathrm{Temp}(M),$ note that $L(0,\sigma'_{\lambda\chi},\rho_{M'})=L(0,\sigma_{\lambda\chi},\rho_{M})$ if $\sigma$ is a subrepresentation of $\Ind_{P'\cap M}^M\sigma'$ for some standard parabolic subgroup $P'\le P$ and $\sigma'\in \Pi_2(M').$ A similar argument as above proves (2). 
\end{proof}


\begin{remark}\label{remark:small}
    Up to a normalization shift, it is proved in \cite{shahidi2022resolution} that if $G$ is split then $b_\rho\in C^\infty_c(G(F))+\mathcal{F}_\rho(C^\infty_c(G(F)))$ presuming that $\mathbf{1}_{\mathcal{O}_F^n}\in C^\infty_c(\GG_m^n(F))+\mathcal{F}_n(C^\infty_c(\GG_m^n(F)))$ for any $n$. We give a simple proof in \hyperref[rem:faillocal]{Remark \ref{rem:faillocal}} below to show their assumption fails whenever $n\ge 2$. Therefore, $I_{M_0,1}=\CC[\Lambda_{M_0}]$ does not follow from $C^\infty_c(G)(F)\subseteq \mathcal{S}_\rho(G(F))$. More generally, by the proof of \hyperref[lem:naivegen]{Lemma \ref{lem:naivegen}}, the ideal
    \begin{align*}
        \mathrm{Span}_\CC\left\{\frac{Z(\mathrm{Ind}_P^G \sigma_\chi,f,v,w)}{L(1/2,\sigma_\chi,\rho_M)}: f\in C^\infty_c(G(F))+\mathcal{F}_\rho(C^\infty_c(G(F))),(v,w)\in \Ind_{K_P}^K\sigma\vert_{K_M}\times \Ind_{K_P}^K\sigma^\vee\vert_{K_M}\right\}
    \end{align*}
    in $\CC[\Lambda_M]$ is always proper if the split rank of $M$ is at least $2$. Therefore, by \hyperref[multigcd]{Theorem \ref{multigcd}} $C^\infty_c(G(F))+\mathcal{F}_\rho(C^\infty_c(G(F)))\subsetneq |\nu|^{1/2}\mathcal{S}_\rho(G(F))$ except possibly when $G$ has split rank $1$.
\end{remark}

\section{On affine toric varieties}\label{sec:toric:Schwartz}

In this section, we discuss the underlying geometry of the function space $\mathcal{S}_\rho(G(F))$ when $G=T$ is a torus.

\subsection{Affine normal toric varieties}\label{subsec:affinetoric}
Let $T$ be a torus over $F$. Assume first $T$ is split. We review some standard facts on affine normal toric varieties with torus $T$ in \cite{Fulton:toric}. While the base field in ibid. is $\CC$, but the same argument applies over general fields.

Let $\sigma \subseteq X_*(T)$ be a submonoid such that $\RR_{\ge 0}\sigma$ is a strongly convex rational polyhedral cone in $X_*(T)_{\mathbb{R}}$. We shall write $\sigma^\vee := (\mathbb{R}_{\geq0}\sigma)^\vee\subseteq X^*(T)_\RR$ for brevity. The same convention applies throughout this section. Then $\overline{T}:=\Spec F[\sigma^\lor\cap X^\ast(T)]$ is an affine normal $T$-toric variety.  By \cite[\S2.3]{Fulton:toric} a character $\chi:T\to\mathbb{G}_m$ extends to a morphism $\overline{T}\to\mathbb{A}^1$ if and only if $\chi\in \sigma^\vee\cap X^*(T),$ and a cocharacter $\lambda:\mathbb{G}_m\to T$ extends to a morphism $\mathbb{A}^1\to \overline{T}$ if and only if $\lambda\in \sigma$. Furthermore, by \cite[\S3.1]{Fulton:toric} there is an order-reversing bijection between faces $\tau$ of $\RR_{\ge 0}\sigma$ and $T$-orbits $\mathcal{O}(\tau)$ of $\overline{T}:$
\begin{align*}
    \tau_1\subseteq \tau_2 \quad\quad \xLeftrightarrow{\qquad}\quad\quad \mathcal{O}(\tau_2)\subseteq \overline{\mathcal{O}}(\tau_1). 
\end{align*}
Explicitly, for a face $\tau,$ $\mathcal{O}(\tau)$ is the minimal $T$-orbit in the affine toric variety attached to $\tau$. The map $\mathbf{1}_{\tau^\perp\cap X^*(T)}:\sigma^\lor\cap X^\ast(T)\longrightarrow \ZZ$ defines a distinguished point $\gamma_\tau\in\mathcal{O}(\tau)(F)$ such that for any $\lambda \in X_*(T)$ in the relative interior of $\tau$, one has $\lim\limits_{\substack{|z|\to 0}}\lambda(z) = \gamma_\tau$ and $\mathcal{O}(\tau)(F)=T(F)\gamma_\tau.$ As a consequence, we have

\begin{lemma}\label{toricdense} Suppose $T$ is split. Then $T(F)$ is dense in $\overline{T}(F)$.\qed
\end{lemma}

Now let $T$ be general. Let $\sigma \subseteq X_*(T_{F^{\mathrm{sep}}})$ be a submonoid stable under $\Gal_F$-action such that $\RR_{\ge 0}\sigma$ is a strongly convex rational polyhedral cone in $X_*(T_{F^{\mathrm{sep}}})_{\mathbb{R}}$. By Galois descent, 
\begin{align*}
    \overline{T}:=\Spec F^{\mathrm{sep}}[\sigma^\vee\cap X^*(T_{F^{\mathrm{sep}}})]^{\Gal_F}
\end{align*}
is an integral affine normal scheme of finite type over $F$ containing $T$ as an open subscheme. It has a unique monoid structure extending the group structure on $T$. Note that if $T_E$ is split for some finite Galois extension $E/F,$ then $X_\ast(T_{F^{\mathrm{sep}}})=X_\ast(T_E)$ and the above construction can be carried out over $E$. Suppose $T$ is unramified, i.e.,  $T_E$ is split for some finite unramified extension $E/F$. Then $\mathcal{T}:=\Spec\mathcal{O}_{E}[X^*(T_{E})]^{\Gal(E/F)}$ is a smooth model of $T$. The same proof as in \cite[\S 2.1]{Fulton:toric} shows that $\Spec\mathcal{O}_{E}[\sigma^\vee\cap X^*(T_{E})]$ is normal. Thus by faithfully flat descent
\begin{align*}
    \overline{\mathcal{T}}:=\Spec\mathcal{O}_{E}[\sigma^\vee\cap X^*(T_{E})]^{\Gal(E/F)}
\end{align*}
defines an affine normal integral model of $\overline{T}$. We let $\overline{T}(\mathcal{O}_F):=\overline{\mathcal{T}}(\mathcal{O}_F).$

Fix a finite Galois extension $E/F$ such that $T_E$ is split, so $\sigma\subseteq X_\ast(T_E).$  Since $X^*(T_{E})/X^*(T_{E})^{\Gal(E/F)}$ is torsion-free, the canonical map 
\begin{align*}
    X_*(T_{E}) = \Hom_\ZZ(X^*(T_{E}),\ZZ) \twoheadrightarrow \Hom_\ZZ(X^*(T_{E})^{\Gal(E/F)},\ZZ) = \Hom_\ZZ(X^*(T),\ZZ)
\end{align*}
is surjective. We let $\widetilde{\sigma}$ denote the image of $\sigma$ under this map. Let $\mathcal{B}_0\subseteq X^*(T)$ be a finite generating set of the dual cone $\widetilde{\sigma}^\vee\subseteq X^*(T)_\RR$. 

\begin{lemma}\label{toric:relcpt:1}
For any $c>0$, the set $\{x\in T(F): |\chi(x)|\leq c\text{ for all }\chi\in \mathcal{B}_0\}$ is relatively compact in $\overline{T}(F)$.
\end{lemma}
\begin{proof} Let $\mathrm{av}:X^*(T_E)_\RR\to X^*(T)_\RR$ denote the averaging map
\begin{align*}
    \mathrm{av}(\chi) := \frac{1}{[E:F]}\sum_{\gamma\in\Gal(E/F)}\gamma.\chi.
\end{align*}
Let $\chi\in \sigma^\vee\cap X^*(T_E)$ and $x\in T(F)$. Then
\begin{align*}
    |\chi|(x)= \left|\prod_{\gamma\in \Gal(E/F)}\gamma.\chi(x)\right|^{\frac{1}{[E:F]}} = |\mathrm{av}(\chi)|(x).
\end{align*}
Let $S_c := \{x\in T(F): |\chi(x)|\leq c\text{ for all }\chi\in \mathcal{B}_0\}$. As $\mathrm{av}(\sigma^\vee)\subseteq\widetilde{\sigma}^\vee$ and $\mathcal{B}_0$ generates $\widetilde{\sigma}^\vee$, we have $|\chi|(S_c)$ is bounded in $\mathbb{R}$ for all $\chi\in\sigma^\vee\cap X^*(T_E)$. Since $\overline{T}(E)$ has the initial topology induced by $\sigma^\vee\cap X^*(T_E)$ and $\overline{T}(F)$ is closed in $\overline{T}(E)$, we conclude that $S_c$ is relatively compact in $\overline{T}(F)$.
\end{proof}

\begin{corollary}\label{toric:Srho:cptsupp} Every function in $\mathcal{S}_\rho(T(F))$ has compact support in $\overline{T}(F)$.
\end{corollary}
\begin{proof} 
Combine \hyperref[Sas:f1alpha:polynomial]{Corollary \ref{Sas:f1alpha:polynomial}} and \hyperref[toric:relcpt:1]{Lemma \ref{toric:relcpt:1}}.
\end{proof}

\begin{lemma}\label{toric:relcpt:2} If $U\subseteq \overline{T}(F)$ is relatively compact, then 
\begin{align*}
    U\cap T(F)\subseteq \{x\in T(F): |\chi(x)|\leq c\text{ for all }\chi\in \mathcal{B}_0\}
\end{align*}
for some $c>0$.
\end{lemma}
\begin{proof} Keep the notation in the proof of \hyperref[toric:relcpt:1]{Lemma \ref{toric:relcpt:1}}. Since $\overline{T}(F)$ is closed in $\overline{T}(E)$, $U\cap T(F)$ is relatively compact in $\overline{T}(E)$. Hence $|\chi|(U) = |\mathrm{av}(\chi)|(U)$ is bounded in $\mathbb{R}$ for all $\chi\in\sigma^\vee\cap X^*(T_E)$. As $\widetilde{\sigma}^\vee\subseteq\mathrm{av}(\sigma^\vee)$ under the natural inclusion $X^*(T)_\RR\subseteq X^*(T_E)_\RR$, the lemma follows.
\end{proof}

\subsection{Comparison to Luo-Ng\^o's construction} \label{subsec:LN}
Let $\rho:{}^LT\to \GL_{V_\rho}(\mathbb{C})$ be a tempered representation such that $(T,\rho)$ satisfies \eqref{G}, so there exists a distinguished character $\nu = \nu_T\in X^*(T)$. Decompose $\rho\vert_{T^\vee} = \rho_1\oplus\cdots\oplus \rho_n$ into characters. Then $[n]$ admits a $\Gal_F$-action and
\begin{align*}
    \sigma_{\rho,T} = \sum_{i=1}^n \mathbb{Z}_{\geq 0}\rho_i \subseteq X^*(T^\vee)
\end{align*}
defines a submonoid invariant under $\Gal_F$-action. By \eqref{G} $\mathbb{R}_{\geq 0}\sigma_{\rho,T}$ is a strongly convex rational cone in $X^*(T^\vee)_{\mathbb{R}} = X_*(T_{F^{\mathrm{sep}}})_{\mathbb{R}}$, stable under $\Gal_F$-action. Following the recipe in \S\ref{subsec:affinetoric}, we obtain an affine normal $T$-toric variety $\overline{T}$ over $F$.

Under the identification $X^*(T^\vee) = X_*(T_{F^{\mathrm{sep}}})$, the characters $\rho_1,\ldots,\rho_n$ define a morphism $\mathbb{G}_{mF^{\mathrm{sep}}}^n\to T_{F^{\mathrm{sep}}}$ which extends to a morphism of monoids $\mathbb{A}_{F^{\mathrm{sep}}}^n\to \overline{T}_{F^{\mathrm{sep}}}$. By Galois descent, we have a commutative diagram
\begin{center}
    \begin{tikzcd}
        D^\rho\arrow[r]\arrow[d,"\rho_T"]&\mathbb{A}^\rho\arrow[d,"\rho_{\overline{T}}"]\\
        T\arrow[r]&\overline{T}.
    \end{tikzcd}
\end{center}
Let $S_1,\ldots,S_k\subseteq [n]$ denote the Galois orbits. For each $i\in[k],$ let $E_i/F$ denote the finite separable extension corresponding to the orbit $S_i$. Then
\begin{align*}
    D^\rho \cong \prod_{i=1}^k \mathrm{Res}_{E_i/F}\mathbb{G}_{mE_i},\qquad \mathbb{A}^\rho \cong \prod_{i=1}^k \mathrm{Res}_{E_i/F}\mathbb{A}^1_{E_i}.
\end{align*}
Let 
\begin{align}\label{eq:pii}
\pi_i:\mathrm{Res}_{E_i/F}\mathbb{G}_{mE_i}\longrightarrow D^\rho\xrightarrow{\,\,\rho_T\,\,}T.
\end{align}
\begin{lemma}\label{etaleTate:local:factr:equal} 
We have identities of functions in $\chi\in \Hom_{\mathbf{TopGp}}(T(F),\mathbb{C}^\times)$
\begin{align*}
    L(0,\chi,\rho) &= \prod_{i=1}^k L(0,\chi\circ \pi_i),\\
    \gamma(0,\chi,\rho,\psi) &= \prod_{i=1}^k \gamma(0,\chi\circ \pi_i,\psi),
\end{align*}
where the local factors on the right are the $\GL_1$ local factors of Tate.
\end{lemma}
\begin{proof} 
By definition $L(0,\chi,\rho) = \prod_{i=1}^k\prod_{j\in S_i} L(0,\chi,\rho_j),$
so we only need to show for each $i$ that
\begin{align*}
    L(s,\chi\circ \pi_i) = \prod_{j\in S_i} L(s,\chi,\rho_j).
\end{align*}
Thus we assume $k=1$ and drop the index $i$.

Recall notations in \S \ref{ssec:assumptions}. By the construction of $\mathrm{LL}_T$ in \cite{Langlands:tori}, one has a commutative diagram
\begin{center}
    \begin{tikzcd}
            \Temp(\GG_{mE})_\CC\arrow[d,equal]\arrow[rr,"\mathrm{LL}_{\GG_{mE}}"]&&\Phi(\GG_{mE})\arrow[d,"\sim"]\\
        \Temp(\mathrm{Res}_{E/F} \GG_{mE})_\CC\arrow[rr,"\mathrm{LL}_{\mathrm{Res}_{E/F}\GG_{mE}}"]&&\Phi(\mathrm{Res}_{E/F}\GG_{mE}).
    \end{tikzcd}
\end{center}
Here the right vertical map is given by Shapiro's Lemma. Therefore, viewing $\mathrm{LL}_{\mathrm{Res}_{E/F}\mathbb{G}_{mE}}(\chi\circ \pi): W_F\to (\mathrm{Res}_{E/F}\mathbb{G}_{mE})^\vee(\CC)\cong (\CC^\times)^{[E:F]}$ as a representation of $W_F$, we have
\begin{align*}
    \mathrm{LL}_{\mathrm{Res}_{E/F}\mathbb{G}_{mE}}(\chi\circ\pi) = \Ind_{W_{E}}^{W_F}\mathrm{LL}_{\mathbb{G}_{mE}}(\chi\circ\pi).
\end{align*}
Note that $\rho$ factors through the dual morphism $\pi^\lor=\bigoplus \rho_j:{}^LT\longrightarrow {}^L\mathrm{Res}_{E/F}\GG_{mE}$ induced by $\pi.$ Thus $\pi^\vee\circ\mathrm{LL}_{T}(\chi) = \mathrm{LL}_{\mathrm{Res}_{E/F}\mathbb{G}_{mE}}(\chi\circ\pi)$ and by the inductivity of Artin $L$-functions
\begin{align*}
     \prod_{j\in S} L(0,\chi,\rho_j)&=L(0,\pi^\vee\circ\mathrm{LL}_{T}(\chi))=L(0, \mathrm{LL}_{\mathrm{Res}_{E/F}\mathbb{G}_{mE}}(\chi\circ\pi))\\
     &=L(0,\Ind_{W_{E}}^{W_F}\mathrm{LL}_{\mathbb{G}_{mE}}(\chi\circ\pi)) = L(0,\mathrm{LL}_{\mathbb{G}_{mE}}(\chi\circ\pi)) = L(0,\chi\circ\pi).
\end{align*}
This proves the equality for $L$-factors. The identity for $\gamma$-factors can be proved similarly.
\end{proof}

Let $\mathrm{std}_{D^\rho}:{}^LD^\rho\to \GL_{n}(\CC)$ denote the standard representation. Let $\mathrm{N}_{\mathbb{A}^\rho/F}$ be the norm map. Note that $(D^\rho,\mathrm{std}_{D^\rho},\mathrm{N}_{\mathbb{A}^\rho/F})$ satisfies the assumption \eqref{G}.
 
\begin{lemma}\label{lem:Tate:etale} One has $\mathcal{S}(\mathbb{A}^\rho(F)) = \mathcal{S}_{\mathrm{std}_{D^\rho}}(D^\rho(F))$.
\end{lemma}
\begin{proof} By \hyperref[etaleTate:local:factr:equal]{Lemma \ref{etaleTate:local:factr:equal}}, one has $L(0,\chi,\rho) = L(0,\chi\circ\rho_T,\mathrm{std}_{D^\rho})$ for $\chi\in\Hom_{\textbf{TopGp}}(T(F),\CC^\times)$. The lemma follows from Tate thesis (c.f. \cite[\S1.5]{Igusa}). Strictly speaking, both references only deal with the case when $\mathbb{A}^\rho$ is a field, but the argument easily generalizes to the case of an et\'ale algebra.
\end{proof}

 Consider the extension
\begin{center}
    \begin{tikzcd}
    D^\rho\times T\arrow[r]\arrow[d,"\rho_T'"]&\mathbb{A}^\rho\times T\arrow[d,"\rho_{\overline{T}}'"]\\
        T\arrow[r]&\overline{T}
    \end{tikzcd}
\end{center}
where $\rho_T'(t,a):=\rho_T(t)a$ and similarly for $\rho'_{\overline{T}}.$ Let $U:=\ker\rho_T' \subseteq D^\rho\times T$. Let $du$ denote the induced Haar measure on $U(F),$ and define 
\begin{center}
    \begin{tikzcd}
        (\rho_{T}')_!:\mathcal{S}(\mathbb{A}^\rho(F)\times T(F))\arrow[r]& C^\infty(T(F))
    \end{tikzcd}
\end{center}
by the absolutely convergent integral
\begin{align*}
    (\rho_{T}')_!f(t) := \int_{U(F)} f(t'u) du
\end{align*}
for $t\in T(F)$ where $t'\in \rho_T'^{-1}(t)$. Note that we have
\begin{align*}
     \int_{T(F)}(\rho_{T}')_!f(t)dt=\int_{(D^\rho\times T)(F)} f(t)dt,
\end{align*}
whenever the integral on the right converges absolutely. Set $\mathcal{S}^{\mathrm{NL}}_\rho(T(F)) := (\rho_T')_!\mathcal{S}(\mathbb{A}^\rho(F)\times T(F)).$

\begin{lemma}\label{SNL:in:Srho} With the setting as above, one has $
\mathcal{S}^{\mathrm{NL}}_\rho(T(F))\subseteq\mathcal{S}_\rho(T(F)).$
\end{lemma}
\begin{proof} 
Let $\phi\otimes \phi'\in \mathcal{S}(\mathbb{A}^\rho(F))\otimes C^\infty_c(T(F))$ and $\chi\in\Hom_{\mathbf{TopGp}}(T(F),\mathbb{C}^\times).$ One formally has
\begin{align}\label{etaleTate:zeta:decomp}
    \int_{T(F)} (\rho_T')_!(\phi \otimes \phi')(t) \chi(t) dt = \int_{D^\rho(F)} \phi(t)(\chi\circ\rho_T)(t)dt \times \int_{T(F)} \phi'(t)\chi(t) dt.
\end{align}
The latter integral on the right converges absolutely for any $\phi'$ and $\chi$. We claim there exists $r>0$ such that the integral over $D^\rho(F)$ is absolutely convergent for $\chi\in |\nu_T|^rC_{\rho,T}$.

Assume $\phi=\otimes\phi_i\in\mathcal{S}(\mathbb{A}^\rho(F)) = \bigotimes\mathcal{S}(E_i)$. Then
\begin{align}\label{etaleTate:puretensor}
    \int_{D^\rho(F)} \phi(t)(\chi\circ\rho_T)(t) dt = \prod_{i=1}^k \int_{E_i^\times} \phi_i(t)(\chi\circ \pi_i)(t) d^\times t,
\end{align}
where $\pi_i$ is defined in \eqref{eq:pii}. Observe that for $t\in E_i^\times,$ 
\begin{align*}
    |\chi\circ \pi_i (t)| = |t|_{E_i}^{\sum_{j\in S_i}\langle\chi,\rho_j\rangle}.
\end{align*}
For $\chi = |\nu_T|^r\chi'$ with $\chi'\in C_{\rho,T}$, one has
\begin{align*}
    \int_{E_i^\times} |\phi_i(t)||(\chi\circ \pi_i)(t)| d^\times t = \int_{E_i^\times} \phi_i(t) |t|_{E_i}^{\sum_{j\in S_i}\langle r\nu_T + \chi',\rho_j\rangle} d^\times t = \int_{E_i^\times} \phi_i(t) |t|_{E_i}^{r + \sum_{j\in S_i}\langle\chi',\rho_j\rangle} d^\times t.
\end{align*}
Since $\langle\chi',\rho_j\rangle\geq 0$, this integral is convergent when $r>0$ by Tate's thesis. Hence $\mathcal{S}^{\mathrm{NL}}_\rho(T(F))\subseteq \mathcal{S}^\infty(T(F))$.

To finish the proof, by \eqref{etaleTate:puretensor} and Tate's thesis, the expression
\begin{align*}
    \chi\mapsto \dfrac{\displaystyle\int_{D^\rho(F)} \phi(t)(\chi\circ\rho_T)(t) dt}{\prod\limits_{i=1}^k L(0,\chi\circ \pi_i)}\times \int_{T(F)} \phi'(t)\chi(t) dt
\end{align*}
extends to polynomial in $\chi\in\Lambda_T$. The proof is concluded by \hyperref[etaleTate:local:factr:equal]{Lemma \ref{etaleTate:local:factr:equal}}. 
\end{proof}

\begin{proposition}\label{NL=Srho} One has $\mathcal{S}^{\mathrm{NL}}_\rho(T(F))=\mathcal{S}_\rho(T(F)).$
\end{proposition}
\begin{proof} Let $f\in \mathcal{S}_\rho(T(F))$. By definition, we can find $h\in C_c^\infty(T(F))$ such that
\begin{align*}
    Z(f,\chi) = L(0,\chi,\rho)Z(h,\chi) 
\end{align*}
for all $\chi\in\Hom_{\textbf{TopGp}}(T(F),\CC^\times)$. Choose $h'\in C_c^\infty(D^\rho(F) \times T(F))$ such that $(\rho_T')_!h' = h$. By \hyperref[lem:Tate:etale]{Lemma \ref{lem:Tate:etale}}, there exists $f'\in \mathcal{S}(\mathbb{A}^\rho(F)\times T(F))$ such that
\begin{align*}
    Z(f',\chi') = L(0,\chi'\vert_{D^\rho(F)},\mathrm{std}_{D^\rho})Z(h',\chi')
\end{align*}
for all $\chi'\in \Hom_{\textbf{TopGp}}(D^\rho(F)\times T(F),\CC^\times)$. A repeated use of \eqref{etaleTate:zeta:decomp} shows for all $\chi\in\Hom_{\textbf{TopGp}}(T(F),\CC^\times)$ that
\begin{align*}
    Z((\rho_T')_!f',\chi) = Z(f',\chi\circ \rho_T') &=  L(0,(\chi\circ \rho_T')\vert_{D^\rho(F)},\mathrm{std}_{D^\rho})Z(h',\chi\circ \rho_T')\\
    &= L(0,\chi\circ \rho_T,\mathrm{std}_{D^\rho})Z((\rho_T')_!h',\chi)= L(0,\chi,\rho)Z(h,\chi) = Z(f,\chi).
\end{align*}
Hence $(\rho_T')_!f' = f$. This finishes the proof.
\end{proof}

View $C^\infty(\overline{T}(F))$ as a subspace of $C^\infty(T(F))$ by restriction.
\begin{theorem}\label{thm:localT}
$\mathcal{S}_\rho(T(F))$ is a $C^\infty(\overline{T}(F))$-submodule under the usual function multiplication.
\end{theorem}
\begin{proof} Let $h\in C^\infty(\overline{T}(F))$ and  $f\in \mathcal{S}_\rho(T(F))$. By \hyperref[NL=Srho]{Proposition \ref{NL=Srho}} there exists $f'\in \mathcal{S}(\mathbb{A}^\rho(F)\times T(F))$ with $(\rho_T')_!f' = f$. Since $h\circ \rho_{\overline{T}}'\in C^\infty(\mathbb{A}^\rho(F)\times T(F))$, we have $f'\cdot (h\circ\rho_{\overline{T}}')\in\mathcal{S}(\mathbb{A}^\rho(F)\times T(F))$. Therefore,
\begin{align*}
    fh = (\rho_T')_!(f'\cdot (h\circ\rho_{\overline{T}}')) \in \mathcal{S}_\rho(T(F)).
\end{align*}
This finishes the proof.
\end{proof}


Fix a nontrivial additive character $\psi:F\to \mathbb{C}^\times$, and let $\psi_{\mathbb{A}^\rho} = \prod\limits_{i=1}^k\psi_{i}$, where $\psi_i:=\psi\circ\mathrm{Tr}_{E_i/F}$. Let $dx_i$ be the self-dual measure on $E_i$ with respect to $\psi_i$. Then the product measure $dx:=\bigotimes dx_i$ is the self-dual Haar measure on $\mathbb{A}^{\rho}(F)$ with respect to $\psi_i$. Let $\mathcal{F}_{E_i}$ and $\mathcal{F}_{\mathbb{A}^\rho}$ be the Fourier transform on $E_i$ and $\mathbb{A}^\rho(F)$ defined by $\psi_i$ and $\psi_{\mathbb{A}^\rho}$, respectively. Following \cite{Luo:Ngo}, we have a Fourier transform on $\mathcal{S}_\rho^{\mathrm{NL}}(T(F))$ by descent
\begin{equation*}
    \begin{tikzcd}
        \mathcal{S}(\mathbb{A}^\rho(F)\times T(F))\arrow[rrr,"\mathcal{F}_{\mathbb{A}^\rho}\otimes  |\nu_T|^{-1}(\cdot)^\lor"]\arrow[d]& & & \mathcal{S}(\mathbb{A}^\rho(F)\times T(F))\arrow[d]\\
       \mathcal{S}_\rho^{\mathrm{NL}}(T(F))\arrow[rrr,"\mathcal{F}_\rho^{\mathrm{NL}}"]& & &\mathcal{S}_\rho^{\mathrm{NL}}(T(F)).
    \end{tikzcd}
\end{equation*}
Here recall that for a function $f$, we set $f^\vee(x):= f(x^{-1})$. By \hyperref[Sas:FE:near0]{Corollary \ref{Sas:FE:near0}}, \hyperref[etaleTate:local:factr:equal]{Lemma \ref{etaleTate:local:factr:equal}}, \hyperref[NL=Srho]{Proposition \ref{NL=Srho}}  and the functional equations for $\mathcal{F}_\rho^{\mathrm{NL}}$ proved in \cite{Luo:Ngo}, we see that
\begin{align*}
    \mathcal{F}_\rho^{\mathrm{NL}} = |\nu_T|^{-\frac{1}{2}}\mathcal{F}_{\rho}|\nu_T|^{\frac{1}{2}}.
\end{align*}

\begin{lemma}\label{toric:basic} Suppose $T$ is unramified. One has
\begin{align*}
    (\rho'_T)_!(\mathbf{1}_{\mathbb{A}^\rho(\mathcal{O}_F)}\otimes\mathbf{1}_{T(\mathcal{O}_F)})= b_\rho|\nu_T|^{-1/2}.
\end{align*}
In particular, $b_\rho(\mathrm{id})=1.$
\end{lemma}
\begin{proof} 
Since $\mathbf{1}_{\mathbb{A}^\rho(\mathcal{O}_F)}\otimes\mathbf{1}_{T(\mathcal{O}_F)}$ is $D^\rho(\mathcal{O}_F)\times T(\mathcal{O}_F)$-invariant, $(\rho'_T)_!(\mathbf{1}_{\mathbb{A}^\rho(\mathcal{O}_F)}\otimes\mathbf{1}_{T(\mathcal{O}_F)})$ is $T(\mathcal{O}_F)$-invariant. Then
\begin{align*}
    \int_{T(F)} (\rho'_T)_!(\mathbf{1}_{\mathbb{A}^\rho(\mathcal{O}_F)}\otimes\mathbf{1}_{T(\mathcal{O}_F)})(t)\chi(t)dt
\end{align*}
vanishes unless $\chi\in\Lambda_T$, in which case by \eqref{etaleTate:zeta:decomp} it is equal to
\begin{align*}
    \left(\prod_{i=1}^k \int_{E^\times_i}\mathbf{1}_{\mathcal{O}_{E_i}}(t_i)(\chi\circ\pi_i)(t_i) d^\times t_i\right) \times \int_{T(F)} \mathbf{1}_{T(\mathcal{O}_F)}(t) \chi(t) dt.
\end{align*}
By \hyperref[etaleTate:local:factr:equal]{Lemma \ref{etaleTate:local:factr:equal}} this equals
\begin{align*}
    L(0,\chi,\rho)\vol(T(\mathcal{O}_F),dt) = L(0,\chi,\rho).
\end{align*}
The first assertion follows from the definition of the basic function. Since $\rho_T^{-1}(T(\mathcal{O}_F))=D^\rho(\mathcal{O}_F),$ we have
\begin{align*}
     b_\rho(\mathrm{id})&=(\rho'_T)_!(\mathbf{1}_{\mathbb{A}^\rho(\mathcal{O}_F)}\otimes\mathbf{1}_{T(\mathcal{O}_F)})(\mathrm{id})=\mathrm{vol}(D^\rho(\calO_F)\times T(\mathcal{O}_F),dt)=1.
\end{align*}
\end{proof}

\begin{remark}\label{rem:faillocal}
    In general, $C^\infty_c(T(F))+\mathcal{F}_\rho(C^\infty_c(T(F)))$
    is a proper subspace of $\mathcal{S}_\rho(T(F))$ if the split rank of $T$ is greater than $1$. When
    $T$ is unramified, the space does not necessarily contain $b_\rho$. This is already the case for the standard representation of $\GG_m^n$ for $n\ge 2$. We claim 
    \begin{center}
        $\mathbf{1}_{\mathcal{O}_F^n}\notin C^\infty_c(\GG_m^n(F))+\mathcal{F}_n(C^\infty_c(\GG_m^n(F)))$ for $n\ge 2,$
    \end{center}
    where $\mathcal{F}_n$ is the standard Fourier transform on $\mathbb{A}^n(F)=F^n.$ We prove the claim for $n=2.$ The argument can be easily generalized.

    Consider the space $$\mathcal{S}_c:=\{f\in \mathcal{S}(F): f(0)=\mathcal{F}_1(f)(0)=0\}=\{f\in C^\infty_c(F^\times):\mathcal{F}_1(f)(0)=0\}.$$
    One can easily show that 
    \begin{align*}
        &C_c^\infty(\GG_m^2(F))+\mathcal{F}_{2}(\mathcal{S}_c\otimes C^\infty_c(F^\times))=C^\infty_c(\GG_m^2(F))+\mathcal{S}_c\otimes \mathcal{S}(F),\\
        &C_c^\infty(\GG_m^2(F))+\mathcal{F}_{2}(C^\infty_c(F^\times)\otimes S_c)=C_c^\infty(\GG_m^2(F))+\mathcal{S}(F)\otimes \mathcal{S}_c.
    \end{align*}
    Since 
    \begin{align*}
        C^\infty_c(\GG_m^2(F))=\mathcal{S}_c\otimes C^\infty_c(F^\times)+C^\infty_c(F^\times)\otimes \mathcal{S}_c+\CC\mathbf1{_{\mathcal{O}_F^\times\times \mathcal{O}_F^\times}},
    \end{align*}
 we have
    \begin{align*}
       &C^\infty_c(\GG_m^2(F))+\mathcal{F}_2(C^\infty_c(\GG_m^2(F)))=C^\infty_c(\GG_m^2(F))+\mathcal{S}_c\otimes \mathcal{S}(F)+ \mathcal{S}(F)\otimes \mathcal{S}_c+\CC\mathcal{F}_2(\mathbf{1}_{\mathcal{O}_F^\times\times \mathcal{O}_F^\times}).
    \end{align*}
Write $\mathcal{F}_2=\mathcal{F}_\ell\otimes\mathcal{F}_r$ as the tensor of partial Fourier transforms in the first and the second entry. Then we have an isomorphism of vector spaces
\begin{align*}
    \varphi:\mathcal{S}(F^2)/\bigg(C^\infty_c(\GG_m^2(F))+\mathcal{S}_c\otimes \mathcal{S}(F)+ \mathcal{S}(F)\otimes \mathcal{S}_c\bigg)&\tilde{\longrightarrow}\, \CC^3\\
    f&\mapsto (f(0,0),\mathcal{F}_\ell(f)(0,0),\mathcal{F}_r(f)(0,0)),
\end{align*}
and $\varphi(\mathbf{1}_{\mathcal{O}_F^2})=(1,1,1).$ On the other hand, $\mathcal{F}_1(\mathbf{1}_{\mathcal{O}_F^\times})=\mathbf{1}_{\mathcal{O}_F}-q^{-1}\mathbf{1}_{\varpi^{-1}\mathcal{O}_F}.$ Thus,
\begin{align*}
    \varphi(\mathcal{F}_2(\mathbf{1}_{\mathcal{O}_F^\times\times \mathcal{O}_F^\times}))=\big((1-q^{-1})^2,1-q^{-1},1-q^{-1}\big)
\end{align*}
is not parallel to $\varphi(\mathbf{1}_{\mathcal{O}_F^2}).$ This justifies the claim. 
\end{remark}

Suppose $T$ is split. Let $\mathcal{S}_{\mathrm{ES}}(\overline{T}(F)):=C_c^\infty(\overline{T}(F))$ which we view as a subspace of $C^\infty(T(F))$ by restriction (c.f. \hyperref[toricdense]{Lemma \ref{toricdense}}). 
By reindexing we assume $\{\rho_1,\ldots,\rho_k\}$ is a minimal generating set of $\mathbb{R}_{\geq 0}\sigma_{\rho,T}$. Then $\mathbb{R}_{\geq 0}\rho_1,\ldots,\mathbb{R}_{\geq 0}\rho_k$ exhaust the rays of $\mathbb{R}_{\geq 0}\sigma_{\rho,T}$.

\begin{proposition}\label{prop:ESinrho} Suppose $T$ is split. One has $\mathcal{S}_{\mathrm{ES}}(\overline{T}(F))\subseteq \mathcal{S}_\rho(T(F))$.
\end{proposition}
\begin{proof} Let $f\in \mathcal{S}_{\mathrm{ES}}(\overline{T}(F))$ and let $\sigma$ be a unitary character of $T(F)$. By \hyperref[SESinSinfty]{Lemma \ref{SESinSinfty}} below, we have $f\in \mathcal{S}^\infty(T(F))$ and $|\nu_T|^r\mathcal{S}_{\mathrm{ES}}(\overline{T}(F))\subseteq \mathcal{C}(T(F))$ for any $r>0$ sufficiently large. Thus it remains to show for $\chi\in \Lambda_T$ that
\begin{align}\label{eq:toric:original}
     \prod_{i=1}^k L(0,\sigma_\chi,\rho_i)^{-1} \times \int_{T(F)} f(t)\sigma_\chi(t)dt\in \CC[\Lambda_T].
\end{align}
By reordering, we may assume $\sigma\circ \rho_i:F^\times\to S^1$ is unramified for $1\le i\le \ell$, and $\sigma\circ\rho_i$ is ramified for $\ell<i\le k$. For $1\le i\le\ell$, define an operator $L_{\rho_i}:\mathcal{S}_{\mathrm{ES}}(\overline{T}(F))\to \mathcal{S}_{\mathrm{ES}}(\overline{T}(F))$ by $L_{\rho_i}(h)(t) := h(\rho_i(\varpi)^{-1}t).$
Then \eqref{eq:toric:original} is
\begin{align*}
    \prod_{i=1}^\ell (1-\sigma_\chi(\rho_i(\varpi))\times \int_{T(F)} f(t)\sigma_\chi(t)dt = \int_{T(F)} \left(\prod_{i=1}^\ell(\mathrm{id}-L_{\rho_i})\right)f(t) \sigma_\chi(t) dt.
\end{align*}

For $1\le i\le \ell,$ let 
\begin{align*}
h_i := \left(\prod_{\substack{1\leq j\leq \ell\\j\neq i}}(\mathrm{id}-L_{\rho_j})\right)f, \quad  h:=\left(\prod_{\substack{1\leq j\leq \ell}}(\mathrm{id}-L_{\rho_j})\right)f.
\end{align*}
Since $\overline{T}(F)$ has a metric topology induced by the closed embedding $\overline{T}(F)\to F^{\#\mathcal{B}_0}$ given by $x\mapsto (\tau(x))_{\tau\in\mathcal{B}_0}$, functions in $\mathcal{S}_{\mathrm{ES}}(\overline{T}(F))$ are uniformly locally constant. Namely, there exists $N>0$ such that  $h_i(x) = h_i(y)$ and $h(x) = h(y)$ whenever $x,y\in \overline{T}(F)$ satisfy $|\tau(x) - \tau(y)|< q^{-N}$ for all $\tau\in \mathcal{B}_0$.

Write
\begin{align*}
    \int_{T(F)} h(t)\sigma_\chi(t)dt = \sum_{\lambda\in X_*(T)} \sigma_\chi(\lambda(\varpi)) \int_{T(\mathcal{O}_F)} h(\lambda(\varpi)u) \sigma(u)du.
\end{align*}
Fix $1\le i\le k$, and let $S_i := \{\tau\in\mathcal{B}_0: \langle\tau,\rho_i\rangle>0\}$. For $\tau\in \mathcal{B}_0$ and $v\in \mathcal{O}_F^\times$, one has
\begin{align*}
    |\tau(\lambda(\varpi)u) - \tau(\lambda(\varpi) u \rho_i(v))|= q^{-\langle \tau,\lambda\rangle} |1-v^{\langle\tau,\rho_i\rangle}| \le \left\{\begin{array}{ll}
        0 &\textrm{if }\tau\in\mathcal{B}_0-S_i,  \\
        q^{-\langle\tau,\lambda\rangle} &\textrm{if }\tau\in S_i.
    \end{array}\right.
\end{align*}
Hence, for $\lambda\in X_*(T)$ with $\langle\tau,\lambda\rangle>N$ for all $\tau\in S_i$, one has
\begin{align*}
    \int_{T(\mathcal{O}_F)} h(\lambda(\varpi)u) \sigma(u)du &= \int_{T(\mathcal{O}_F)/\rho_i(\mathcal{O}_F^\times)}\left(\int_{\rho_i(\mathcal{O}_F^\times)} h(\lambda(\varpi)uv) \sigma(uv) dv\right)du\\
    &=\int_{T(\mathcal{O}_F)/\rho_i(\mathcal{O}_F^\times)} h(\lambda(\varpi)u)\sigma(u)\left(\int_{\rho_i(\mathcal{O}_F^\times)}\sigma(v) dv\right)du.
\end{align*}
This is zero if $\sigma\circ\rho_i$ is ramified. Assume $\sigma\circ\rho_i$ is unramified. Then by definition
\begin{align*}
    \int_{T(\mathcal{O}_F)} h(\lambda(\varpi)u) \sigma(u)du = \int_{T(\mathcal{O}_F)} \left(h_i(\lambda(\varpi)u)-h_i(\lambda(\varpi)\rho_i(\varpi)^{-1}u)\right)\sigma(u)du.
\end{align*}
One has 
\begin{align*}
    |\tau(\lambda(\varpi)u) - \tau(\lambda(\varpi)\rho_i(\varpi)^{-1}u)| = q^{-\langle\tau,\lambda-\rho_i\rangle} |\varpi^{\langle\tau,\rho_i\rangle} - 1| = \left\{\begin{array}{ll}
        0 &\text{if }\tau\in\mathcal{B}_0-S_i,  \\
        q^{-\langle\tau,\lambda-\rho_i\rangle} &\text{if }\chi\in S_i.
    \end{array}\right.
\end{align*}
Hence if $\langle\tau,\lambda-\rho_i\rangle>N$ for all $\tau\in S_i$, one has $h_i(\lambda(\varpi)u) = h_i(\lambda(\varpi)\rho_i(\varpi)^{-1}u)$ for all $u\in T(\mathcal{O}_F)$, and the integral vanishes.

In conclusion, we obtain that for $\lambda\in X_\ast(T)$
\begin{align*}
    \int_{T(\mathcal{O}_F)} h(\lambda(\varpi)u) \sigma(u)du = 0 
\end{align*}
if there is some $1\le i\le k$ such that $\langle \tau,\lambda\rangle>N+\delta_{\mathrm{unr},i}\langle \tau,\rho_i\rangle$ for all $\tau\in S_i,$ where 
\begin{align*}
    \delta_{\mathrm{unr},i}:=\begin{cases}
        1 & \textrm{if } \sigma\circ \rho_i \textrm{ is unramified,}\\
        0 & \textrm{if } \sigma\circ \rho_i \textrm{ is ramified.}
    \end{cases}
\end{align*}
To finish the proof, by \hyperref[Sas:f1alpha:polynomial]{Corollary \ref{Sas:f1alpha:polynomial}} it remains to show that
\begin{align*}
    \left(\bigcap_{i\in[k]} \{\lambda\in X_*(T): \langle \tau,\lambda\rangle \le  N+\delta_{\mathrm{unr},i}\langle \tau,\rho_i\rangle \text{ for some }\tau\in S_i\}\right) \cap (\alpha + \sigma_{\rho,T})
\end{align*}
is finite for all $\alpha\in X_*(T)$. It suffices to notice for each $N\geq 1$ and $(\tau_1,\ldots,\tau_k)\in S_1\times\cdots\times S_k$ that the set 
\begin{align*}
    \left\{\lambda\in \alpha+\sigma_{\rho,T}: \langle\tau_i,\lambda\rangle \leq N\text{ for }1\leq i\leq k \right\}
\end{align*}
is finite. 
\end{proof}



\section{On $L$-monoids}\label{sec:Lmonoids}

Now we discuss the underlying geometry of $\mathcal{S}_\rho(G(F))$ for general $G$. As explained in \cite[\S5.1]{Ngo:Hankel}, we can associate to each $\rho$ a normal reductive monoid with unit group $G$ following the construction of Renner in \cite[\S 5.1]{Renner:AlgebraMonoids}. We remark that in  \cite[\S 5.1]{Renner:AlgebraMonoids} the classification of normal reductive monoids is proved over algebraically closed fields. However, the same arguments with minor modification work over separably closed fields. In the more general setting of spherical varieties, this is proved in \cite{Wedhorn:Spherical}. 

Let $A_0\le T\le M_0$ be a maximal torus of $G$. Let $\sigma_{\rho}\subseteq X_*(T_{F^\mathrm{sep}}) \cong X^*((T_{F^\mathrm{sep}})^\vee)$ denote the submonoid generated by the $(T_{F^{\mathrm{sep}}})^\vee$-weights of $\rho$. Then $\sigma_{\rho}$ is stable under actions of both $W(G_{F^{\mathrm{sep}}},T_{F^{\mathrm{sep}}})$ and $\Gal_F$, and $\mathbb{R}_{\geq 0}\sigma_{\rho}$ is a strongly convex rational polyhedral cone by \eqref{G}. By \cite[Theorem 5.4]{Renner:AlgebraMonoids}, there exists a unique normal reductive monoid $X_{\rho,F^{\mathrm{sep}}}$ with unit group $G_{F^{\mathrm{sep}}},$ and such that the Zariski closure of $T_{F^{\mathrm{sep}}}$ in  $X_{\rho,F^{\mathrm{sep}}}$ is the affine normal toric variety $\overline{T}_{F^{\mathrm{sep}}}$ attached to the cone $\RR_{\ge 0}\sigma_{\rho}$. By the uniqueness and Galois descent, $X_{\rho,F^{\mathrm{sep}}}$ admits an $F$-rational structure $X_\rho:=X_{\rho,G}$ which is a normal monoid with unit group $G$. 


For later use we recall the explicit construction of $X_{\rho,F^{\mathrm{sep}}}$. Fix a Borel subgroup $T_{F^{\mathrm{sep}}}\le B\le G_{F^{\mathrm{sep}}}$. Let $\{\alpha_1,\ldots,\alpha_r\}\subset \sigma_{\rho}^\vee\cap X^*(T_{F^{\mathrm{sep}}})^+$ be a $\Gal_F$-invariant set of dominant weights such that its $W(G_{F^{\mathrm{sep}}},T_{F^{\mathrm{sep}}})$-orbit  generates $\sigma_{\rho}^\vee$. Let $V_{\alpha_i,F^{\mathrm{sep}}}$ be the irreducible $G_{F^{\mathrm{sep}}}$-representation of highest weight $\alpha_i$. Then we have an immersion
\begin{align}\label{eq:immersion}
    \omega:G_{F^{\mathrm{sep}}}\longrightarrow \bigoplus_{i=1}^r\mathrm{End}_{F^{\mathrm{sep}}}(V_{\alpha_i,F^{\mathrm{sep}}}).
\end{align}
The normalization of the Zariski closure of $\omega(G_{F^{\mathrm{sep}}})$ is (isomorphic to) $X_{\rho,F^{\mathrm{sep}}}$.

Let $\calB_0\subset X^\ast(M_0)$ be a finite $W(G,A_0)$-stable generating set of  $\overline{C}_{\rho,M_0}= (\mathbb{R}_{\geq0}\widetilde{\sigma}_{\rho,M_0})^\vee$. Recall that for $r\in \RR$
\begin{align*}
     S_0(r)=K\left\{ m\in M_0(F): |\chi|(m)\le q^r \textrm{ for all } \chi\in \calB_0\right\}K.
\end{align*}
\begin{lemma}\label{lem:bounded}
    For any $r\in\RR$, the set $S_0(r)$ is relatively compact in $X_\rho(F)$. 
\end{lemma}
\begin{proof}
Choose a finite separable extension $E/F$ such that $G_E$ is split. The Zariski closure $\overline{T}_E$ of $T_E$ in $X_{\rho,E}$ is isomorphic to the affine toric variety of $T_E$ attached to the cone $\RR_{\geq 0}\sigma_{\rho}$.

Since $X_{\rho}(F)\subseteq X_{\rho}(E)$ is closed, it suffices to prove $S_0(r)$ is relatively compact in $X_{\rho}(E)$. In other words, we must show
\begin{align*}
    S'_0(r):=\left\{ m\in T(F): |\chi|(m)\le q^r \textrm{ for all } \chi\in \calB_0\right\}
\end{align*}
is relatively compact in $\overline{T}_E(E)$. Consider the surjection
\begin{align*}
    \varphi:X^*(T_E)_\RR \longrightarrow X^*(A_{(M_0)_E})_\RR = X^*((M_{0})_E)_\RR \stackrel{\mathrm{av}}{\longrightarrow}X^*(M_0)_\RR
\end{align*}
where the last arrow $\mathrm{av}$ is defined as in the proof of \hyperref[toric:relcpt:1]{Lemma \ref{toric:relcpt:1}}. One has $\varphi(\sigma_\rho^\vee) = \mathrm{av}(\sigma_{\rho,M_0}^\vee) = \widetilde{\sigma}_{\rho,M_0}^\vee$. With this in mind, by the same argument as in \hyperref[toric:relcpt:1]{Lemma \ref{toric:relcpt:1}}, we see every function in $\sigma_{\rho}^\vee\cap X^*(T_E)$ is bounded on $S'_0(r)$.  This finishes the proof.
\end{proof}

\begin{corollary}\label{cor:compactsupport} Every function in $\mathcal{S}_\rho(G(F))$ has compact support in $X_\rho(F)$.
\end{corollary}
\begin{proof} Combine \hyperref[Schwartz=as+cs]{Theorem \ref{Schwartz=as+cs}} and \hyperref[lem:bounded]{Lemma \ref{lem:bounded}}.
\end{proof}

\begin{lemma}\label{lem:bounded:2} If $U\subseteq X_\rho(F)$ is relatively compact, then $U\cap G(F)\subseteq S_0(r)$ for some $r\in\RR$. 
\end{lemma}
\begin{proof} Retain the notation in the proof of \hyperref[lem:bounded]{Lemma \ref{lem:bounded}}. Let $\mathcal{B}'\subseteq X^*(T_E)$ be a finite generating set of the cone $\sigma_{\rho}^\vee\subseteq X^*(T_E)_\RR$ such that $\varphi(\mathcal{B}') \supseteq \mathcal{B}_0$. Choose a finite subset $J\subseteq M_0(F)$ such that $KT(F)JK=KM_0(F)K$. By \hyperref[toric:relcpt:2]{Lemma \ref{toric:relcpt:2}} and the Cartan decomposition, there is some $c>0$ such that $U\cap G(F)\subseteq K\{x\in T(E) : |\chi(x)|\leq c\text{ for all }\chi\in \mathcal{B}'\}JK$. This implies $U\cap G(F) \subseteq S_0(r)$ for some $r>0$.
\end{proof}

Let $\mathcal{S}_{\mathrm{ES}}(X_\rho(F)) := C_c^\infty(X_\rho(F))$. Note that if $X_\rho\hookrightarrow V$ is any $G\times G$-equivariant closed $F$-embedding into a finite-dimensional vector space $V$ over $F$, then $\mathcal{S}_{\mathrm{ES}}(X_\rho(F))=\mathcal{S}(V(F))|_{X_\rho(F)}.$ We will often view $f\in \mathcal{S}_{\mathrm{ES}}(X_\rho(F))$ as a function on $G(F)$ by restriction.

\begin{lemma}\label{SESinSinfty}
    We have $\mathcal{S}_{\mathrm{ES}}(X_\rho(F))\subseteq \mathcal{S}^\infty(G(F)).$ Furthermore,  $|\nu|^r\mathcal{S}_{\mathrm{ES}}(X_\rho(F))\subseteq  \mathcal{C}(G(F))$ for any $r>0$ sufficiently large.
\end{lemma}

\begin{proof}
   For the first assertion, by \hyperref[Cfpositive]{Proposition \ref{Cfpositive}} and \hyperref[lem:bounded]{Lemma \ref{lem:bounded:2}}, it suffices to show
   \begin{align*}
       \int_{S_0(r)} \Xi_\chi(g)dg
   \end{align*}
   converges absolutely for $\chi$ in a positive cone in $\Re \Lambda_{M_0}$. The proof is similar to that of \hyperref[supporttrunc]{Lemma \ref{supporttrunc}}. We leave it to the reader.

   For the second statement, by the definition of log norm, there is a finite subset $S\subset X^\ast(M_0)$ such that $S=-S$ and for any $d\ge 0,$ $\sigma(m)^d\ll_d \max_{\chi \in S} |\chi(m)|.$ Therefore, by \cite[Lemme II.1.1]{Waldspurger:Plancherel} for $m\in \Omega^+$
   \begin{align*}
       \Xi(m)\sigma(m)^{-d}|\nu|^{-r}(m)\gg_d \min_{\chi\in S} (|\nu^{-r}\chi|\delta_{P_0}^{1/2})(m).
   \end{align*}
   Given $f\in \mathcal{S}_{\mathrm{ES}}(X_\rho(F)),$ by \hyperref[lem:bounded]{Lemma \ref{lem:bounded:2}} there exist $c_1,c_2>0$ such that $|f|(g)\le c_1\mathbf{1}_{S_0(c_2)}(g)$ for all $g\in G(F).$ For $r>0$ large, we have $\left(|\nu^{-r}\chi|\delta_{P_0}^{1/2}\right)^{-1}\in C_{\rho,M_0}$ for all $\chi\in S.$ Thus for $m\in \Omega^+\cap S_0(c_2)$ and all $\chi\in S,$ $(|\nu^{-r}\chi|\delta_{P_0}^{1/2})(m)\gg_{c_2} 1.$ Hence for $g\in \mathrm{supp}(f),$ if $g\in KmK$ for some $m\in \Omega^+$ then
   \begin{align*}
      |f(g)|\le c_1\ll_{c_1,c_2} \min_{\chi\in S} (|\nu^{-r}\chi|\delta_{P_0}^{1/2})(m)\ll_d \Xi(m)\sigma(m)^{-d}|\nu|^{-r}(m)=\Xi(g)\sigma(g)^{-d}|\nu|^{-r}(g)
   \end{align*}
   for any $d\ge 0$ and $r>0$ large. This completes the proof.
\end{proof}

We expect \hyperref[prop:ESinrho]{Proposition \ref{prop:ESinrho}} to be true in general.

\begin{conjecture}  \label{conj:local}
The space $\mathcal{S}_{\mathrm{ES}}(X_\rho(F))$ is contained in $\mathcal{S}_\rho(G(F))$, up to a unique central twist.
\end{conjecture}

\noindent 
If \hyperref[conj:local]{Conjecture \ref{conj:local}} holds, we can write down a quite general Schwartz function in $\mathcal{S}_\rho(G(F))$ easily. In addition, it provides insight into the $G(F)\times G(F)$-module structure of $\mathcal{S}_\rho(G(F)).$ When $G=\GL_n$ and $\rho=\mathrm{std},$ the conjecture is well-known. One has $\mathcal{S}(M_{n}(F))\subseteq|\det|^{-\frac{n-1}{2}}\mathcal{S}_\rho(\GL_n(F))$\footnote{The containment is expected to be an equality.}. If the center of $G$ is one-dimensional and $\rho$ is irreducible with highest weight $\lambda_\rho,$ the central shift is relevant to $|\nu|^{-\langle \eta_G,\lambda_\rho\rangle}$ (c.f. \cite[\S 4.2]{Ngo:Hankel}). We will verify this expectation in several cases when $G=\GL_2$ and $\SL_2\times \GG_m$ in \S \ref{ssec:example}. As a preparation, we prove the following.

\begin{lemma}\label{lem:cuspidalhol}
    Suppose $G/Z_G$ is isotropic and $\RR_{\ge 0}\widetilde{\sigma}_{\rho,G}$ is a ray. Let $f\in \mathcal{S}_{\mathrm{ES}}(X_\rho(F)).$ For any $\pi\in \Pi_0(G)$ and $(v,w)\in\pi\times \pi^\lor,$ $\chi\mapsto Z(\pi_\chi,f,v,w)\in \CC[\Lambda_G].$ 
\end{lemma}

\begin{proof}
    We identify $m\in X_\ast(Z_G)$ with its image $m(\varpi)$ in $G(F)$. Write
    \begin{align*}
        G(F)=\bigsqcup_{j\in J}\bigsqcup_{m\in X_\ast(Z_G)} mG(F)^1j
    \end{align*}
    for some finite subset $J\subset G(F)$. Let $\chi_\pi$ be the (unitary) central character of $\pi$. Since $f\in \mathcal{S}_{\mathrm{ES}}(X_\rho(F)),$ by \hyperref[lem:bounded:2]{Lemma \ref{lem:bounded:2}} there exists a finite subset $S\subset X_*(Z_G)$ such that 
    \begin{align}\label{cuspidalhol:expansion}
        Z(\pi_\chi,f,v,w)=\sum_{m\in S+\widetilde{\sigma}_{\rho,G}\cap X_\ast(Z_G)} \sum_{j\in J}q^{-\langle \chi,m\rangle}\chi_\pi(m)\chi(j)\int_{G(F)^1} f(mgj)\langle \pi(g)\pi(j)v,w\rangle dg.
    \end{align}
    We show that the sum over $m$ is finite so the lemma follows.

    Let $U$ be the union over $j$ of the support of $g\mapsto \langle \pi(g)\pi(j)v,w\rangle$ in $G(F)^1.$ The set $U$ is compact. Let $E/F$ be a finite Galois extension such that $G_E$ is split, so the morphism \eqref{eq:immersion} is defined over $E$. For $1\leq i\leq r$, let $C_i$ be a finite set of matrix coefficients of $V_{\alpha_i,E}$ that forms an $E$-basis of $\mathrm{End}_E(V_{\alpha_i,E})$. Since $f\in\mathcal{S}_{\mathrm{ES}}(X_\rho(F))$, there is $n\geq 0$ such that $f(x)=f(y)$ whenever $|\phi(x) - \phi(y)|_E \leq q^{-n}$ for all $\phi\in \bigcup\limits_{1\leq i\leq r}C_i$.
    
    For $\phi\in C_i$, $m\in S+\widetilde{\sigma}_{\rho,G}\cap X_\ast(Z_G)$ and $m_0\in \widetilde{\sigma}_{\rho,G}\cap X_\ast(Z_G)$, one has 
    \begin{align*}
        |\phi(m_0gj) - \phi(mgj)|_E = q^{-[E:F]\langle\alpha_i,m_0\rangle}|\phi(gj)|_E |1-\varpi_F^{\langle\alpha_i,m-m_0\rangle}|_E.
    \end{align*}
    Since $U$ is compact, $S$ is finite, and $\mathbb{R}_{\geq 0}\widetilde{\sigma}_{\rho,G}$ is a ray, we can find $m_0\in\widetilde{\sigma}_{\rho,G}\cap X_\ast(Z_G)$ such that
    \begin{align*}
        |\phi(mgj) - \phi(m_0gj)|_E \leq q^{-n}
    \end{align*}
    for all $g\in U$, $m\in S+m_0+\widetilde{\sigma}_{\rho,G}\cap X_\ast(Z_G)$ and $\phi\in \bigcup\limits_{1\leq i\leq r}C_i$. Then
    \begin{align*}
        \int_{G(F)^1} f(mgj)\langle \pi(g)\pi(j)v,w\rangle dg=f(m_0gj)\int_{G(F)^1} \langle \pi(g)\pi(j)v,w\rangle dg
    \end{align*}
    for all $m\in S+m_0+\widetilde{\sigma}_{\rho,G}\cap X_{\ast}(Z_G)$. As $G/Z_G$ is isotropic, $\pi$ is infinite-dimensional, and thus there are no $G(F)^1$-fixed vectors. This implies the integral above is zero for $m\in S+m_0+\widetilde{\sigma}_{\rho,G}\cap X_\ast(Z_G)$, so that \eqref{cuspidalhol:expansion} is a finite sum because $\mathbb{R}_{\geq 0}\widetilde{\sigma}_{\rho,G}$ is a ray.
\end{proof}

\subsection{On the basic function}
Assume $G$ is unramified in this subsection. For future global application, we establish analytic properties of $b_\rho$. Let $\mathcal{G}$ be a smooth model of $G$ such that $\mathcal{G}(\mathcal{O}_F)=K$.  Let $E/F$ be a finite unramified extension such that $G_E$ is split. By \cite[I.10.4]{Jantzen:Reps}, we can choose an $\mathcal{O}_E$-lattice $V_{\alpha_i,\mathcal{O}_{E}}$ of $V_{\alpha_i,E}$ that is stable under $\mathcal{G}_{\mathcal{O}_E}$. Furthermore, we choose the lattices in a way so that $\{V_{\alpha_i,\mathcal{O}_{E}}: i=1,\ldots,r\}$ is $\Gal(E/F)$-stable. Then the morphism \eqref{eq:immersion} can be defined over $\mathcal{O}_E$ 
\begin{align*}
    \omega:\mathcal{G}_{\mathcal{O}_E}\longrightarrow \bigoplus_{i=1}^r\mathrm{End}_{\mathcal{O}_E}(V_{\alpha_i,\mathcal{O}_E}).
\end{align*}
Let $\mathcal{X}_{\rho,\mathcal{O}_E}\longrightarrow\overline{\omega(\mathcal{G}_{\mathcal{O}_E})}$ be its normalization. It is a finite morphism since the schematic closure $\overline{\omega(\mathcal{G}_{\mathcal{O}_E})}$ is an excellent integral scheme \cite[\href{https://stacks.math.columbia.edu/tag/07QW}{Tag 07QW}]{stacks-project}. By faithfully flat descent, $\mathcal{X}_{\rho,\mathcal{O}_E}$ admits an $\mathcal{O}_F$-rational structure $\mathcal{X}_\rho,$ which is an integral affine normal scheme of finite type.  Since $\mathcal{O}_F$ is a discrete valuation ring and $\mathcal{X}_\rho$ is integral and hence torsion-free, $\mathcal{X}_\rho$ is flat over $\mathcal{O}_F$. Thus $\mathcal{X}_\rho$ is a model of $X_\rho$ by its very definition. We let $X_\rho(\mathcal{O}_F):=\mathcal{X}_\rho(\mathcal{O}_F).$ 

Since $G$ is unramified, $G$ is quasi-split and $T=M_0$ is a maximal torus that is unramified. Therefore, the Zariski closure $\overline{T}$ of $T$ in $X_\rho$ is the affine normal toric variety attached to the cone $\RR_{\ge 0}\sigma_\rho$ defined in \S \ref{subsec:affinetoric}. Recall the $\mathcal{O}_F$-models $\mathcal{T}$ and $\overline{\mathcal{T}}$ of $T$ and $\overline{T},$ respectively.

\begin{lemma}\label{lem:modelmap}
 Up to automorphisms of $\mathcal{T},$ we have a canonical closed immersion of group schemes $\mathcal{T}\longrightarrow \mathcal{G}.$ It extends to a closed immersion of monoids $\overline{\mathcal{T}}\longrightarrow\mathcal{X}_\rho$.
\end{lemma}

\begin{proof}
By \cite[Proposition A.8.2]{KP:BruhatTits} $\mathcal{G}$ contains a closed $\mathcal{O}_F$-torus $\mathcal{T}'$ that is a smooth model of $T$. It follows from \cite[Corollary B.3.5]{Conrad:reductivegroup} that $\mathcal{T}$ and $\mathcal{T}'$ are canonically isomorphic up to automorphisms of $\mathcal{T}.$ 
To prove $\mathcal{T}\longrightarrow \mathcal{G}$ extends, by faithfully flat descent we may assume $G$ is split, so we take $E=F$. The schematic closure of $\mathcal{T}$ in $\bigoplus_{i=1}^r\mathrm{End}_{\mathcal{O}_F}(V_{\alpha_i,\mathcal{O}_F})$ is isomorphic to $\overline{\mathcal{T}}$ by its very definition. Let $Z$ be the closure of $\mathcal{T}$ in $\mathcal{X}_\rho$ equipped with the reduced structure. Then we have a finite birational morphism of integral schemes  $Z\longrightarrow\overline{\mathcal{T}}$. It is an isomorphism by  \cite[\href{https://stacks.math.columbia.edu/tag/0AB1}{Tag 0AB1}]{stacks-project}.

\end{proof}

\begin{lemma}[Satake]\label{Satake} Suppose $G$ is unramified. The map $(-)^{(P_0)}$ restricts to an isomorphism
\begin{align*}
    \mathrm{Sat}:\mathcal{C}(G(F)/\!/K)\longrightarrow \mathcal{C}(M_0(F)/\!/K_{M_0})^{W(G,A_0)}
\end{align*}
extending the usual Satake isomorphism (c.f. \cite[Theorem 7.5.3]{GH:book})
\begin{align*}
    \mathrm{Sat}:C_c^\infty(G(F)/\!/K)\longrightarrow C_c^\infty(M_0(F)/\!/K_{M_0})^{W(G,A_0)}.
\end{align*}
\end{lemma}

\begin{proof} The diagram in \eqref{CT:spectral} induces a commutative diagram
\begin{center}
    \begin{tikzcd}
        \mathcal{C}(G(F)/\!/K)\arrow[d,swap,"(-)^{(P_0)}"]\arrow[rr,"\HP_G"]& &\mathcal{C}(\Temp_\Ind(G))^{K\times K}\arrow[d,"(-)^{(P_0)}"]\\
        \mathcal{C}(M_0(F)/\!/K_{M_0})\arrow[rr,"\HP_{M_0}"]& &\mathcal{C}(\Temp_\Ind(M_0))^{K_{M_0}\times K_{M_0}}.
    \end{tikzcd}
\end{center}
Let $\Im\Lambda_{M_0}$ denote the connected component in $\Temp_\Ind(G)$ containing $(M_0,1)$. Since the trivial representation is Weyl-invariant, by \cite[Lemma 4.11]{DRS} restriction to $\Im\Lambda_{M_0}$ gives an isomorphism 
\begin{align}\label{CTemp:Kbiinv}
    \mathcal{C}(\Temp_{\Ind}(G))^{K\times K} \cong C^\infty(\Im\Lambda_{M_0})^{W(G,A_0)}.
\end{align}
Since $M_0$ is a torus, one has $\mathcal{C}(\Temp_\Ind(M_0))^{K_{M_0}\times K_{M_0}} = C^\infty(\Im\Lambda_{M_0})$. The lemma follows from \eqref{CTemp:Kbiinv} and the density of $C_c^\infty(G(F)/\!/K)$ in $\mathcal{C}(G(F)/\!/K)$.
\end{proof}

\begin{lemma}
    One has $\supp b_\rho \subseteq K\{g\in M_0(F): \chi(g)\leq 1 \text{ for all }\chi\in \overline{C}_{\rho,M_0}\}K$ and $b_\rho(\mathrm{id})=1.$ In particular, $\supp b_\rho\subseteq X_\rho(\mathcal{O}_F).$
\end{lemma}

\begin{proof}
 By \hyperref[toric:basic]{Lemma \ref{toric:basic}} $\mathrm{Sat}(b_\rho) = b_{\rho_{M_0}}$ is a nonnegative function with $b_{\rho_{M_0}}(\mathrm{id})=1.$ Since for $\chi\in C_{\rho,M_0}$
 \begin{align*}
    \int_{M_0(F)}b_{\rho_{M_0}}(m) |\nu|^{-1/2}\chi(m)dm=L(0,\chi,\rho_{M_0}),
 \end{align*}
we have
\begin{align*}
    \supp\, b_{\rho_{M_0}}\subseteq \{ m\in M_0(F): \chi(m)\le 1 \textrm{ for all } \chi\in \overline{C}_{\rho,M_0} \}=\overline{T}(O_F).
\end{align*}

By \hyperref[Schwartz=as+cs]{Theorem \ref{Schwartz=as+cs}} there exists $r\in \RR$ such that $\supp b_\rho \subseteq S_0(r).$ We need to show we can take $r=0$. Recall $\Omega=X_\ast(A_0)$ as $G$ is unramified. Let $\ge$ denote the usual partial order on $\Omega=X_*(A_0)$ given by specifying $y\ge x $ if and only if $y-x = \sum_{\gamma\in \Delta}c_\gamma\gamma^\vee$ with $c_\gamma\in \mathbb{Z}_{\geq 0}.$ Suppose on the contrary that $\supp b_\rho \not\subseteq S_0(0).$ Choose $m_0\in S_0(r)-S_0(0).$ Let $X^\ast(G)_\RR^\perp:=\{\lambda\in \Omega: \langle \chi,\lambda\rangle = 0\text{ for all }\chi\in X^*(G)_\RR\}.$ As $b_\rho\in C^\infty_{\mathrm{ac}}(G(F)),$ the set 
\begin{align*}
    J:=\{m\in \Omega\cap \supp b_\rho: m_0 - m\in X^\ast(G)_\RR^\perp\}
\end{align*}
is finite. Since $\mathcal{B}_0$ is $W(G,A_0)$-invariant, the space $(S_0(r)-S_0(0))\cap \supp b_\rho\cap \Omega$ is $W(G,A_0)$-invariant. Therefore, we may choose $m_0\in \Omega^+$ to be maximal with respect to the partial order $\ge$.

Write $b_\rho = \sum_{\lambda\in \Omega^+} b_\rho(\lambda)\mathbf{1}_{K\lambda K}.$ Then
\begin{align*}
    b_{\rho_{M_0}}(m_0) = \mathrm{Sat}(b_\rho)(m_0) &= \sum_{\lambda\in \Omega^+} b_\rho(\lambda) \mathrm{Sat}(\mathbf{1}_{K\lambda K})(m_0) \\
    &= \sum_{\lambda\in \Omega^+} b_\rho(\lambda)\sum_{\mu\in \Omega^+} a_{\lambda}(\mu) \left(\sum_{s\in W(G,A_0)/\mathrm{Stab}(\mu)}\mathbf{1}_{s.\mu T(F)^1}\right)(m_0)
\end{align*}
for some constants $a_\lambda(\mu)\in \RR_{\ge 0}.$ By the discussion around \cite[(6.7)]{Satake}, one has $a_{\lambda}(\mu)\neq 0$ only if $\mu\leq\lambda$ and $\lambda-\mu\in X^*(G)^\perp_\RR$. Hence 
\begin{align*}
    b_{\rho_{M_0}}(m_0) &= \sum_{\lambda\in \Omega^+\cap J}b_\rho(\lambda)\sum_{\substack{\mu\in \Omega^+\cap (m_0+X^\ast(G)_\RR^\perp)\\ \mu\le \lambda}} a_{\lambda}(\mu) \left(\sum_{s\in W(G,A_0)/\mathrm{Stab}(\mu)}\mathbf{1}_{s.\mu T(F)^1}\right)(m_0).
\end{align*}
If $b_\rho(\lambda)\neq 0$ and $s.\mu=m_0,$ then $\lambda\ge \mu\ge s.\mu=m_0,$ so it must be $\lambda=\mu=m_0$ by the maximality of $m_0$. Therefore,
\begin{align*}
    b_{\rho_{M_0}}(m_0) = b_\rho(m_0)a_{m_0}(m_0).
\end{align*}
By loc. cit. one has $a_{m_0}(m_0)\neq 0$. This implies $b_{\rho_{M_0}}(m_0)\neq 0$, a contradiction. Hence $\supp b_\rho\subseteq S_0(0)$. The last assertion follows from \hyperref[lem:modelmap]{Lemma \ref{lem:modelmap}}. Since $\mathrm{Sat}(\mathbf{1}_{K})(\mathrm{id})=1$, the discussion above then implies $1 = b_{\rho_{M_0}}(\mathrm{id}) = b_\rho(\mathrm{id})$. 
\end{proof}

As a special case of the \hyperref[HCPlan]{Harish-Chandra Plancherel theorem}, we have the spherical inversion formula: For $f\in\mathcal{C}(G(F)/\!/K)$, one has
\begin{align*}
    f(g) = \frac{\gamma(G|M_0)}{\#W(G,A_0)}\int_{\Im\Lambda_{M_0}} \HP(f)(M_0,\chi)\Xi_{\chi^{-1}}(g) j(\chi)^{-1}d\chi
\end{align*}
and
\begin{align*}
    \HP(f)(M_0,\chi) = \int_{G(F)} f(g)\Xi_\chi(g)dg.
\end{align*}

Let $m\in G(F)$. Since
\begin{align*}
    \HP(\mathbf{1}_{KmK})(M_0,\chi) = \int_{G(F)} \mathbf{1}_{KmK}(g)\Xi_\chi(g) dg = \vol(KmK) \Xi_\chi(m),
\end{align*}
applying the spherical inversion formula to the function $\mathbf{1}_{KmK},$ we have
\begin{align*}
    1 = \mathbf{1}_{KmK}(m) = \frac{\gamma(G|M_0)}{\# W(G,A_0)}\int_{\Im\Lambda_{M_0}} \vol(KmK) \Xi_\chi(m)\Xi_{\chi^{-1}}(m)j(\chi)^{-1} d\chi,
\end{align*}
so 
\begin{align}
\begin{split}\label{spherical:KmK}
    \vol(KmK)^{-1} &= \frac{\gamma(G|M_0)}{\# W(G,A_0)}\int_{\Im\Lambda_{M_0}} \Xi_\chi(m)\Xi_{\chi^{-1}}(m)j(\chi)^{-1} d\chi\\
    &=\frac{\gamma(G|M_0)}{\# W(G,A_0)}\int_{\Im\Lambda_{M_0}} |\Xi_\chi(m)|^2j(\chi)^{-1} d\chi.
\end{split}
\end{align}

Let $\{\widetilde{\mu}_j\}$ be the image of the central characters of irreducible subrepresentations of $\rho$ in $\Hom_\ZZ(X^\ast(G),\ZZ)$ as considered in \S\ref{ssec:LLC}. Let $\zeta_F(s)$ be the usual local zeta function attached to $F$.

\begin{proposition} Let $\tau\in |\nu|^{-1/2}C_{\rho,G}$ and $\epsilon=\mathrm{min}_j \langle \tau, \widetilde{\mu}_j\rangle>-1/2.$ Then  
\begin{align*}
    \tau(m)b_\rho(m)\le  \zeta_F(1/2+\epsilon)^{\dim \rho}\delta_{P_0}^{1/2}(m)
\end{align*}
for all $m\in \Omega^+\cap \overline{C}_{\rho,M_0}^\lor.$
\end{proposition}

\begin{proof} 
Let $m\in \Omega^+\cap \overline{C}_{\rho,M_0}^\lor.$ By  the spherical inversion formula, we have
\begin{align*}
    b_\rho(m) = \frac{\gamma(G|M_0)}{\# W(G,A_0)}\int_{\Im\Lambda_{M_0}} L(1/2,\chi,\rho_{M_0}) \Xi_{\chi^{-1}}(m)j(\chi)^{-1} d\chi.
\end{align*}
Let $\tau\in |\nu|^{-1/2}C_{\rho,G}$ and $\epsilon=\mathrm{min}_j \langle \tau, \widetilde{\mu}_j\rangle>-1/2.$ By shifting contour of $\chi$ to $\Re(\chi)=\tau$, we have
\begin{align*}
    \tau(m)b_\rho(m) = \frac{\gamma(G|M_0)}{\# W(G,A_0)}\int_{\Im\Lambda_{M_0}} L(1/2+\langle \tau,\rho\rangle,\chi,\rho_{M_0}) \Xi_{\chi^{-1}}(m)j(\chi)^{-1} d\chi.
\end{align*}
By the Cauchy-Schwarz inequality, its absolute value is bounded by the product of 
\begin{align}\label{eq:unramified1}
     \left(\frac{\gamma(G|M_0)}{\# W(G,A_0)}\int_{\Im\Lambda_{M_0}} |L(1/2+\langle \tau,\rho\rangle,\chi,\rho_{M_0}) |^2j(\chi)^{-1} d\chi\right)^{1/2}
\end{align}
and 
\begin{align}\label{eq:unramified2}
     \left(\frac{\gamma(G|M_0)}{\# W(G,A_0)}\int_{\Im\Lambda_{M_0}} |\Xi_{\chi}(m)|^2j(\chi)^{-1} d\chi\right)^{1/2}.
\end{align}

By \eqref{spherical:KmK}, the quantity \eqref{eq:unramified2} is $\mathrm{vol}(KmK)^{-1/2}$. We claim $\mathrm{vol}(KmK)^{-1/2}\le \delta_{P_0}^{1/2}(m).$  When $G$ is split and simply connected, this is a consequence of \cite[Proposition 3.2.15]{Macdonald:spherical} (see also \cite{Gross:Satake}). The statements in loc. ~cit. can be generalized when $G$ is unramified and $K$ is hyperspecial \cite[\S 1.6]{CasselmanCelyHales}. On the other hand, the quantity \eqref{eq:unramified1} is bounded by
\begin{align*}
   & \sup_{\chi\in \Im\Lambda_{M_0}} |L(1/2+\langle \tau,\rho\rangle,\chi,\rho_{M_0})|\left(\frac{\gamma(G|M_0)}{\# W(G,A_0)}\int_{\Im\Lambda_{M_0}} j(\chi)^{-1} d\chi\right)^{1/2}.
\end{align*}
Apply \eqref{spherical:KmK} again with $m=\mathrm{id}_{G(F)}\in K$. Using the fact that $\Xi_\chi(\mathrm{id}) =1$, one has
\begin{align*}
    \frac{\gamma(G|M_0)}{\# W(G,A_0)}\int_{\Im\Lambda_{M_0}} j(\chi)^{-1} d\chi=\vol(K) = 1.
\end{align*}
It remains to show $|L(1/2+\langle \tau,\rho\rangle,\chi,\rho_{M_0})|\le \zeta_F(1/2+\epsilon)^{\dim \rho},$ which is justified by the lemma below.
\end{proof}

\begin{lemma} Let $T$ be an unramified torus over $F$ and $\rho:{}^LT\to \GL_{V_\rho}(\CC)$ be a tempered representation. For $\epsilon>0$ and any $\chi\in\Im\Lambda_T,$
\begin{align*}
    |L(\epsilon,\chi,\rho)| \leq \zeta_F(\epsilon)^{\dim \rho}.
\end{align*}

\end{lemma}
\begin{proof} Let $D^\rho$ denote the torus over $F$ as constructed in \S\ref{subsec:LN}, and $\rho_T:D^\rho\to T$ denote the canonical map. Since $T$ is unramified, we have $D^\rho \cong \prod\limits_{i=1}^k\mathrm{Res}_{E_i/F}\mathbb{G}_{mE_i}$ for some finite unramified extensions $E_i$ of $F$. Let $\pi_i$ be morphisms defined in  \eqref{eq:pii}. By \hyperref[etaleTate:local:factr:equal]{Lemma \ref{etaleTate:local:factr:equal}}, one has 
\begin{align*}
    L(s,\chi,\rho) = \prod_{i=1}^k L(s,\chi\circ \pi_i)= \prod_{i=1}^k\frac{1}{1-(\chi\circ \pi_i)(\varpi_{E_i})q^{-[E_i:F]s}}.
\end{align*}
Since $\chi$ is unitary, for $\epsilon>0$ 
\begin{align*}
    |L(\epsilon,\chi\circ \pi_i)| \leq \frac{1}{1-q^{-[E_i:F]\epsilon}} \leq \frac{1}{1-q^{-\epsilon}}=\zeta_F(\epsilon).
\end{align*}
As $\dim\rho\geq k$ and $\zeta_F(\epsilon)\geq 1$, we conclude that
\begin{align*}
    |L(\epsilon,\chi,\rho)|\le \zeta_F(\epsilon)^k \le \zeta_F(\epsilon)^{\dim \rho}.
\end{align*}
\end{proof}

\subsection{Examples}\label{ssec:example} 

In this subsection we discuss examples related to symmetric powers of $\GL_2$ (c.f. \cite{BKNtriples}). 



\subsubsection{Odd symmetric powers} Suppose $G=\GL_2$ and let $\nu:=\det.$ Let $V_{\mathrm{std}}$ be the standard representation of $\GL_2$. Let $B> T$ be the group of upper triangular matrices and diagonal matrices respectively. Identify $X^\ast(T)=\ZZ^2$ so that $\begin{psmatrix}
    t_1\\
    & t_2
\end{psmatrix}\mapsto t_i$ corresponds to the standard basis $e_i$. Fix $n\ge 1.$ Let $\rho=\mathrm{Sym}^{2n+1} V_{\mathrm{std}}\otimes \det^{-n}$ be an irreducible representation of $\GL_2(\CC)=\GL_2^\vee(\CC).$  Its highest weight $\lambda_\rho$ is $(n+1,-n).$ The Hilbert basis of the cone $\overline{C}_{\rho,T}\cap \overline{X^\ast(T)}^+$ is $(1,1),  (n+1,n).$ Consider the map
\begin{align*}
    \omega:\GL_2&\longrightarrow \mathrm{End}(V_{\mathrm{std}}\otimes \det \nolimits^n)\oplus\mathrm{End}(\det)\\
    g &\mapsto \left(g\det g^{n},\det g\right).
\end{align*}
Then $\overline{\omega(\GL_2)}=\{(A,z)\in \mathrm{Mat}_{2}\times \GG_a: \det A=z^{2n+1}\}.$ It has a unique singularity at the origin. Since $\overline{\omega(\GL_2)}$ is a hypersurface in $\mathrm{Mat}_{2}\times \GG_a,$ by Serre's criterion for normality (c.f. \cite[\href{https://stacks.math.columbia.edu/tag/031S}{Tag 031S}]{stacks-project}), $\overline{\omega(\GL_2)}$ is normal and thus $X_\rho=\overline{\omega(\GL_2)}.$

Define 
\begin{align}\label{eq:centralshift}
    \chi_{\rho,G} : =|\nu|^{-(2n+1)/2}=|\nu|^{-\langle \eta_G,\lambda_\rho\rangle}.
\end{align}
One has 
\begin{align*}
     \chi_{\rho,G}\begin{psmatrix}
        t_1 & \\
        & t_2
    \end{psmatrix}:=|t_1t_2|^{-(2n+1)/2}.
\end{align*}

\begin{lemma}\label{lem:SESdescent}
    Let $f\in \mathcal{S}_{\mathrm{ES}}(X_\rho(F)).$ Then $f^{(B)}\in \chi_{\rho,G}\mathcal{S}_{\mathrm{ES}}(X_{\rho_T}(F)).$
\end{lemma}

\begin{proof}
By \hyperref[lemma:constant]{Lemma \ref{lemma:constant}} and \hyperref[SESinSinfty]{Lemma \ref{SESinSinfty}}, the integral defining $f^{(B)}$ is absolutely convergent. As $\mathcal{S}_{\mathrm{ES}}(X_{\rho_T}(F))$ is stable under multiplication by $C^\infty(X_{\rho_T}(F))$, to prove the lemma, it suffices to show that for $f\in \mathcal{S}(F)$, the smooth function on $T(F)$
\begin{align*}
   h\begin{psmatrix}
       t_1\\
       & t_2
   \end{psmatrix}:=\chi_{\rho,G}^{-1}\delta_{B}^{1/2}\begin{psmatrix}
       t_1& \\
        & t_2
   \end{psmatrix}\int_{F}f(t_1^{n+1}t_2^nu) du.
\end{align*}
extends to a smooth function on $X_{\rho_T}(F).$ By a change of variables $u\mapsto ut_1^{-(n+1)}t_2^{-n},$ we have
\begin{align*}
    h\begin{psmatrix}
       t_1\\
       & t_2
   \end{psmatrix}=\int_{F}f(u) du
\end{align*}
is a constant. The lemma follows.
\end{proof}

\begin{corollary}
    For any $\epsilon>0,$ $\chi_{\rho,G}^{-1}\mathcal{S}_{\mathrm{ES}}(X_\rho(F))\subset L^1(G(F),|\det g|^{\epsilon}dg).$
\end{corollary}

\begin{proof}
    By the Iwasawa decomposition and the computation above, it suffices to show for $f\in \mathcal{S}(F^3)$
\begin{align*}
    \int_{(F^\times)^2} f(t_1^{n+1}t_2^n,t_1^n t_2^{n+1},t_1t_2)|t_1t_2|^{\epsilon}d^\times t_1d^\times t_2<\infty.
\end{align*}
By changing variables $t_1\mapsto t_1t_2^{-1}$ and  $t_2\mapsto t_1^{-n}t_2$ successively, the integral is
\begin{align*}
    \int_{(F^\times)^2} f(t_1^{2n+1}t_2^{-1},t_2, t_1) |t_1|^{\epsilon} d^\times t_1d^\times t_2,
\end{align*}
which is clearly finite.
\end{proof}

\begin{proposition}\label{prop:contain}
    We have $\mathcal{S}_{\mathrm{ES}}(X_\rho(F))\subseteq \chi_{\rho,G}\mathcal{S}_\rho(G(F)).$ 
\end{proposition}

\begin{proof}
    Let $f\in \mathcal{S}_{\mathrm{ES}}(X_\rho(F)).$ View $f$ as a function on $G(F)$ by restriction. By \hyperref[SESinSinfty]{Lemma \ref{SESinSinfty}}, we only need to check $\chi_{\rho,G}^{-1}f$ satisfies \eqref{S2}. For $\sigma\in \Pi_2(T),$ by \hyperref[constant:zeta:identity]{Lemma \ref{constant:zeta:identity}}, \hyperref[prop:ESinrho]{Proposition \ref{prop:ESinrho}} and \hyperref[lem:SESdescent]{Lemma  \ref{lem:SESdescent}}, we have for $(v,w)\in \Ind_{K_B}^K \sigma|_{K_T}\times \Ind_{K_B}^K \sigma^\lor |_{K_T}$
    \begin{align}\label{eq:inductionstep}
    \begin{split}
        &Z(\mathrm{Ind}_P^G \sigma_\chi,\chi_{\rho,G}^{-1}f,v,w)\\
        &=\vol(K_0,dg)^2\sum_{k,k'\in K/K_0} Z(\sigma_\chi,\chi_{\rho,G}^{-1}(R(k,k')f)^{(B)},v(k'),w(k))\in L(0,\sigma_\chi,\rho_T)\CC[\Lambda_T].
    \end{split}
    \end{align}
 
Let $\pi\in \Pi_2(G)$ and $(v,w)\in \pi\times\pi^\lor.$ When $\pi\in \Pi_0(G),$ then $Z(\pi_{|\nu|^s},\chi_{\rho,G}^{-1}f,v,w)\in \CC[q^{\pm s}]$ by \hyperref[lem:cuspidalhol]{Lemma \ref{lem:cuspidalhol}}. Thus assume $\pi\in \Pi_2(G)-\Pi_0(G),$ say $\pi< \Ind_B^G \sigma$ where $\sigma$ is a quasi-character of $T(F)$. By \eqref{eq:inductionstep}, we have $Z(\pi_{|\nu|^s},\chi_{\rho,G}^{-1}f,v,w)\in L(s,\sigma,\rho_{T})\CC[q^{\pm s}]$. Therefore, by \eqref{eq:centralshift} to prove the proposition it suffices to show
\begin{align}\label{eq:containment:entire}
    \frac{Z(\pi_{|\nu|^s},f,v,w)}{L(0,\pi_{|\nu|^s\chi_{\rho,G}},\rho)}= \frac{Z(\pi_{|\nu|^s},f,v,w)}{L(s-\tfrac{2n+1}{2},\pi,\rho)}
\end{align}
is entire.

Decompose $\rho_T=\oplus_i\rho_i$ into irreducible submodules. By the Bernstein-Zelevinsky classification (c.f. \cite[Theorem 8.4.3]{GH:book}), $\pi$ is a twisted Steinberg representation and $\sigma=\eta_G\cdot (\sigma_0\circ\det)$ for some unitary character $\sigma_0$ of $F^\times$. We have 
\begin{align*}
    L(s-\tfrac{2n+1}{2},\sigma,\rho_T)=\prod_{i=0}^{2n+1} L(s-\tfrac{2n+1}{2}, \sigma_0|\cdot|^{\tfrac{2n+1-2i}{2}})=\prod_{i=0}^{2n+1}L(s-i,\sigma_0).
\end{align*}
Therefore, if $\sigma_0$ is ramified, then $Z(\pi_{|\nu|^s},f,v,w)\in \CC[q^{\pm s}].$ We hence assume $\sigma_0$ is trivial, i.e., $\pi$ has trivial central character. We claim $Z(\pi_{|\nu|^s}, f,v,w)\in \zeta_F(s)\CC[q^{\pm s}].$ Note that
\begin{align*}
     L(1+\tfrac{2n+1}{2}-s,\sigma^\lor,\rho_{T})=\prod_{i=0}^{2n+1} L(1-s+i,\sigma_0^\lor)=\prod_{i=0}^{2n+1}\zeta_F(-(s-i-1)).
\end{align*}
Thus by \eqref{LLC:gamma}
\begin{align*}
    \gamma(0,\pi_{|\nu|^s\chi_{\rho,G}},\rho)=\gamma(0,\mathrm{Ind}_B^G \sigma_{|\nu|^s\chi_{\rho,G}},\rho)=\gamma(0, \sigma_{|\nu|^s\chi_{\rho,G}},\rho_T)
\end{align*}
vanishes at $s=0,$ so $L(0,\pi_{|\nu|^s\chi_{\rho,G}},\rho)$ has a pole at $s=0$. The claim implies \eqref{eq:containment:entire} is entire.

To prove the claim, we may assume $f=f_1\otimes f_2\in \mathcal{S}(\mathrm{Mat}_2(F))\otimes \mathcal{S}(F)$. We need to show
\begin{align*}
    \int_{\GL_2(F)} f_1(g\det g^{n})f_2(\det g)\langle\pi(g)v ,w\rangle |\det g|^s dg\in \zeta_F(s)\CC[q^{\pm s}].
\end{align*}
Since $f_2$ is locally constant at the origin, it amounts to showing
\begin{align}\label{eq:GL2goal}
    \int_{|\det g|\le 1} f_1(g\det g^{n})\langle\pi(g)v ,w\rangle |\det g|^s dg\in \zeta_F(s)\CC[q^{\pm s}].
\end{align}
Recall the well-known results of Godement-Jacquet: for $f_1\in \mathcal{S}(\mathrm{Mat}_2(F))$ and characters $\chi:F^\times \to \CC^\times$
\begin{align}\label{eq:GJ}
    \int_{\GL_{2}(F)} f_1(g)\langle\pi(g)v,w\rangle \chi(\det g)|\det g|^s dg\in L(s,\chi)\CC[q^{\pm s}].
\end{align}
Since $\pi$ has trivial central character, one has for $j\ge 0$ 
\begin{align*}
     &\int_{|\det g|=q^{-j}} f_1(g\det g^{n})\langle\pi(g)v ,w\rangle |\det g|^s dg\\
     &=\int_{|\det g|=q^{-j}} f_1(g\det g^{n})\langle\pi(g\det g^n)v ,w\rangle |\det g|^s dg\\
     &=q^{-js}\int_{\mathcal{O}_F^\times}\int_{\SL_2(F)} f_1\left(g\begin{psmatrix}
         \varpi^{2nj+1}a^{2n+1}&\\
         & 1
     \end{psmatrix}\right)\langle\pi\left(g\begin{psmatrix}
         \varpi^{2nj+1}a^{2n+1}&\\
         & 1
     \end{psmatrix}\right)v ,w\rangle dgda.
\end{align*}
By $K$-finiteness, there is an open subgroup $U\le \mathcal{O}_F^\times$ such that for all $g\in \GL_2(F)$ and $u\in U$
\begin{align*}
    f_1\left(g\begin{psmatrix}
        u&\\
         & 1
     \end{psmatrix}\right)\langle\pi\left(g\begin{psmatrix}
         u&\\
         & 1\end{psmatrix}\right)v ,w\rangle =f_1\left(g\right)\langle\pi\left(g\right)v ,w\rangle.
\end{align*}
 
Consider the finite abelian group $H:=\mathcal{O}_F^\times/(U\cdot(\mathcal{O}_F^\times)^{2n+1}).$ Note that every $\chi\in \mathrm{Hom}_{\mathbf{Gp}}(H,\CC^\times )$ can be extended (non-canonically) to a unitary character of $F^\times$ by setting $\chi(\varpi)=1.$ By the Fourier inversion of finite abelian groups, there is an absolute constant $c>0$ such that 
the integral above is equal to
\begin{align*}
     &cq^{-js}\sum_{\chi\in \mathrm{Hom}_{\mathbf{Gp}}(H,\CC^\times )}\int_{|\det g|=q^{-(2n+1)j}} f_1(g)\langle\pi(g)v ,w\rangle \chi(\det g)dg.
\end{align*}
By \eqref{eq:GJ} there is a constant $c'\in \CC$ such that for any $\chi$ and $j$ sufficiently large 
\begin{align*}
    \int_{|\det g|=q^{-(2n+1)j}} f_1(g)\langle\pi(g)v ,w\rangle \chi(\det g)dg=c'\delta_{\chi=1}.
\end{align*}
We conclude that \eqref{eq:GL2goal} is an element in $\zeta_F(s)\CC[q^{\pm s}]$. This concludes the proof. 
\end{proof}

\subsubsection{Even symmetric powers}

Suppose $G=\SL_2\times \GG_m$ and let $\nu$ be the projection to $\GG_m$. Let $B> T$ be the group of upper triangular matrices and diagonal matrices in $\SL_2$ respectively. Let $\alpha$ be the unique positive root of $\SL_2$ with respect to $B$. We identify $X^\ast(T\times \GG_m)=\frac{1}{2}\ZZ\alpha\oplus  \ZZ\nu$. Fix $n\ge 1.$ Let $\rho:=\mathrm{Sym}^{2n} V_{\mathrm{std}}\otimes \nu$ be an irreducible representation of $G^\vee(\CC)=\mathrm{PGL}_2(\CC)\times \CC^\times.$ The Hilbert basis of the cone $\overline{C}_{\rho,T\times\GG_m}\cap \overline{X^\ast(T\times \GG_m)}^+$ is $(0,\nu), (\frac{1}{2}\alpha,n\nu).$ Consider the map
\begin{align*}
    \omega:G&\longrightarrow \mathrm{End}(V_{\mathrm{std}}\otimes \nu^{n})\oplus \mathrm{End(}\nu)\\
    (g,a) &\mapsto \left(a^{n}g,a\right).
\end{align*}
Then $\overline{\omega(G)}=\{(A,z)\in \mathrm{Mat}_2(F)\oplus F: \det A=z^{2n}\}.$ It has a unique singularity at the origin, so it is normal by Serre's normality criterion. Thus $X_\rho=\overline{\omega(G)}.$

Let $f\in \mathcal{S}_{\mathrm{ES}}(X_\rho(F)).$ By a change of variables as in \hyperref[lem:SESdescent]{Lemma \ref{lem:SESdescent}}, one has $f^{(B\times \GG_m)}\in \chi_{\rho,G}\mathcal{S}_{\mathrm{ES}}(X_{\rho_{T\times\GG_m}}(F)),$ where $\chi_{\rho,G}:=-n\nu$. Note that the highest weight of $\rho$ is $\lambda_\rho=(n\alpha^\lor,\nu^\lor).$ Therefore,
\begin{align*}
    \langle \eta_G,(n\alpha^\lor,\nu^\lor)\rangle=n=-\langle \chi_{\rho,G},(n\alpha^\lor,\nu^\lor)\rangle.
\end{align*}
By a similar argument as in \hyperref[prop:contain]{Proposition \ref{prop:contain}}, we have $\mathcal{S}_{\mathrm{ES}}(X_\rho(F))\subseteq \chi_{\rho,G}\mathcal{S}_\rho(G(F)).$ 


\bibliographystyle{alpha}
\bibliography{ref}
\end{document}